\documentclass[a4paper,11pt,DIV=10]{scrartcl}
\usepackage[T1]{fontenc}
\usepackage[backend=biber,style=numeric,maxbibnames=99,giveninits,doi=false,isbn=false,sorting=nyt]{biblatex}
\bibliography{refpc.bib}

\AtBeginBibliography{\small}

\usepackage{multicol}
\usepackage{amsmath}
\usepackage{amssymb}
\addtokomafont{disposition}{\rmfamily\normalfont\scshape}
\usepackage{mathtools}
\usepackage[osf,sc]{mathpazo}
\usepackage{relsize}
\usepackage{amsthm}
\usepackage{enumerate}
\usepackage[british]{babel}
\usepackage[autostyle=true]{csquotes}
\usepackage{verbatim}
\usepackage{nicefrac}
\usepackage{tikz}
\usepackage{tikz-cd}
\usepackage{microtype}
\usepackage{pdfpages}
\usepackage[hidelinks,unicode]{hyperref}
\usepackage[format=plain]{caption}
\recalctypearea

\usepackage{algorithm}
\usepackage{algpseudocode}

\usepackage{array}
\newcolumntype{H}{>{\setbox0=\hbox\bgroup}c<{\egroup}@{}}
\usepackage{hyperref}
\usepackage{relsize}
\usepackage{wrapfig}
\usepackage{amsmath}
\usepackage{amssymb}
\usepackage{mathtools}
\usepackage{amsthm}

\usepackage[capitalize,noabbrev]{cleveref}

\theoremstyle{plain}
\newtheorem{theorem}{Theorem}[section]

\theoremstyle{definition}
\newtheorem{definition}[theorem]{Definition}

\theoremstyle{remark}
\newtheorem{remark}[theorem]{Remark}

\usepackage[T1]{fontenc}
\usepackage{amsmath}
\usepackage{amssymb}
\usepackage{mathtools}
\usepackage{thmtools}
\usepackage{thm-restate}
\usepackage{enumerate}
\usepackage[autostyle=true]{csquotes}
\usepackage{verbatim}
\usepackage{tikz}
\usepackage{tikz-cd}
\usepackage{microtype}
\usepackage{pdfpages}
\usepackage{booktabs}
\usepackage{subcaption}
\usepackage{bm}
\usepackage{csquotes}
\usepackage{soul}
\usepackage{makecell}
\usepackage[boldmath]{numprint}
\npthousandsep{\,}
\renewcommand{\smash}[1]{#1}

\newtoggle{arxiv}
\toggletrue{arxiv}
\newcommand{\STATE}{\State}
\newcommand{\TCBS}{\textsc{tcbs}}
\newcommand{\ARI}{\textsc{ari}}

\newcommand{\Z}{\mathbb{Z}}

\newcommand{\SC}{\mathcal{S}}

\newcommand{\N}{\mathbb{N}}

\newcommand{\TDA}{\textsc{tda}}
\newcommand{\PH}{\textsc{ph}}
\newcommand{\VR}{\textsc{vr}}
\newcommand{\tSC}{\textsc{sc}}

\newcommand{\R}{\mathbb{R}}
\newcommand{\bound}{\mathcal{B}}
\newcommand{\betti}{B}

\newcommand{\feat}{\mathcal{F}}

\newcommand{\TOPF}{\textsc{topf}}

\newcommand{\ZthreeZ}{\mathbb{Z}/3\mathbb{Z}}
\newcommand{\TPCC}{\textsc{tpcc}}

\newcommand{\VAE}{\textsc{vae}}
\newcommand{\features}{\mathcal{F}}

\DeclareMathOperator{\Ima}{Im}
\DeclareMathOperator{\rk}{rk}

\DeclareMathOperator{\diag}{diag}
\DeclareMathOperator*{\argmin}{arg\,min}

\newcommand{\figuretitle}{\textsc}
\newcommand{\citep}[1]{\cite{#1}}
\newcommand{\citet}[1]{\cite{#1}}

\newcommand\michael[1]{\noindent{\textcolor{magenta}{[\textsc{mts}: #1]}}}
\newcommand\vincent[1]{\noindent{\textcolor{magenta}{[\textsc{vpg}: #1]}}}
\renewcommand\vincent[1]{}\renewcommand\michael[1]{}
\begin{document}
	\title{Point-Level Topological Representation Learning on Point Clouds}
	
	\author{%
		Vincent P.~Grande and  Michael T. Schaub\\
		\textsc{rwth} Aachen University, Germany
		%\texttt{grande@cs.rwth-aachen.de}
		%\And
		%	Michael T.~Schaub \\
		%Department of Computer Science\\
		%\textsc{RWTH} Aachen University, Germany\\
		%\texttt{schaub@cs.rwth-aachen.de}
	}
	\date{\vspace{-3ex}}
	\maketitle
	% !TeX spellcheck = en_GB
% !TeX root = TOPFICML.tex
% !TeX program = pdflatex
\begin{abstract}
	\noindent
	Topological Data Analysis (\TDA) allows us to extract powerful topological and higher-order information on the global shape of a data set or point cloud.
	Tools like Persistent Homology give a \emph{single} complex description of the \emph{global structure} of the point cloud.
	However, common machine learning applications like classification require \emph{point-level} information and features. % to be available.
	In this paper, we bridge this gap and propose a novel method to extract node-level topological features from complex point clouds using discrete variants of concepts from algebraic topology and differential geometry.
	We verify the effectiveness of these topological point features (\TOPF) on both synthetic and real-world data and study their robustness under noise and heterogeneous sampling.
\end{abstract}
\iftoggle{arxiv}{
\begin{figure}[ht!]
	\begin{center}
	\centerline{
		\includegraphics[width=\linewidth]{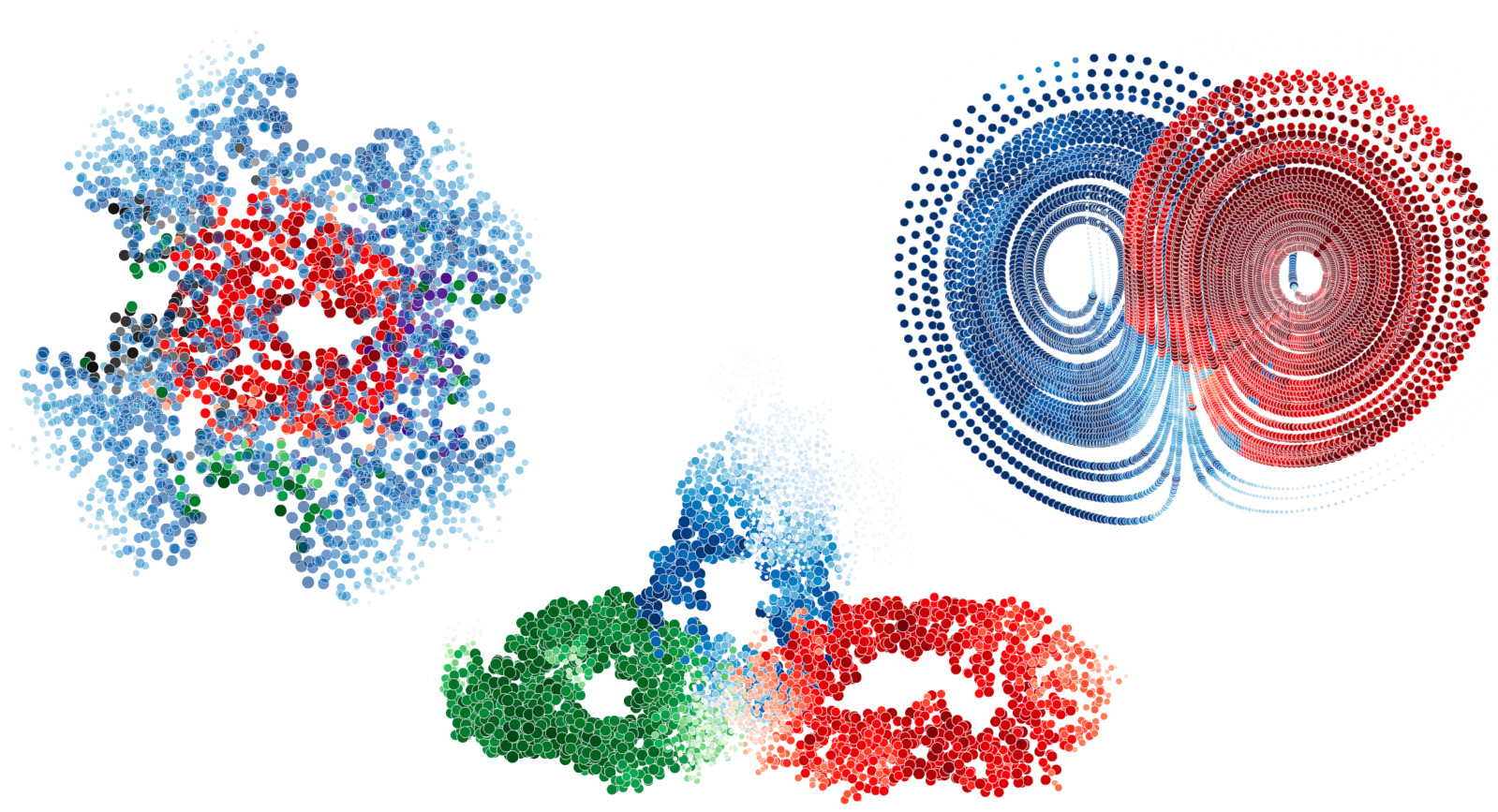}
	}\captionsetup{labelformat=empty}
	\caption{\figuretitle{
			Topological point features (\TOPF{}) on real-world and simulated 3d point clouds.}
			For every point, we highlight the largest corresponding topological feature, where colour stands for different features and saturation for the value of the feature. (Cf.\ \Cref{fig:QualitativeExperiments})
			}
			\addtocounter{figure}{-1}
			\vspace{-3cm}
\end{center}
\end{figure}
\newpage
}{}
\begin{figure*}[ht!]
	\begin{center}
		\includegraphics[width=\linewidth]{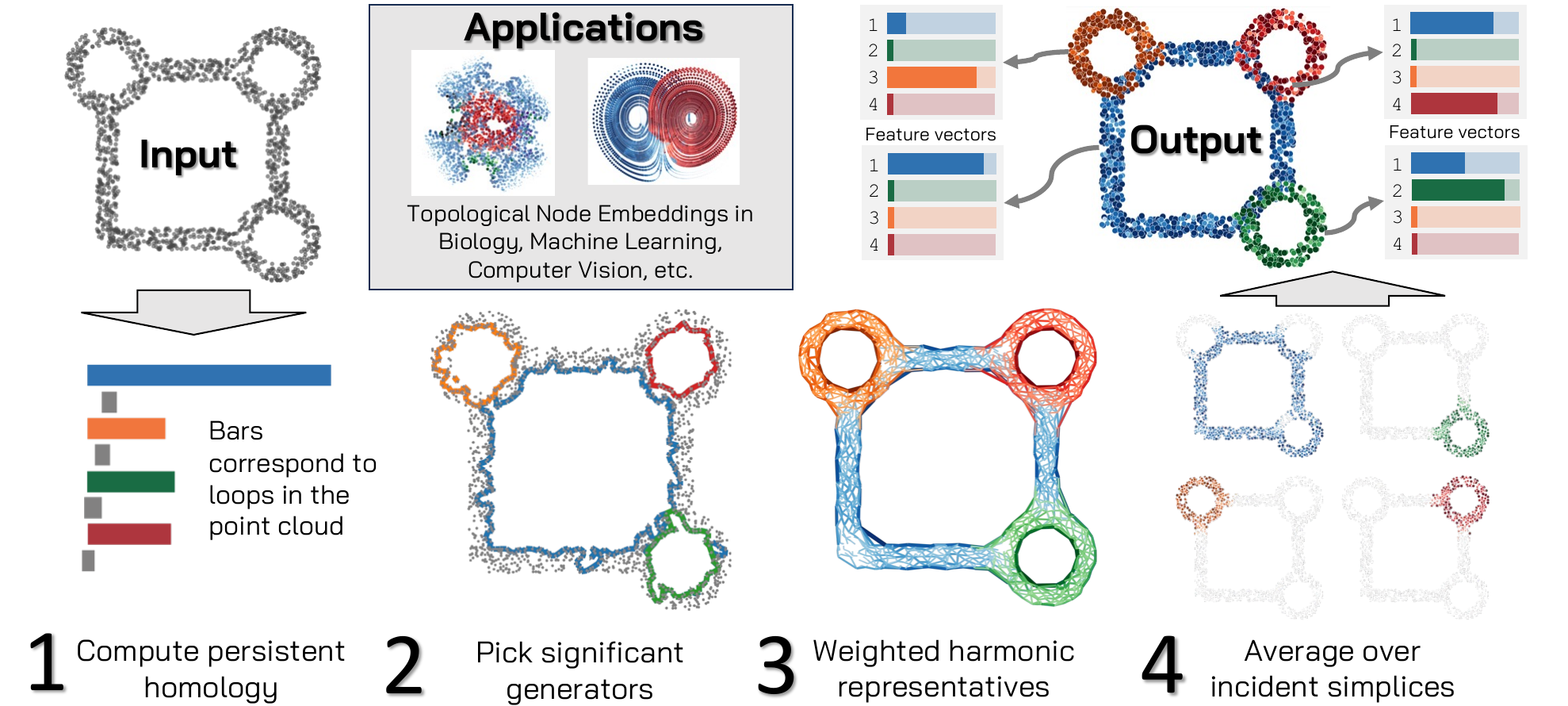}
\caption{\figuretitle{Computing Topological Point Features (\TOPF)}.
			\figuretitle{Input.} A point cloud $X$ in $n$-dimensional space.
			\figuretitle{Step 1.} To extract global topological information, the persistent homology is computed on an $\alpha$/\VR-filtration.
			The most significant topological features $\features$ across all specified dimensions are selected.
			\figuretitle{Step 2.} $k$-homology generators associated to all features $f_{i,k}\in\features$ are computed. For every feature, a simplicial complex is built at a step of the filtration where $f_{i,k}$ is alive.
			\figuretitle{Step 3.} The homology generators are projected to the harmonic space.
			\figuretitle{Step 4.}  The vectors are normalised to obtain vectors $\smash{\mathbf{e_k^i}}$ indexed over the $k$-simplices. For every point $x$ and feature $\smash{f\in\features}$, we compute the mean of the entries of $\mathbf{e_k^i}$ corresponding to simplices containing $x$.
			The output is a $|X|\times|\features|$ matrix which can be used for downstream \textsc{ml} tasks.
			\figuretitle{Optional.} We weigh the simplicial complexes resulting in a topologically more faithful harmonic representative in \figuretitle{Step 3}.}
		\label{fig:fig1}
	\end{center}
	\vspace{-0.5cm}
\end{figure*}
\section{Introduction}
In modern machine learning~\citep{pml1Book}, objects are described by feature vectors. % within a high-dimensional space.
However, the coordinates of a single vector can often only be understood in relation to the entire data set:
if the value $x$ is small, average, large, or even an outlier depends on the remaining data.
In a low-dimensional case this issue can be addressed simply by normalising the data points according to the global mean and standard deviation or similar procedures.
We can interpret this as the most straight-forward way to construct \emph{local} features informed by the \emph{global} structure of the data set.

In the case where not all data dimensions are equally relevant, or contain correlated and redundant information, we can apply (sparse) \textsc{pca} to project the data points to a lower-dimensional space using information about the \emph{global structure} of the point cloud. % \citep{zou2006sparse}.
For even more complex data, we may first have to learn the encoded structure itself:
indeed, a typical assumption underpinning many unsupervised learning methods is the so-called ``manifold hypothesis'' which posits that real world data can be described well via submanifolds of $n$-dimensional space \citep{ma2012manifold,Fefferman:2016}.
Using eigenvectors of some Laplacian, we can then obtain a coordinate system intrinsic to the point cloud (see e.g.~\cite{shi2000normalized, belkin2003laplacian,coifman2006diffusion}).
Common to all these above examples is the goal is to construct locally interpretable point-level features that encode \emph{globally meaningful positional information} robust to local perturbations of the data.

Instead of focussing on the interpretation of individual points, topological data analysis (\TDA), \cite{Carlsson:2021}, follows a different approach.
\TDA{} extracts a global description of the shape of data, which is typically considered in the form of a high-dimensional point cloud.
This is done measuring topological features like persistent homology, which counts the number of generalised ``holes'' on multiple scales.
Due to their flexibility and robustness these global topological features have been shown to contain relevant information in a broad range of application scenarios:
In medicine, \TDA{} has provided methods to analyse cancer progression \citep{lawson2019persistent}.
In biology, persistent homology has been used to analyse knotted protein structures \citep{benjamin2023homology}, and the spectrum of the Hodge Laplacian has been used for predicting protein behaviour \citep{wee2024integration}.

This success of topological data analysis is a testament to the fact that relevant information is encoded in the global topological structure of point cloud data.
Such higher-order topological information is however invisible to the tools of data analysis discussed above like \textsc{pca} or diffusion maps.
We are now faced by the problem that \textbf{(i)} important parts of the global structure of a complex point cloud can only be described by the language of applied topology, however \textbf{(ii)} most standard methods to obtain positional point-level information are not sensitive to the higher-order topology of the point cloud.

\paragraph{Contributions}
We introduce \TOPF{} (\Cref{fig:fig1}), a novel method to compute \textbf{point-level topological features} relating individual points to global topological structures of point clouds.
\TOPF{} \textbf{(i)} \emph{outperforms} other methods and embeddings for clustering downstream tasks on topologically structured data, returns \textbf{(ii)} \emph{provably meaningful representations}, and is \textbf{(iii)} \emph{robust to noise and heterogeneous sampling}. 
Finally, we introduce the topological clustering benchmark suite, the first benchmark for topological clustering.

\paragraph{Related Work}

The intersection of topological data analysis, topological signal processing and geometry processing has many interesting related developments in the past few years.
On the side of homology and \TDA{}, the authors in \cite{DeSilva2009persistent} and \cite{Perea2020} use harmonic \emph{co}homology representatives to reparametrise point clouds based on circular coordinates.
This implicitly assumes that the underlying structure of the point cloud is amenable to such a characterization.

In \cite{Basu2022harmonic,Gurnari2023probing}, the authors develop and use harmonic persistent homology for data analysis.
However, among other differences their focus is not on providing robust topological point features.
\cite{carriere2015stable} construct point features using Topology as well.
However, their signatures only summarise the local topology of the neighbourhood of the point rather than the relation between the point and the \emph{global} topological features.
\cite{Grande:2023} uses the harmonic space of the Hodge Laplacians to cluster point clouds respecting topology, but is unstable against some form of noise, has no possibility for features selection across scales and is computationally far more expensive than \TOPF.
An overview over the thriving field of topological data analysis can be found in \citep{wasserman2018topological, munch2017user}.
For a more in-depth review of related work, see \Cref{app:RelatedWork}.
Because there are different views on what constitutes \emph{representation learning}, we note that the learnt representations of \TOPF{} are the result of \emph{homology and linear algebra computations}, rather than trained features of an autoencoder or other neural network.

\paragraph{Organisation of the paper}

In \Cref{sec:Background}, we give an overview over the main ideas and concepts behind of \TOPF.
In \Cref{sec:algorithm}, we describe how to compute \TOPF.
Finally, we will apply \TOPF{} on synthetic and real-world data in \Cref{sec:experiments}.
Furthermore, \Cref{app:extendedbackground} contains a brief history of topology and a detailed discussion of related work.
\Cref{app:TheoreticalConsiderations} contains additional theoretical considerations.
In \Cref{sec:theoreticalguarantee}, we give a theoretical result guaranteeing the correctness of \TOPF.
 \Cref{app:BenchmarkSuite} describes the novel topological clustering benchmark suite, \Cref{app:implementation} contains details on the implementation and the choice of hyperparameters, \Cref{app:featureselection} gives a detailed treatment of feature selection, \Cref{app:WeightedSCs} discusses simplicial weights, and \Cref{app:limitations} discusses limitations in detail.

\paragraph{Code}

We provide \TOPF{} as an easy-to-use python package with example notebooks 
\iftoggle{arxiv}{at \url{https://github.com/vincent-grande/topf}.}
{in the supplementary material.}

\section{Main Ideas of \iftoggle{arxiv}{\TOPF}{TOPF}}
\label{sec:Background}

A main goal of algebraic topology is to capture the shape of spaces.
Techniques from topology describe globally meaningful structures that are indifferent to local perturbations and deformations.
This robustness of topological features to local perturbations is particularly useful for the analysis of large-scale noisy datasets.
%To apply the ideas of algebraic topology in our \TOPF{} pipeline, we need to formalise and explain the notion of \emph{topological features}.
%An important observation for this is that high-dimensional point clouds and data may be seen as being sampled from topological spaces\,---\,most of the time, even low-dimensional submanifolds of $\R^n$ \citep{Fefferman:2016}.

In this section we provide a broad overview over the most important concepts of topology and \TDA{} for our context, prioritising intuition over technical formalities.
For more details, the reader is referred to~\cite{Bredon:1993} (algebraic topology) and \cite{munch2017user} (\TDA).

\paragraph{Simplicial Complexes}

Spaces in topology are \emph{continuous} (\emph{connected}), consist of \emph{infinitely} many points, and often live in \emph{abstract space}.
Our input data sets however consist of \emph{finitely} many points embedded in \emph{real space} $\R^n$.
In order to bridge this gap and open up topology to computational methods, we need a notion of discretised topological spaces consisting of finitely many base points with finite description length.
A \emph{simplicial complex} is the simplest discrete model that can still approximate any topological space occuring in practice \citep{Quillen:1967}:

\begin{definition}[Simplicial complexes]
	A \emph{simplicial complex} (\tSC) $\SC$ consists of a set of vertices $V$ and a set of finite non-empty subsets (simplices, $S$) of $V$ closed under taking non-empty subsets, such that the union over all simplices $\bigcup_{\sigma\in S}\sigma$ is $V$.
	We will often identify $\SC$ with its set of simplicies $S$ and denote by $\SC_k$ the set of simplices $\sigma\in S$ with $|\sigma| =k+1$, called \emph{$k$-simplices}.
	We say that $\SC$ is $n$-dimensional, where $n$ is the largest $k$ such that %the set of $k$-simplices 
	$\SC_k$ is non-empty.
	The \emph{$k$-skeleton} of $\SC$ contains the simplices of dimension at most $k$.
	If the vertices $V$ lie in $\R^n$, we call the convex hull in $\R^n$ of a simplex $\sigma$ its \emph{geometric realisation} $|\sigma|$.
	When doing this for every simplex of $\SC$, we call this the \emph{geometric realisation of $\SC$}, $|\SC|\subset\R^n$.
\end{definition}	

We can construct an $n$-dimensional \tSC{} $\SC$ in $n+1$ steps:
We start with a set of vertices $V$ ($\SC_0$).
We then connect certain pairs of vertices with edges ($\SC_1$).
Afterwards, we fill in some fully connected triangles ($\SC_2$), and so on.

\paragraph{Vietoris--Rips and $\alpha$-complexes}

We now need a way to construct a \emph{simplicial complex} that approximates the \emph{topological structure} inherent in our data set $X\subset\R^n$.
\iftoggle{arxiv}{Such a construction will always depend on the scale of the structures we are interested in.
When looking from a very large distance, the point cloud will appear as a singular connected blob in the otherwise empty and infinite real space, on the other hand when we continue to zoom in, the point cloud will at some point appear as a collection of individual points separated by empty space; all  interesting information can be found in-between these two extreme scales where some vertices are joined by simplices and others are not.
}{}
Instead of having to pick a single scale, the \emph{Vietoris--Rips (\VR) filtration} and the \emph{$\alpha$-filtration} take as input a point cloud and return a nested sequence of simplicial complexes indexed by a scale parameter $\varepsilon$ approximating the topology of the data across all possible scales.

\begin{definition}[\VR~complex]
	Given a finite point cloud $X$ in a metric space $(\mathcal{M},d)$ and a non-negative real number $\varepsilon\in\R_{\ge 0}$, the associated \VR~complex $VR_\varepsilon(X)$ is given by the vertex set $X$ and the set of simplices
	$
	S = \left\{\sigma\subset X \mid \sigma\not = \emptyset, \forall x,y\in \sigma : d(x,y)\le \varepsilon\right\}
	$.
\end{definition}

A \VR~complex at $\varepsilon$ consists of all simplices $\sigma$ where all vertices $x\in\sigma$ have distance of at most $\varepsilon$.
For $r\le r'$, we obtain the canonical inclusions $i_{r,r'}(X)\colon VR_r(X)\hookrightarrow VR_{r'}(X)$.
For a simplex $\sigma$ and a filtration $F$, we denote by its \emph{filtration value} $F(\sigma)$ the smallest $\varepsilon$ such that $\sigma \in F_\varepsilon$.
The set of \VR~complexes on $X$ for all possible $r\in\R_{\ge 0}$ together with the inclusions then form the \emph{\VR~filtration} on $X$.
For large point clouds, the \VR~filtration becomes computationally inefficient becomes expensive due to its large number of simplices.
In contrast, the $\alpha$-filtration approximates the topology of a point cloud using far fewer simplices. Thus we will make use of $\alpha$-complexes in settings of ambient dimension lower than $4$, where their construction is computationally feasible.
For a complete discussion and definition of $\alpha$-complexes, see \Cref{app:TheoreticalConsiderations}.
%We note that the filtration value of a $k$-simplex in the $\alpha$-filtration is related to the radius of its circum-$k$-sphere, which differs from its filtration value in the \VR~filtration.

\paragraph{Boundary matrices}

\iftoggle{arxiv}{So far, we have discussed a discretised version of topological spaces in the form of \tSC{}s and a way to turn point clouds into a sequence of \tSC{}s indexed by a scale parameter.
However, w}{W}e still need an \emph{algebraic representation} of simplicial complexes that is capable of encoding the structure of the \tSC{} and enables extraction of the \emph{topological features}:
The \emph{boundary matrices} $\bound_k$ associated to an \tSC{} $\SC$ store all structural information of the \tSC.
The rows of $\bound_k$ are indexed by the $k$-simplices of $\SC$ and the columns are indexed by the $(k+1)$-simplices. 
\begin{definition}[Boundary matrices]
	Let $\SC$ be a simplicial complex and $\preceq$ a total order on its vertices $V$.
	Then, the $i$-th face map in dimension $n$ $f^n_i\colon\SC_n\rightarrow \SC_{n-1}$ is given by
	\begin{align*}
		f^n_i&\colon \{v_0,v_1,\dots,v_n\}\mapsto \{v_0,v_1,\dots,\widehat{v}_i,\dots,v_n\}
	\end{align*}
	with $v_0\preceq v_1\preceq\dots\preceq v_n$ and $\widehat{v}_i$ denoting the omission of $v_i$.
	Now, the $n$-th \emph{boundary operator} $\bound_n\colon \R[\SC_{n+1}]\rightarrow\R[\SC_{n}]$ with $\R[\SC_{n}]$ being the real vector space over the basis $\SC_{n}$ is given by
	\[
		\bound_n\colon \sigma\mapsto\smash{\sum_{i=0}^{n+1}}(-1)^if^{n+1}_i(\sigma).
	\]
	
	When lexicographically ordering the simplex basis, we can view $\bound_n$ as a \emph{matrix}.
	We call $\R[\SC_{n}]$ the space of $n$-chains.
	$\bound_0$ is the vertex-edge incidence matrix of the associated graph consisting of the $0$- and $1$-simplices of~$\SC$ and $\bound_1$ is the edge-triangle incidence matrix of $\SC$
\end{definition}
\paragraph{Betti Numbers and Persistent Homology}
\begin{figure*}[tb!]
	\vskip -0.1in
	\begin{center}
		\centerline{
			\includegraphics[width=0.15\linewidth]{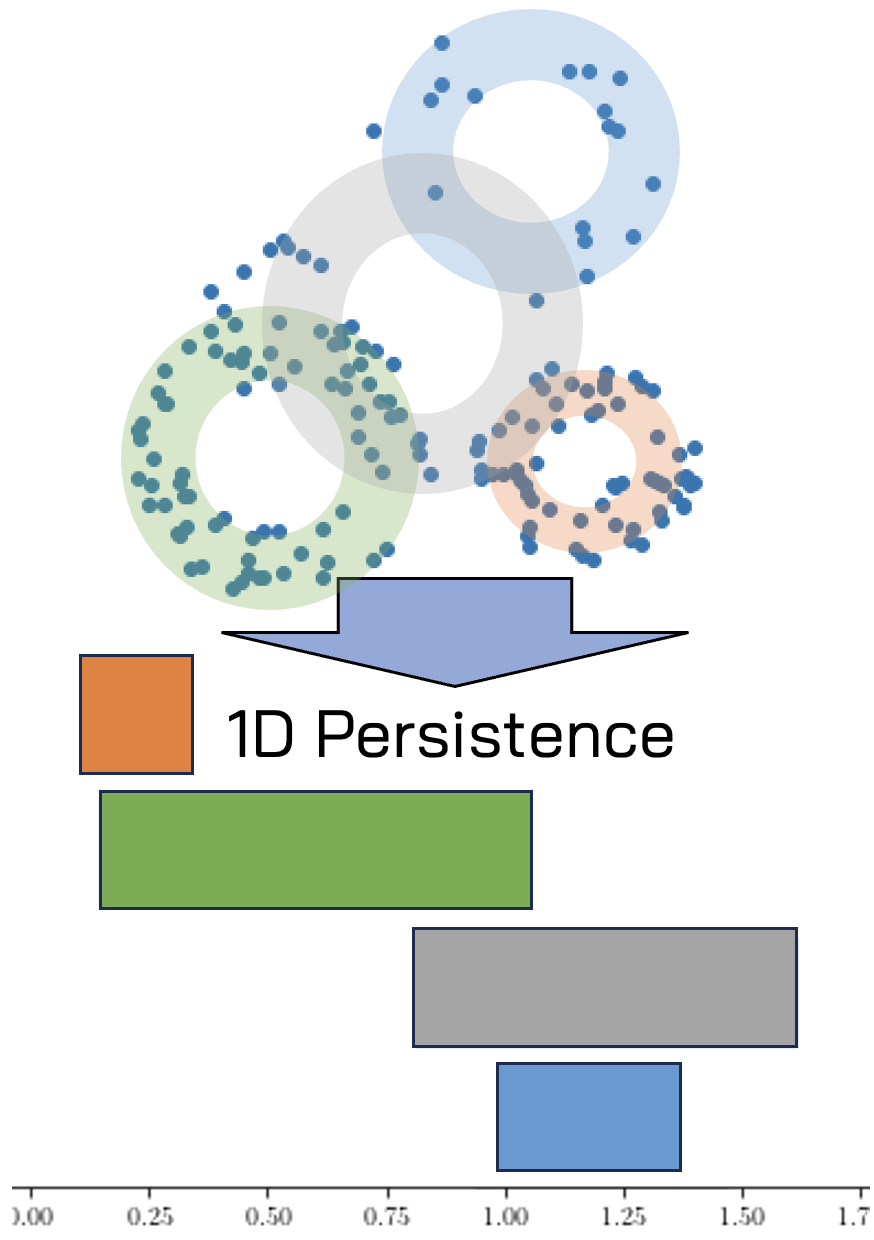}
			\includegraphics[width=0.28\linewidth]{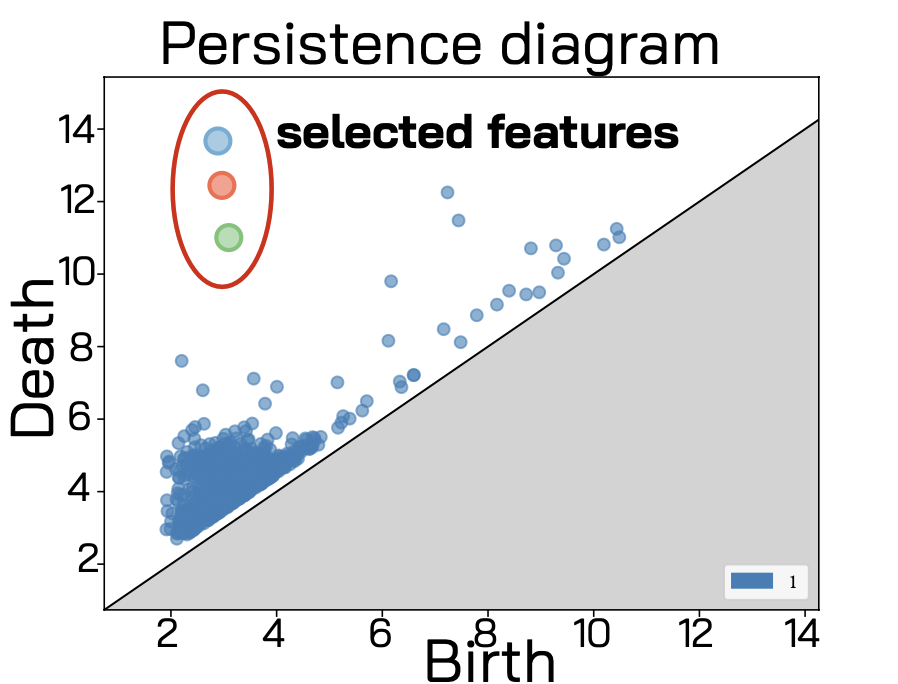}
			\includegraphics[width=0.28\linewidth]{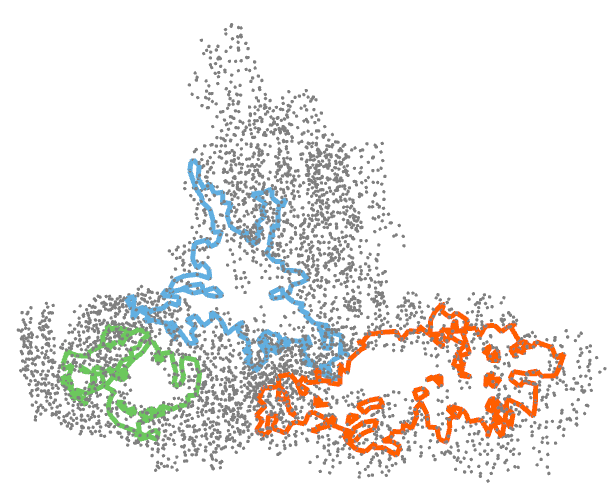}
			\includegraphics[width=0.27\linewidth]{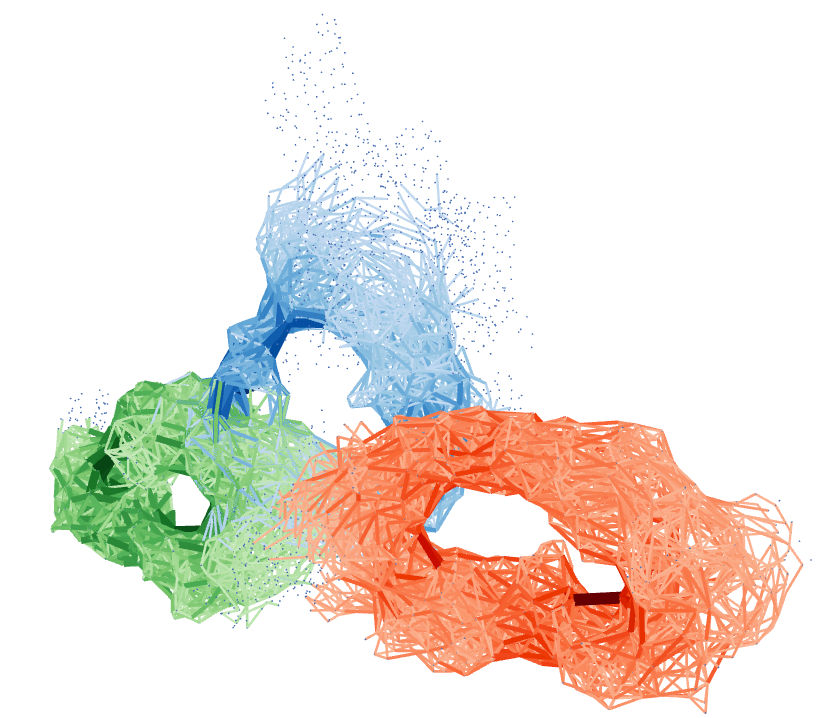}
		}
		\caption{\figuretitle{\PH{} sketch and \TOPF{} pipeline applied to \textsc{nalcn} channelosome, a membrane protein \citep{Kschonsak:2022}.}
			\emph{Left:} Bars represent life times of features \cite{grande2023nonisotropic}.
			\emph{Centre left:} Steps \figuretitle{1}\&\figuretitle{2a}, when computing persistent $1$-homology, three classes are more prominent than the rest.
			\emph{Centre right:} Step \textbf{2b}: The selected homology generators.
			\emph{Right:} Step \textbf{3}: The projections of the generators into harmonic space are now each supported on one of the rings.}
		\label{fig:ProteinAlgorithmSteps}
		\label{fig:Persistence}
	\end{center}
	\vskip -0.3in
\end{figure*}
We now turn to the notion of \emph{topological features} and how to extract them.
\emph{Homology} is one of the main algebraic invariants to capture the shape of topological spaces and \tSC.
The $k$-th homology module $H_k(\SC)$ of an \tSC~$\SC$ with boundary operators $\bound_k$ is defined as $H_k(\SC)\coloneq\ker \bound_{k-1}/\Ima \bound_{k}$.
The \emph{generator} or representative of a homology class is an element of the kernel $\ker \bound_{k-1}$.
In dimension $1$, these are given by formal sums of $1$-simplices forming closed loops in the \tSC{}.
Importantly, the rank $\rk H_k(\SC)$ is called the $k$-th \emph{Betti number} $\betti_k$ of $\SC$.
In dimension $0$, $\betti_0$ counts the number of connected components, $\betti_1$ counts the number of loops around \enquote*{holes} of the space, $\betti_2$ counts the number of $3$-dimensional voids with $2$-dimensional boundary, and so on.

\iftoggle{arxiv}{If we are now given a filtration of simplicial complexes instead of a single \tSC{}, we can track how the homology modules \emph{evolve} as the simplicial complex \emph{grows}.
}{Given a \emph{filtration} of \tSC{}s, we can track how the homology modules \emph{evolve} as the simplicial complex \emph{grows}.}
The mathematical formalisation, \emph{persistent homology}, thus turns a point cloud via a simplicial filtration into an algebraic object summarising the topological features of the point cloud \citep{Edelsbrunner2008}.
\iftoggle{arxiv}{For performance reasons, this is done in a small prime field $\Z/p\Z$.
Because we use signs of numbers to distinguish simplex orientations, we set $p=3$\iftoggle{arxiv}{, with $\Z/3\Z$ being the smallest field being able to distinguish $1$ and $-1$.}{.}}{}

\paragraph{The Hodge Laplacian and the Harmonic Space}
\iftoggle{arxiv}{In the previous part, we have introduced a language to characterise the global shape of spaces and point clouds.
However, w}{W}e need to relate the \emph{global characterisations} of \PH{} back to \emph{local properties} of the point cloud.
We will do so by using ideas and concepts from differential geometry and topology:
The simplicial Hodge Laplacian is a discretisation and generalisation of the Hodge--Laplace operator acting on differential forms of manifolds:

\begin{definition}[Hodge Laplacian]
	Given a simplicial complex $\SC$ with boundary operators $\bound_k$, we define the $n$-th Hodge Laplacian $L_n\colon \R[\SC_{n}]\rightarrow\R[\SC_{n}]$ by setting
	\[
	\smash{L_n\coloneq \bound_{n-1}^\top\bound_{n-1}+\bound_n\bound_{n}^\top.}
	\]
\end{definition}
%The Hodge Laplacian gives rise to the Hodge decomposition:

\begin{theorem}[Hodge Decomposition \citep{Lim:2020,  Schaub:2021, Roddenberry:2021}]
	\label{thm:HodgeDecomposition}
	For an \tSC{} $\SC$ with boundary matrices $\left(\bound_i\right)$ and Hodge Laplacians $\left(L_i\right)$, we have in every dimension $k$
	\[
	\R[\SC_k]=
	\iftoggle{arxiv}{
		\underbrace{\Ima \bound_{k-1}^\top}_\text{gradient space} \oplus  \underbrace{\ker L_k}_\text{harmonic space}\oplus \underbrace{\Ima \bound_{k}}_\text{curl space}.
	}{\smash{\underbrace{\Ima \bound_{k-1}^\top}_\text{gradient space} \oplus  \underbrace{\ker L_k}_\text{harmonic space}\oplus \underbrace{\Ima \bound_{k}}_\text{curl space}.}}
	\]
\end{theorem}
This, together with the fact that the $k$-th harmonic space is isomorphic to the $k$-th real-valued homology group $\ker L_k\cong H_k(\R)$ means that we can associate a \emph{unique harmonic representative} to every homology class.
\iftoggle{arxiv}{The harmonic space encodes higher-order generalisations of smooth flow around the holes of the simplicial complex.}{}
Intuitively, this means that for every \emph{abstract global homology class} of persistent homology at filtration step $t$ from above we can now compute one \emph{unique harmonic representative} in $\ker L_k$ that assigns every simplex a value based on how much it contributes to the homology class.
Thus, the Hodge Laplacian is a gateway between the \emph{global topological features} and the \emph{local properties} of our \tSC.
\iftoggle{arxiv}{
It is easy to show that the kernel of the Hodge Laplacian is the intersection of the kernel of the boundary and the coboundary map $\ker L_k = \ker \bound_{n-1}\cap\ker \bound_n^\top$.
Because we have finite \tSC{}s we can identify the spaces of chains and cochains.
This leads to another characterisation of the harmonic space: The space of chains that are simultaneously homology and cohomology representatives.
}{}
We discuss how this relates to the theory of differential forms and Hodge theory in the continuous case in \Cref{app:differentialforms}.

\section{How to Compute Topological Point Features}
\label{sec:algorithm}
\begin{algorithm}[tb]
	\caption{Topological Point Features (\TOPF)}
	\label{alg:TOPF}
	\begin{algorithmic}
		\STATE {\bfseries Input:} Point cloud $X\in\R^n$, maximum homology dimension $d\in\N$, interpolation coeff.\ $\lambda$.
		\STATE \textbf{1.} Compute persistent homology with generators in dim.\ $k\le d$. (Sec.\ 2: \textit{Betti Numbers \& Persistent Homology})
		\STATE \textbf{2.} Select set of significant features $(b_i,d_i,g_i)$ with birth, death, and generator in $\mathbb{F}_3$ coordinates (See \textbf{Step 2}).
		\STATE \textbf{3.} Embed $g_i$ into real space \eqref{eq:coefficientlifting}, and project into harmonic subspace \eqref{eq:curlgradprojection} of \tSC{} at step $d_i^\lambda b_i^{1-\lambda}$ or $\lambda d_i +(1-\lambda) b_i$.
		\STATE \textbf{4.} Normalise projections to $\mathbf{e}_i^k$ and compute $F_k^i(x)\coloneq \operatorname{avg}_{x\in\sigma}(\mathbf{e}_i^kl(\sigma))$ for all points $x\in X$ \eqref{eq:featurevector}.
		\STATE \textbf{Output:} Features of $x\in X$
	\end{algorithmic}
\end{algorithm}
\begin{comment}
	\begin{algorithm}[tb]
		\caption{Topological Point Features (\TOPF)}
		\label{alg:TOPF}
		\begin{algorithmic}
			\STATE {\bfseries Input:} Point cloud $X\in\R^n$, maximum homology dimension $d\in\N$, interpolation coeff.\ $\lambda$.
			\STATE \textbf{1.} Compute persistent homology with generators in dimension $k\le d$.
			\STATE \textbf{2.} Select set of significant features $(b_i,d_i,g_i)$ with birth, death, and generator in $\mathbb{F}_3$ coordinates.
			\STATE \textbf{3.} Embed $g_i$ into real space and project into harmonic subspace of \tSC{} at step $t=d_i^\lambda b_i^1-\lambda$ or $t=\lambda d_i +(1-\lambda) b_i$.
			\STATE \textbf{4.} Normalise projections to $\mathbf{e}_i^k$ and compute $F_k^i(x)\coloneq \operatorname{avg}_{x\in\sigma}(\mathbf{e}_i^kl(\sigma))$ for all points $x\in X$.
			\STATE \textbf{Output:} Features of $x\in X$
		\end{algorithmic}
	\end{algorithm}
\end{comment}
In this section, we will combine the ideas and insights of the previous section to give a complete account of how to compute topological point features (\TOPF). A pseudo-code version can be found in \Cref{alg:TOPF} and an overview in \Cref{fig:fig1}.
We start with a finite point cloud $X\subset \R^n$.

\paragraph{Step 1: Computing the persistent homology}

First, we determine the \emph{most significant persistent homology classes} which determine the shape of the point cloud.
Doing this, we also extract the \enquote{interesting} scales of the data set.
We will later use this to construct \tSC{}s to localise the global homology features.
Thus we first compute the persistent $k$-homology modules $P_k$ including a set of homology representatives $R_k$ of $X$ using an $\alpha$-filtration for $n\le 3$ and a \VR~filtration for $n>3$.
We use $\Z/3\Z$ coefficients to be sensitive to simplex orientations.
In case we have prior knowledge on the data set, we can choose a real number $R\in\R_{>0}$ and only compute the filtration up to the parameter $R$.
In data sets like protein atom coordinates, this might be useful as we have prior knowledge on what constitutes the \enquote{interesting} scale, reducing computational complexity.
See \Cref{fig:ProteinAlgorithmSteps} \emph{centre left} for a \PH{} diagram.

\paragraph{Step 2: Selecting the relevant topological features}

We now need to select the relevant \emph{persistent homology classes} which carry the most important \emph{global information}.
The persistent homology $P_k$ module in dimension $k$ is given to us as a list of pairs of birth and death times $(b_i^k,d_i^k)$.
We can assume these pairs are ordered in non-increasing order of the durations $l_i^k=d_i^k-b_i^k$.
\iftoggle{arxiv}{This list is typically very long and consists to a large part of noisy homological features which vanish right after they appear.
In contrast, w}{W}e are interested in connected components, loops, cavities, etc.\ that \emph{persist} over a long time, indicating that they are important for the shape of the point cloud.
Distinguishing between the relevant and the irrelevant features is in general difficult and may depend on additional insights on the domain of application.
In order to provide a heuristic which does not depend on any a-priori assumptions on the number of relevant features we pick the smallest quotient $q_i^k\coloneq l_{i+1}^k/l_i^k>0$ as the point of cut-off $\smash{N_k\coloneq\argmin_{i} q_i^k}$.
The only underlying assumption of this approach is that the band of \enquote{relevant} features is separated from the \enquote{noisy} homological features by a drop in persistence.
If this assumption is violated, the only possible way to do meaningful feature selection depends on application-specific domain knowledge.
We found that our proposed heuristics work well across a large scale of applications.
See \Cref{fig:ProteinAlgorithmSteps} \emph{left} and \emph{centre} for an illustration and \Cref{app:featureselection} for more technical details and ways to improve and adapt the feature selection module of \TOPF.
We call the chosen  $k$-homology classes with $k$-homology generators $f^i_k$.

\paragraph{Step 3: Projecting the features into harmonic space and normalising}

In this step, we need to relate the \emph{global topology} extracted in the previous step to the simplices which we will use to compute the \emph{local} topological point features.
Every selected feature $f_k^i$ of the previous step comes with a birth time $b_{i,k}$ and a death time $d_{i,k}$.
This means that the homology class $f_k^i$ is present in every \tSC{} of the filtration between step $\varepsilon = b_{i,k}$ and $\varepsilon = d_{i,k}$ and we could choose any of the \tSC{}s for the next step.
Picking a \emph{small} $\varepsilon$ will lead to \emph{fewer} simplices in the \tSC{} and thus to a very \emph{localised} harmonic representative.
Picking a \emph{large} $\varepsilon$ will lead to \emph{many} simplices in the \tSC{} and thus to a very \emph{smooth} and \enquote{blurry} harmonic representative with large support.
Finding a middle ground between these regimes returns optimal results.
Given interpolation hyperparameter $\gamma\in (0,1)$, we will thus consider the simplicial complex $\SC^{t_{i,k}}(X)$ at step $\smash{t_{i,k}\coloneq b_{i,k}^{1-\gamma} d_{i,k}^{\gamma}}$ for $k>0$ and at step $t_{i,k}\coloneq \gamma d_{i,k}$ for $k=0$ of the simplicial filtration.
At this point, the homology class $f_k^i$ is still alive.
We then consider the real vector space $\smash{\R[\SC_k^{t_{i,k}}(X)]}$ with formal basis consisting of the $k$-simplices of the \tSC{}~$\SC^{t_{i,k}}$.
From the persistent homology computation of the first step, we also obtain a generator of the feature $f_k^i$, consisting of a list $\Sigma_k^i$ of simplices $\smash{\hat{\sigma}_j\in \SC^{b_{i,k}}_k}$ and coefficients $c_j\in\Z/3\Z$.
We need to turn this formal sum of simplices with $\Z/3\Z$-coefficients into a vector in the real vector space $\smash{\R[\SC_k^{t_{i,k}}(X)]}$:
Let $\iota\colon\Z/3\Z$ be the map induced by the canonical inclusion of $\{-1,0,1\}\hookrightarrow\R$.
We can now define an indicator vector $\smash{e_k^i\in\R[\SC_k^{t_{i,k}}(X)]}$ associated to the feature $f_k^i$ (Cf.\ \cite{DeSilva2009persistent}).
\begin{equation}
	\label{eq:coefficientlifting}
e_k^i(\sigma)\coloneq\smash{\begin{cases}
	\iota(c_j)&\exists \hat{\sigma}_j\in \Sigma_k^i:\sigma = \hat{\sigma}_j\\
	0 &\text{else}
\end{cases}}.
\end{equation}
Empirically, $e^i_k$ is a homology generator for all real coefficients as well, although our construction only guarantees for $\bound_{k-1}e^i_k\equiv 0 \mod 3$.
We discuss how to fix the rare case where this does not work in \Cref{app:fixcycles} and now assume to work with a real homology representative $e^i_k$.
While this homology representative lives in a real vector space, it is not unique, has a small support, and its value can differ largely even between close simplices.
All of these problems can be solved by projecting the homology representative to the harmonic subspace $\ker L_k$ of $\smash{\R[\SC_k^{t_{i,k}}(X)]}$.
Rather than directly projecting $e_k^i$ to the harmonic subspace, we make use of the Hodge decomposition theorem (\Cref{thm:HodgeDecomposition}) which allows us to %compute the gradient and curl projections 
solve computationally efficient least square problems:% to finally \textbf{compute the harmonic representative}:
\begin{align}
	\label{eq:curlgradprojection}
e_{k,\text{curl}}^i&\coloneq \bound_{k}\argmin_{x\in \R[\SC_{k+1}]} \left\lVert e_k^i-e_{k,\text{grad}}^i -\bound_{k}x\right\rVert_2^2
\end{align}
and get the \textbf{harmonic representative}
$\hat{e}_k^i\coloneqq e_k^i-e_{k,\text{grad}}^i-e_{k,\text{curl}}^i$. (Cf.\ \Cref{fig:ProteinAlgorithmSteps} \emph{right} for a visualisation.)
Because homology representatives are gradient-free, $\smash{e_{k,\text{grad}}^i = 0}$ and we only need to compute $\smash{e_{k,\text{curl}}^i}$. 

\paragraph{Step 4: Processing and aggregation at a point level}
In the previous step, we have computed a set of harmonic representatives of homology classes.
Each harmonic representative assigns each simplex in the correspondings \tSC{}s a value.
However, these simplices likely have no real-world meaning and the underlying simplicical complexes differ depending on the birth and death times of the homology classes.
Hence in this step, we will collect the features on the point-level after performing some necessary preprocessing.
Given a vector $\hat{e}_k^i$ indexed over simplices and a hyperparameter $\delta$, we now construct $\mathbf{e}_k^i\colon \SC_k^{t_{i,k}}(X)\rightarrow [0,1]$ by setting 
$
\mathbf{e}_k^i\colon \sigma \mapsto \in \{|\hat{e}_k^i(\sigma)|/(\delta\max_{\sigma' \in \SC_k^{t_{i,k}}(X)}|\hat{e}_k^i(\sigma')|),1\}
$
 such that $\hat{e}_k^i$ is normalised to $\smash{[0,1]}$, then the values of $\smash{[0,\delta]}$ are mapped linearly to $\smash{[0,1]}$ and everything above is sent to $1$.
We found empirically that a thresholding parameter of $\delta\approx 0.07$ works best across at the range of applications considered below. However, \TOPF{} is not sensitive to small changes to $\delta$ because  entries of $\hat{e}_k^i$ are concentrated around $0$ (cf.\ \Cref{app:hyperparameters}).

For every feature $f_k^i$ in dimension $k$ with processed simplicial feature vector $\mathbf{e}_k^i$ and simplicial complex $\SC^{t_{i,k}}$, we define the point-level feature map $F^k_i\colon X\rightarrow \R$ mapping from the initial point cloud $X$ to $\R$ by setting
\begin{equation}
	\label{eq:featurevector}
F^k_i\colon v\mapsto\frac{\sum_{\sigma_k\in\SC_k^{t_{i,k}}\colon v\in \sigma_k}\mathbf{e}_k^i(\sigma_k)}{\max(1,|\{\sigma_k\in\SC_k^t:v\in \sigma_k\}|)}.
\end{equation}
For every point $v	$, we can thus view the vector $(F^k_i(v)\colon f^k_i\in\feat )$ as a feature vector for $v$.
We call this collection of features \emph{Topological Point Features} (\TOPF{}). (Cf. \Cref{fig:QualitativeExperiments} for an example).

\paragraph{Choosing Simplicial Weights}
The above discussed theory works analogously for weighted \tSC{}s.
We discuss how \TOPF{} employs weighted \tSC{}s to mitigate the influence of noise and heterogeneous sampling in \Cref{app:WeightedSCs}.
\paragraph{Theoretical Guarantees}
In the appendix in \Cref{thm:TOPFsphere} we give theoretical guarantees for when \TOPF{} provably works on an idealised point cloud.
\section{Experiments}

\begin{figure}[tb!]
	\begin{center}
		\centerline{
			\includegraphics[width=\linewidth]{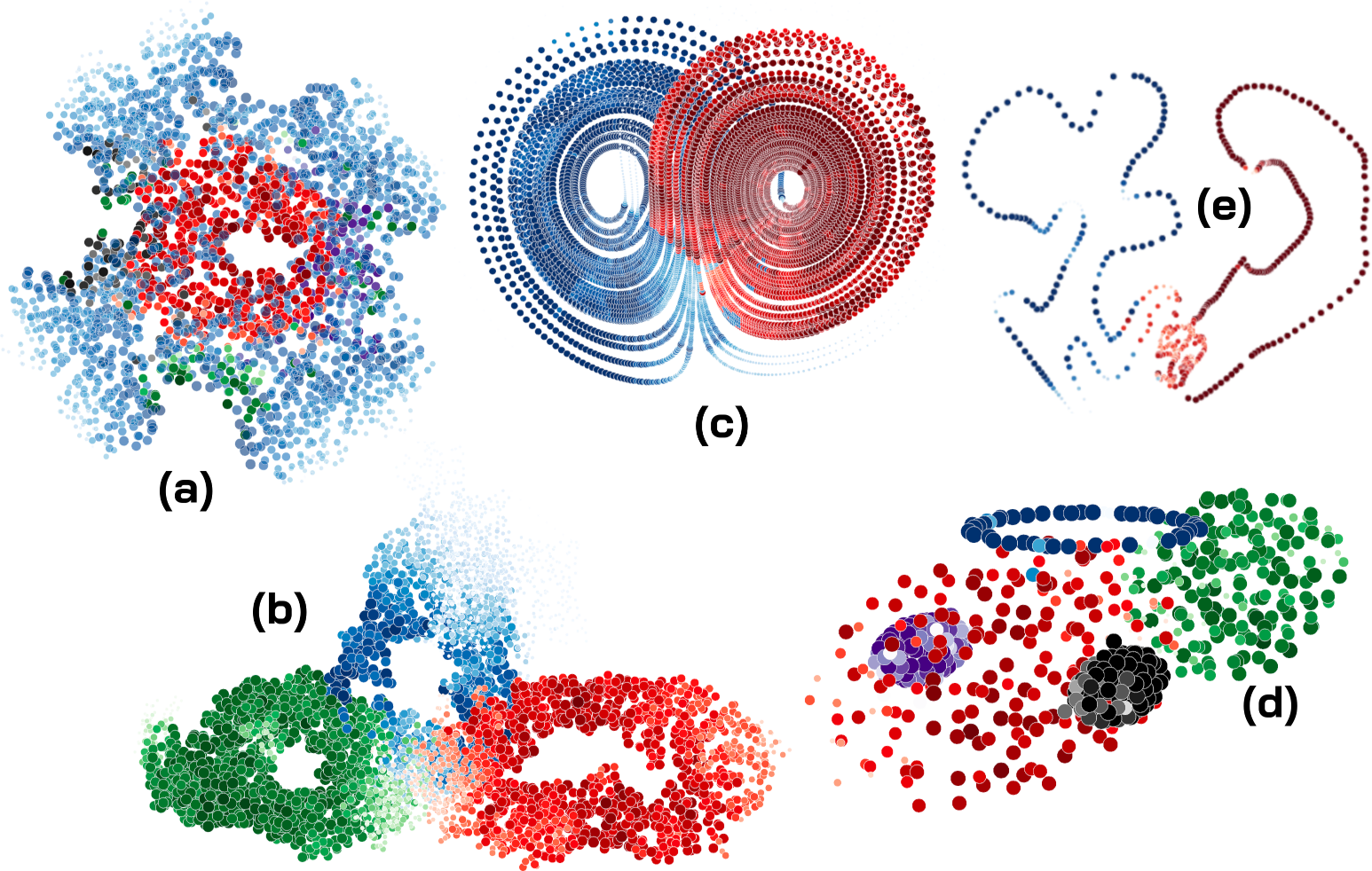}
		}
		\caption{\figuretitle{
				\TOPF{} on \oldstylenums{3}d real-world and synthetic point clouds.}
			For every point, we highlight the largest corresponding topological feature, where colour stands for the different features and saturation for the value of the feature.
			\emph{(b):} Atoms of \textsc{nalcn} channelosome \citep{Kschonsak:2022} display three distinct loops.
			\emph{(c):} Points sampled in the state space of a Lorentz attractor.
			The two features correspond to the two lobes of the attractor.
			\emph{(d):} Point cloud \texttt{spaceship} of our newly introduced topological clustering benchmark suite (See \Cref{app:BenchmarkSuite}).
			\emph{(e):} Latent space of a \VAE{} trained on image patches showing topological structure (See \Cref{fig:figcow} for details).
		}
		\label{fig:QualitativeExperiments}
	\end{center}
	\vspace{-0.2in}
\end{figure}
\label{sec:experiments}
In this section, we conduct experiments on real world and synthetic data, compare the clustering results with clustering by \TPCC{}, other classical clustering algorithms, and other point features, and demonstrate the robustness of \TOPF{} against noise.
In \Cref{tab:ablation}, we perform an ablation study with respect to the harmonic projections of step 3 of \TOPF{}.

\paragraph{Topological Point Cloud Clustering Benchmark}
We introduce the topological clustering benchmark suite (\Cref{app:BenchmarkSuite}) and report running times and the accuracies of clustering based on \TOPF{} and other methods and point embeddings, see \Cref{fig:quantitativeComparison}.
We see that spectral clustering on \TOPF{} vectors (listed as \TOPF{}) \emph{outperforms} all classical clustering algorithms on all but one dataset by a wide margin.
We also see that \TOPF{} closely matches the performance of the only other higher-order topological clustering algorithm, \TPCC{} on two datasets with clear topological features, whereas \TOPF{} \emph{outperforms} \TPCC{} on datasets with more complex structure.
In addition, \TOPF{} has a consistently lower running time with better scaling for the more complex datasets, while not requiring prior knowledge on the best topological scale.
As for the other point embeddings, Node2Vec is not able to capture any meaningful topological information, whereas the performance of clustering using geometric features depends on the data set.
Furthermore, \TOPF{} outperforms PointNet \citep{qi2017pointnet} pretrained on the ShapeNetPart data set and WSDesc \citep{Li2022WSDesc} pretrained on the 3DMatch dataset \citep{zeng20163dmatch}.
This highlights that neural network architectures like PointNet need specific application-specific training data for good performance that is simply not available in all cases.

\TOPF{} is an unsupervised method for extracting interpretable topological features founded in algebraic topology and differential geometry. Neural network architectures like PointNet or WSDesc need labelled training data sets, produce hard-to interpret feature representation, but perform well on down-stream tasks with enough data available.
%A sensible and fair comparison between these two approaches is thus hard to justify.

%\paragraph{Supervised vs.\ Unsupervised}
%
\begin{table*}
	\centering
	\scriptsize
	\caption{
		\figuretitle{Quantitative performance comparison of clustering with \TOPF{} and other features/clustering algorithms.} Four \oldstylenums{2}d and three \oldstylenums{3}d data sets of the topological clustering benchmark suite (\Cref{app:BenchmarkSuite}, cf. \Cref{fig:TCBS} for ground truth and \Cref{fig:TCBSTOPF} for labels by \TOPF).
		We ran each algorithm $20$ times and list the mean adjusted rand index (\ARI{}) with standard deviation $\sigma$ and mean running time.
		We omit $\sigma$ for algorithms with $\sigma =0$.
		Spectral clustering on \TOPF{} vectors consistently outperforms or almost matches the other algorithms while having significantly better run time than the second best performing algorithm \TPCC{}.
		Spectral Clustering (\textsc{\lowercase{SC}}), \textsc{\lowercase{DBSCAN}}\iftoggle{arxiv}{, and Agglomerative Clustering (AgC)}{} (applied on the points) are standard clustering algorithms, \textsc{\lowercase{ToMATo}} is a topological clustering algorithm \citep{Chazal2013}, Geo clusters using $12$-dimensional point geometric features extracted by \texttt{pgeof}\iftoggle{arxiv}{}{, whereas node2vec \citep{grover2016node2vec} produces node embeddings on  a $k$-nearest neighbour graph}. %built upon an affinity matrix.
		We pretrained PointNet (\textsc{\lowercase{PN}}) \citep{qi2017pointnet} on the ShapeNetPart point cloud segmentation data set \citep{chang2015shapenet}.
		\enquote{Weakly Supervised 3D Local Descriptor Learning for Point Cloud Registration} (\textsc{wsd} \citep{Li2022WSDesc}) is pretrained on 3DMatch data %, but has shown good performance on unseen datasets 
		and produces $32$-dimensional feature vectors. %  we use for clustering together with the original coordinates. %, as this significantly boosted performance on the $3D$ datasets.
		\textsc{dgcnn} \cite{wasserman2018topological} is pretrained on ShapeNetPart segmentation, with the parameters determined by a hyperparameter search.
		We highlight all \ARI{} scores within $\pm0.05$ of the best \ARI{} score.
	}
	\iftoggle{arxiv}{
	\begin{tabular}{llHccHcHHcHHccHccc}
	}{
	\begin{tabular}{llHccHcHHccHcccccc}
	}
		& & \TPCC * & \textbf{\textsmaller{TOPF}}& \TPCC & \TOPF * & \textsc{sc} & $k$-means & \textsc{optics} & \textsc{dbscan} & AgC & \textsc{ms} &\textsc{t}o\textsc{mat}o& Geo & node2vec&\textsc{pn}&\textsc{wsd}&\textsc{dgcnn}\\
		\toprule
\texttt{4spheres} &\ARI &  0.3$\pm$0.09 & \textbf{0.81} &0.52$\pm$0.17 & 0.63 & 0.37 & 0.39 & 0.15 & 0.00 & 0.45 & 0.20 & 0.32 & 0.20&0.00$\pm$0.00&0.30 &0.13&0.29$\pm$0.03\\
		& time (s) & 70.3 & 13.9 & 23.3 & 12.7 & 0.2 & <0.1 & 0.2 & <0.1 & <0.1 & 0.9 & <0.1 & 0.2&48.4&<0.1 &<0.1&<0.1\\ \midrule
		\texttt{Ellipses} & \ARI &  0.45$\pm$0.05 & \textbf{0.95} &0.47$\pm$0.04 & 0.89 & 0.25 & 0.55 & 0.36 & 0.19 & 0.52 & 0.29 & 0.29 & 0.81&0.02$\pm$0.00&0.50&0.35&0.70$\pm$0.00\\
		& time (s) & 17.0 &13.5 & 14.4 & 12.5 & 0.1 & <0.1 & 0.1 & <0.1 & <0.1 & 0.2 & <0.1 & 0.1&11.2&<0.1 &<0.1&<0.1\\ \midrule
		\texttt{4Circles+Grid} & \ARI &\texttt{error} & 0.70 & 0.39$\pm$0.04 &  \textbf{0.87} & \textbf{0.90} & 0.83 & 0.68 & \textbf{0.92} & \textbf{0.89} & 0.10 & 0.82& 0.41&0.01$\pm$0.00&0.55&0.06&0.75$\pm$0\\
		& time (s) & 87.6 &13.7 & 28.5 & 12.9 & 0.5 & <0.1 & 0.3 & <0.1 & <0.1 & 2.3 & <0.1 &0.3&63.8&<0.1 &<0.1&<0.1\\ \midrule
		\texttt{Halved Circle}& \ARI{} & 0.25$\pm$0.18 & \textbf{0.71} & 0.18$\pm$0.12 & 0.44 & 0.24 & 0.23 & 0.07 & 0.00 & 0.20 & 0.04 & 0.16&0.08&0.00$\pm$0.01&0.36&0.00&0.39$\pm$0.07\\
		& time (s) &16.0 & 13.6& 14.3 &  12.0 & 0.1 & <0.1 & 0.1 & <0.1 & <0.1 & 0.4 & <0.1 &0.1&18.2&<0.1 &<0.1&<0.1\\ \midrule[1pt]
		\texttt{2Spheres2Circle} & \ARI{} & \texttt{TO} & \textbf{0.94} & \textbf{0.97}$\pm$0.01& \textbf{0.93} & 0.70 & 0.47 & 0.85 & 0.00 & 0.51 & 0.84 & 0.87&0.12&0.00$\pm$0.00&0.55&\textbf{0.95}&0.84$\pm$0.18\\
		& time (s) &  \texttt{TO} & 20.7 &1662.2 & 40.0 & 1.6 & 0.1 & 1.8 & <0.1 & 0.3 & 5.9 & <0.1 &0.9&348.6&<0.1 &<0.1&0.2\\ \midrule
		\texttt{SphereinCircle} & \ARI{} &  0.90 & \textbf{0.97} &\textbf{0.98}$\pm$<0.1 & 0.86 & 0.34 & 0.02 & 0.19 & 0.00 & 0.29 & 0.00 & 0.06 &0.69&0.13$\pm$0.03&0.39&0.46&0.76$\pm$0.06\\
		& time (s) &  219.3 & 14.5 &8.0 & 13.9 & <0.1 & <0.1 & 0.1 & <0.1 & <0.1 & 0.3 & <0.1 &0.08&20.1&<0.1 &<0.1&<0.1\\ \midrule
		\texttt{Spaceship} & \ARI{} & \texttt{error} & \textbf{0.92} & 0.56$\pm$0.03 & 0.72 & 0.28 & 0.42 & 0.40 & 0.26 & 0.47 & 0.32 & 0.30&\textbf{0.87}&0.07$\pm$0.00&0.41&0.76&0.79$\pm$0.00\\
		& time (s) &  \texttt{error} &15.5 &341.8 & 0.5 & 16.7 & <0.1 & 0.2 & <0.1 & <0.1 & 0.8 & <0.1 &0.2&49.8&<0.1 &<0.1&<0.1\\ \bottomrule
		\textbf{mean} &\ARI{} & & \textbf{0.86} &0.58 & & 0.44 & 0.42 & 0.39 & 0.16 & 0.48 & 0.26 & 0.40&0.45&0.03&0.44&0.39&0.64\\
		& time (s) &  &15.1 &298.9& & 0.4 & <0.1 & 0.4 & <0.1 & <0.1 & 1.5 & <0.1 &0.3&80.0&<0.1 &<0.1&<0.1\\
		\bottomrule
	\end{tabular}
	\label{fig:quantitativeComparison}
	\vskip -0.2in
\end{table*}

\begin{figure}[t!]
	\begin{center}
		\iftoggle{arxiv}{
		\includegraphics[width=0.7\linewidth]{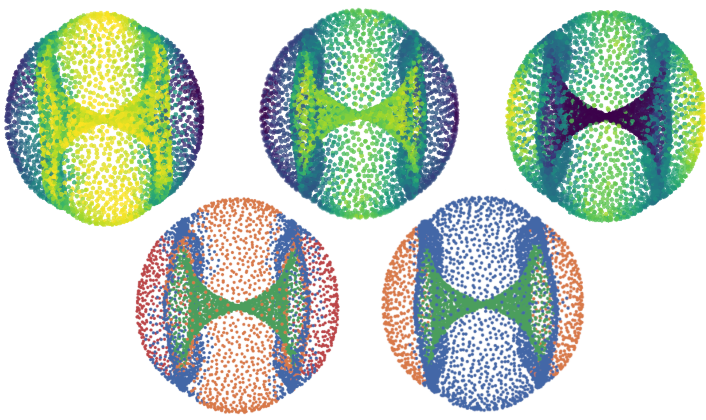}	
		}{
		\includegraphics[width=\linewidth]{figs/Cyclo8Heatmap5plots.png}
		}
		\caption{\figuretitle{\TOPF{} on a high-dimensional point cloud.}
			We used \TOPF{} on $6500$ points sampled from the $24$-dimensional conformation space of cyclooctane \cite{Martin2011}.
			We show the \textsc{\lowercase{ISOMAP}} projection from $24$ dimensions.
			\emph{Top:} Three features automatically selected by \TOPF{}.
			\emph{Bottom:} \TOPF{}-Clustering for 3 and 4 clusters.
			\TOPF{} can correctly cluster similar points according to their topological function.
		Four clusters, \TOPF{} can even identify the \emph{anomalous} points (blue) violating the manifold structure.
		}
		\label{fig:cyclo8}
	\end{center}
	\vskip -0.3 in
\end{figure}
\paragraph{Feature Generation}
In \Cref{fig:QualitativeExperiments}, we show qualitatively that \TOPF{} constructs meaningful topological features on data sets from Biology and Physics, and synthetic data, corresponding to for example rings and pockets in proteins or trajectories around different attractors in dynamical systems, (for individual heatmaps  see \Cref{fig:OldProtein}).
In \Cref{fig:cyclo8}, we show that \TOPF{} works for the \emph{high-dimensional} conformation space of cyclooctane.
\paragraph{Robustness Against Noise, Downsampling, and Sampling Heterogeneity}
\begin{figure}[tb!]
	\begin{center}
		\iftoggle{arxiv}{
		\includegraphics[width=0.8\linewidth]{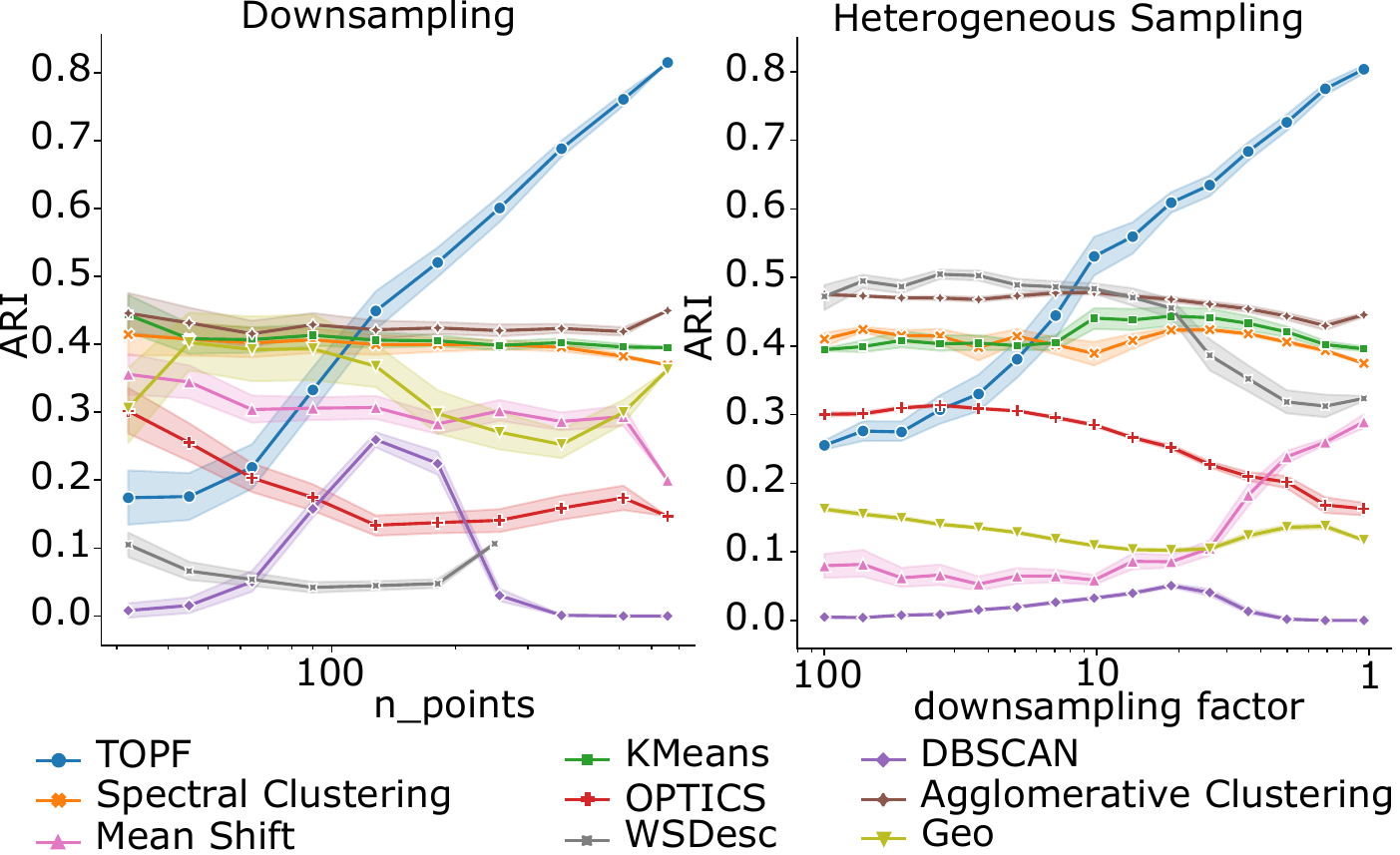}
		}{
		\includegraphics[width=\linewidth]{figs/4spherescombinedsampling.pdf}
		}
		\caption{\figuretitle{Performance of \TOPF{} Clustering under downsampling and heterogeneous sampling on \texttt{4spheres}.}
			\emph{Left} We perform random downsampling to test the performance of \TOPF{} on sparse point clouds. \TOPF{} maintains strong performance until $\sim 100$ points, which proves that \emph{\TOPF{} performs well on very sparse point clouds}. For experiments on all datasets of \TCBS{} see \Cref{fig:Downsampling}.
			\emph{Right:} We investigate the performance under heterogeneous sampling. We divide the point cloud into halves cutting up topological structures and downsample one of the halves with a factor of up to $100$.
			The experiments show that \TOPF{} still outperforms the other methods under a sampling irregularity of $1:10$, indicating a strong robustness against heterogeneous sampling.
			For experiments on more datasets of \TCBS{} see \Cref{fig:Skewedsampling}. 
		}
		\label{fig:Downsampling4spheres}
	\end{center}
	\vskip -0.2in
\end{figure}
We have evaluated the robustness of \TOPF{} against Gaussian noise on the dataset introduced in \citep{Grande:2023} and compared the results against \TPCC{}, Spectral Clustering, Graph Spectral Clustering on the graph constructed by \TPCC{}, and against $k$-means in \cref{fig:Robustness}.
We have also analysed the robustness of \TOPF{} against the addition of outliers in \cref{fig:Robustnessoutliers}.
We see that \TOPF{} performs well in both cases, underlining our claim of robustness.
Finally, in \Cref{fig:Downsampling4spheres}, we show that \TOPF{} performs well under \emph{sampling heterogeneity} and under \emph{downsampling to sparse point clouds}, while comparing \TOPF{} to the other baselines.
We posit that this is due to the robustness of $\alpha$-complexes, the chosen simplicial weights, and the general robustness of features extracted by persistent homology.
In our eyes, the above result showcases the strength of our approach incorporating ideas from topological data analysis and differential geometry, given that \emph{heterogeneous sampling} and \emph{sparse point clouds} constitute a major challenge for classical geometry learning.
\paragraph{Embedding Space of Variational Autoencoders and High-dimensional spaces}
Variational autoencoders (\VAE) are unsupervised neural networks that learn to extract a low-dimensional embedding of a high-dimensional data set.
We have trained $2$ \VAE{}s on $518$ image patches which were sampled along a topological structure on a larger image with a latent space dimension of $3$ and of $16$.
We show that running \TOPF{} on the two latent spaces as well as directly on the 8748(!)-dimensional base space can recover the \emph{topological structure} underlying the capturing process of the training set.
We visualise the feature vectors on the latent space in \Cref{fig:QualitativeExperiments} (e) and \Cref{fig:figcowhigh} and provide additional details in \Cref{fig:figcow} and \Cref{app:experiments}.

In \Cref{fig:cyclo8}, we use \TOPF{} on a high-dimensional input point cloud representing the conformation space of cyclooctane \cite{Martin2011}.
We see in the \textsc{\lowercase{ISOMAP}} projections of \Cref{fig:cyclo8} that \TOPF{} extracts reasonable features representing the topology of the conformation space.
This shows the feasability of \TOPF{} on very high-dimensional spaces.
Note that \TOPF{} even recovers the anomalous points violating the manifold hypothesis, as done in \cite{Stolz2020}.
\begin{figure}[tb!]
	\begin{center}
		\iftoggle{arxiv}{
			\includegraphics[width=\linewidth]{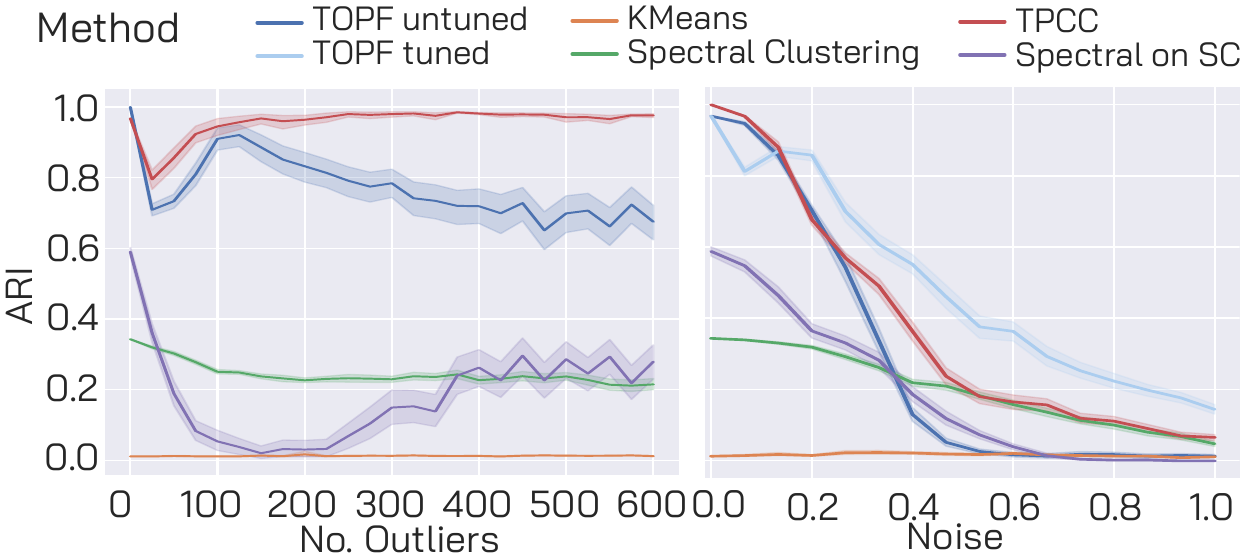}
		}{
			\includegraphics[width=0.8\linewidth]{figs/Robustness_BothICML.pdf}
		}
		\caption{\figuretitle{Performance of \TOPF{} Clustering with noise/outliers on \texttt{SphereinCircle}, 95\% \textsc{\lowercase{CI}}.}
			The radii are $3$ and $1$, the mean distance to the closest neighbour is $0.22$.
			\emph{Left:} We add i.i.d.\ Gaussian noise to every point with standard deviation indicated by the \texttt{noise} parameter.
			Even when compared with \TPCC{} on a data set specifically crafted for \TPCC{}, \TOPF{} requires significantly less information and delivers almost equal performance.
			Tuned for datasets with a high noise, the \TOPF{} outperform \TPCC{}. % and drastically outperform all classical clustering algorithms.
			\emph{Right:}
			We add outliers with same standard deviation as the point cloud. We then measure \ARI{} restricted to original points. Compared with \TPCC{} on a data set specifically crafted for \TPCC{}, \TOPF{} requires significantly less information and delivers matching to superior performance.%, outperforming classical clustering algorithms.
		}
		\label{fig:Robustnessoutliers}
		\label{fig:Robustness}
	\end{center}
	\vskip -0.2in
\end{figure}
\section{Discussion}

\label{sec:conclusion}
\paragraph{Limitations}

\TOPF{} can\,---\,by design\,---\,only produce meaningful output on point clouds with a \emph{topological structure} quantifiable by persistent homology.
In practice it is thus desirable to combine \TOPF{} with some geometric or other point-level feature extractor.
As \TOPF{} relies on the computation of persistent homology, its runtime increases on very large point clouds, especially in higher dimensions where $\alpha$-filtrations are computationally infeasible.
However, subsampling, either randomly or using landmarks, usually preserves relevant topological features while improving run time \citep{Perea2020}.
Finally, selection of the relevant features is a very hard problem. While our proposed heuristics work well across a variety of domains and application scenarios, only domain- and problem-specific knowledge makes correct feature selection feasible.

\begin{table}
	\centering
	\scriptsize
	\caption{
		\figuretitle{Ablation study for step 3.}
	}
	\begin{tabular}{lcc}
		Dataset&\ARI{} \TOPF{} & \makecell{\ARI{} \TOPF{}\\ without harmonic projection}\\
		\toprule
		\texttt{4spheres} & \textbf{0.81} & 0.08\\
		\texttt{Ellipses} & \textbf{0.95} & 0.53\\
		\texttt{4$\bigcirc$+Grid} & \textbf{0.70} & 0.11\\
		\texttt{Halved $\bigcirc$}& \textbf{0.71}&0.15\\\midrule
		\texttt{2Spheres2$\bigcirc$} & 0.94 & \textbf{0.95}\\
		\texttt{Spherein$\bigcirc$} & 0.97 &\textbf{1.00}\\
		\texttt{Spaceship} & \textbf{0.92} &0.49\\
		\midrule
		mean & \textbf{0.86} & 0.47\\
		\bottomrule
	\end{tabular}
	
	\label{tab:ablation}
\end{table}
\paragraph{Future Work}

The integration of higher-order \TOPF{} features into \textsc{\lowercase{ML}} pipelines that require point-level features potentially leads to many new interesting insights across the domains of biology, drug design, graph learning and computer vision.
Furthermore, efficient computation of simplicial weights leading to the provably most faithful topological point features is an exciting open problem.
\paragraph{Conclusion}

We introduced point-level features \TOPF{} founded on algebraic topology relating global structural features to local information.
We gave theoretical guarantees for the correctness of their construction and evaluated them quantitatively and qualitatively on synthetic and real-world data sets.
Finally, we introduced the novel topological clustering benchmark suite and showed that clustering using \TOPF{} outperforms other available clustering methods and features extractors.
In particular, we showed that \TOPF{} performs well both on very sparse datasets and on datasets under heterogeneous sampling.
\iftoggle{arxiv}{
\paragraph{Acknowledgments}
\textsc{\lowercase{VPG}} acknowledges funding by the German Research Council (\textsc{\lowercase{DFG}}) within Research Training Group 2236 %(\textsc{\lowercase{U}}n\textsc{\lowercase{RAV}}e\textsc{\lowercase{L}}).

\textsc{\lowercase{MTS}} acknowledges partial funding by the Ministry of Culture and Science (\textsc{\lowercase{MKW}}) of the German State of North Rhine-Westphalia (\enquote{\textsc{\lowercase{NRW}} Rückkehrprogramm}) and the European Union (\textsc{\lowercase{ERC}}, \textsc{\lowercase{HIGH-HOP}}e\textsc{\lowercase{S}}, 101039827). Views and opinions expressed are however those of the author(s) only and do not necessarily reflect those of the European Union or the European Research Council Executive Agency. Neither the European Union nor the granting authority can be held responsible for them.
}{}

	\printbibliography
	\appendix
	% !TeX spellcheck = en_GB
% !TeX root = TOPFICML.tex
\section{Extended Background}
\label{app:extendedbackground}
\paragraph{A brief history of topology and machine learning}
Algebraic topology is a discipline of Mathematics dating back roughly to the late 19\textsuperscript{th} century \citep{poincare1895analysis}.
Starting with Henri Poincar{\'e} and continuing in the early 20\textsuperscript{th} century, the mathematical community became interested in developing a framework to capture the global shapes of manifolds and topological spaces in concise algebraic terms.
This development was partly made possible by the push towards a formalisation of mathematics and analysis, in particular, which took place inside the mathematical community in the 1800's and early 1900's (e.g.\ \cite{dedekind1888sollen, hilbert1899grundlagen, Hausdorff1908}).
The axiomatisation of analysis in the early 20\textsuperscript{th} century is an important result of this process.
~complexes. 
Over the course of the last 100 years, branching into many sub-areas like low-dimensional topology, differential topology, K-theory or homotopy theory \citep{atiyah2018k, hu1959homotopy}, algebraic topology has resolved many of the important questions and provides a comprehensive tool-box for the study of topological spaces.
These achievements  were tied to an abstraction and generalisation of concepts: topological spaces turned into spectra, diffeomorphism to homotopy equvialences and later weak equivalences, and Topologists turned to category theory \citep{eilenberg1945general}, model categories \citep{Bousfield1975} and recently $\infty$-categories \citep{lurie2006stable} as the language of choice. 

The 21\textsuperscript{st} century saw the advent and rise of topological data analysis (\TDA, \citep{Bubenik2015, chazal2021introduction}).
In short, mathematicians realised that the same notions of shape and topology that their predecessors carefully defined a century earlier were now characterising the difference between healthy and unhealthy tissue, between normal and abnormal behaviour protein behaviour, or more general between different categories in their complex data sets.

Topological data analysis is closely connected to distance- and density-based methods for data analysis like \textsc{dbscan} or tools from manifold learning like \textsc{isomap} \cite{tenenbaum2000global} or \textsc{umap} \cite{mcinnes2018umap}.
Persistent homology captures the shape of a point cloud across all scales and can be captured in persistence barcodes \cite{carlsson2004persistence}, persistent landscapes \cite{Bubenik2015} or persistence images \cite{adams2017persistence}.
In multiparameter persistent homology, an additional filtration parameter like density is included, allowing for stronger representations like graphcodes \cite{kerber2024graphcode} or \textsc{gril} \cite{xin2023gril}.
The persistent homology transform computes persistent homology over a height filtration across different directions in the data set, further uncovering geometric properties of the dataset \cite{Turner2014}.
The mapper graph is a tool to visualise high dimensional point clouds with graphs using the (persistent) homology of slices according to some generalised heigh function \cite{Singh:2007}.

\paragraph{Related Work}
\label{app:RelatedWork}

The intersection of topological data analysis, topological signal processing and geometry processing has many interesting related developments in the past few years.
On the side of homology and \TDA{}, the authors in \cite{DeSilva2009persistent} and \cite{Perea2020} use harmonic \emph{co}homology representatives to reparametrise point clouds based on circular coordinates.
This implicitly assumes that the underlying structure of the point cloud is amenable to such a characterization.
Although circular coordinates are orthogonal to the core goal of \TOPF{}, the approaches share many key ideas and insights.
In \cite{Basu2022harmonic,Gurnari2023probing}, the authors develop and use harmonic persistent homology and provide a way to pool features to the point-level.
However, their focus is not on providing robust topological point features and their approach includes no tunable homology feature selection across dimensions, no support for weighted simplicial complexes, and they only construct the simplicial complex at birth. 
In their paper on topological mode analysis, \cite{Chazal2013} use persistent homology to cluster point clouds.
However, they only consider $0$-dimensional homology to base the clustering on densities and there is no clear way to generalise this to higher dimensions.

On the more geometric-centred side, \cite{Ebli2019} already provide a notion of harmonic clustering on simplices, \cite{Chen2021,chen2021helmholtzian} analyse the notion of geometry and topology encoded in the Hodge Laplacian and its relation to homology decompositions, \cite{schaub2020random} study the normalised and weighted Hodge Laplacian in the context of random walks, and \cite{Grande:2023} use the harmonic space of the Hodge Laplacians to cluster point clouds respecting topology.
Finally, a persistent variant of the Hodge Laplacian is used to study filtrations of simplicial complexes \citep{memoli2022persistent}.

There are other constructions of isometry-invariant local shape descriptors, as for example proposed in \cite{memoli2011gromov}.
However, their aim is not relate the global topology back to the local features.

 In \cite{moor2020topological}, the authors basically set up a topological loss functions that ensures that topological features of the original point cloud are preserved in the latent space of the autoencoder. This is different from the approach \TOPF{} takes, as \TOPF{} tries to detect and extract topological features. Similar things can be said about topological node-2-vec \cite{hiraoka2024topological} and some experiments in \cite{carriere2021optimizing} or some experiments of \cite{carriere2024diffeomorphic}: They all try to come up with ways to preserve global topological structure while embedding a point cloud, while we on the other hand construct the embedding based on the relationship between the points and the global topological structure.

On graphs, \cite{arafat2024witnesses} construct topological representations of graphs to study adversarial graph learning.
They construct a local topological features for nodes and global topological features for the entire graph.
In contrast to our approach, they do not relate the global topological features back to the local level, but rather consider the local topology on the point level and the single global topological summary separately.
Furthermore when constructing their features, they use vision transformers on the persistence images of their witness filtrations, which does not allow for the same amount of interpretability as \TOPF{} does.

In \cite{Grande:2023}, the authors have introduced \TPCC{}, the first method to cluster a point cloud based on the higher-order topological features encoded in the data set.
However, \TPCC{} is \textbf{(i)} computationally expensive due to extensive eigenvector computations,
\textbf{(ii)} depending on high-dimensional subspace clustering algorithms, which are prone to instabilities and errors,
\textbf{(iii)} sensitive to the correct choice of hyperparameters,
\textbf{(iv)} requiring the topological true features and noise to occur in different steps of the simplicial filtration, 
and it \textbf{(v)} solely focussed on clustering the points rather than extracting relevant node-level features.
This paper solves all the above by completely revamping the \TPCC{} pipeline, introducing several new ideas from applied algebraic topology and differential geometry.
The core insight is: When you have the time to compute persistent homology with generators on a data set, you  get the topological node features with similar computational effort.

\section{Theoretical Considerations}

\label{app:TheoreticalConsiderations}

\subsection{More details on \VR{} and $\alpha$-filtrations}
Vietoris--Rips complexes are easy to define, approximate the topological properties of a point cloud across all scales and computationally easy to implement.
However for moderately large $r$, the associated \VR~complex contains a large number of simplices\,---\,up to $\binom{|X|}{n+1}$ $n$-simplices for large enough $r$\,---\,leading to poor computational performance for any downstream task on some large point clouds.
One way to see this is the following: After adding the first edge that connects two components or the final simplex that fills a hole in the simplicial complex the \VR~complex keeps adding more and more simplices in the same area that keep the topology unchanged. 
One way to mitigate this problem is to pre-compute a set of simplices that are able to express the entire topology of the point cloud.
For a point cloud $X\subset\R^n$, the $\alpha$-filtration consists of the intersection of the simplicial complexes of the \VR{}~filtration on $X$ with the (higher-dimensional) Delaunay triangulation of $X$ in $\R$.
Due to algorithmic reasons, the filtration value of a simplex is then related to the radius of the circumscribed sphere instead of the maximum pair-wise distance of vertices, where the filtration value $\alpha^X(\sigma)$ of a simplex $\sigma$ is given by the minimum $\varepsilon$ such that $\sigma$ is contained in $\alpha_\varepsilon(X)$, i.e.\ $\alpha^X(\sigma)\coloneqq \inf \{\varepsilon\in \R:\sigma \in \alpha_\varepsilon(V)\}$.
This reduces the number of required simplices across all dimensions to $\smash{O(|X|^{\lceil n/2\rceil})}$.
However, the Delaunay triangulation becomes computationally infeasible for larger $n$.
\begin{definition}[$n$-dimensional Delaunay triangulation]
	Given a set of vertices $V\subset\R^n$, a Delaunay triangulation $DT(V)$ is a triangulation of $V$ such that for any $n$-simplex $\sigma_n\in DT(V)$ the interior of the circum-hypersphere of $\sigma_n$ contains no point of $DT(V)$.
	A triangulation of $V$ is a \tSC{} $\SC$ with vertex set $V$ such that its geometric realisation covers the convex hull  of $V$ $\operatorname{hull} (V)=|\SC|$ and we have for any two simplices $\sigma$, $\sigma'$ that the intersection of geometric realisations $|\sigma|\cap|\sigma'|$ is either  empty or the geometric realisation $|\hat{\sigma}|$ of a common sub-simplex $\hat{\sigma}\subset \sigma,\sigma'$.
\end{definition}
If $V$ is in general position, the Delaunay triangulation is unique and guaranteed to exist \citep{delaunay1934sphere}.
We will now first introduce a slightly simpler version of the $\alpha$-filtration, the $\alpha^*$-filtration.
\begin{definition}[$\alpha^*$-complex of a point cloud]
	\label{def:alphastar}
	Given a finite point cloud $X$ in real space $\R^n$, the $\alpha^*$-complex $\alpha^*_\varepsilon(X)$ is the subset of the $n$-dimensional Delaunay triangulation $DT(X)$ consisting of all $k$-simplices $\sigma_k\in DT(X)$ with a radius $r$ of its circumscribed $(k-1)$-sphere with $r\le \varepsilon$ for all $k\le n$.
\end{definition}
The $\alpha^*$ and $\alpha$-filtration share the same set of simplices and agree on the filtration values of the simplices on all top-dimensional simplices and all other simplices which are \emph{Gabriel} simplices.
\begin{definition}[Gabriel Simplex]
	Given a set of vertices $V\subset \R^n$ and its Delaunay triangulation $DT(V)$. A simplex $\sigma\in DT(V)$ is called \emph{Gabriel} if there exists no point $v\in V$ such that $v$ is contained in the interior of the circumsphere of $\sigma$.
\end{definition}

\begin{definition}[$\alpha$-complex of a point cloud, \citep{kerber20133d, edelsbrunner2011alpha}]
	\label{def:alpha}
	Given a finite point cloud $X$ in real space $\R^n$, the $\alpha$-complex $\alpha_\varepsilon(X)$ is the $\alpha^*$-complex of $V$, except that the filtration value $\alpha(\sigma)$ of all non-Gabriel simplices $\sigma$ with associated points $X$ in the interior of the circumsphere of $\sigma$ is given by $\min_{x\in X}\alpha(\sigma\cup\{x\})$.
\end{definition}
We note that due to the definition of the Delaunay triangulation, all $n$-simplices are Gabriel simplices and hence this is well-defined.
This is an equivalent formulation of the original definition, which was for example shown in \cite{kerber20133d}.
We chose to go with the above formulation, as it is the form used in implementations of $\alpha$-filtrations.

We note two things: \emph{(i)} The $\alpha^*$-filtration and $\alpha$-filtration have fewer simplices than the \VR{} filtration and \emph{(ii)} the filtration values of individual simplices differ between the $\alpha^*$, the $\alpha$ and the \VR{} filtration.
In particular, the order of the filtration values can be different across the two types of filtration.
For the $1$-skeleton, the $\alpha^*$ and \VR{} filtrations are equivalent:
The \VR{} filtration value of an edge is its length $l$, whereas its $\alpha^*$-filtration value is the radius $r=l/2$ of the associated $0$-sphere consisting only of the two vertices.
\subsection{Differential forms and a continuous analogue of \TOPF{}}
\label{app:differentialforms}
\vincent{Add discussion of Hodge Laplacian vs Combinatorial Laplacian paper/discussion.}
In this section, we will discuss how a continuous analogue of \TOPF{} on Riemannian manifolds would look like.
Simply considering the simplicial homology of the manifold, as we did in the case of a discrete simplicial complex, is however not sufficient:
Harmonic cycles and cocycles don't exist in the simplicial chain complex of the manifold.
This is because the space of chains in simplicial homology on manifolds is infinite-dimensional.
Thus, we cannot canonically identify the space of chains with its dual, the space of cochains.
From an intuitive point of view, a harmonic chain would need to take infinitesimal small values on every simplex  which is not allowed.
Rather, the right language to formalise harmonic representatives is \emph{Differential Geometry, Hodge Theory, and harmonic forms}.
Dodziuk proved in \cite{dodziuk1976finite} that the discrete Hodge Laplacian as used in this paper converges to the Hodge Laplacian on differential forms which are arbitrary-dimensional integrands over manifolds.

There is a beautiful connection between differential geometry and algebraic topology:
A theorem by de Rham states that given a Riemannian manifold $M$, the real-valued homology $H_k(M;\R)$ with coefficients in $\R$ is isomorphic to the de Rham cohomology $H^k_\text{dR}(M;\R)$ on differential forms on $M$ for arbitrary $k\in\Z_{\ge 0}$, some form of cohomology defined via the chain complex of vector spaces of differential $k$-forms on $M$.
We can define an analogue of the discrete Hodge Laplacian on differential forms by the differential induced via exterior derivates $d$ and its adjoint $\sigma$:
\[
\Delta \coloneqq \delta d + d \delta
\]
The kernel $\ker \Delta$ of the Hodge Laplacian is called the space of harmonic forms.
Similar to the discrete case, the Hodge theorem now gives us a natural vector space isomorphism between the space of harmonic $k$-forms on $M$, $\mathcal{H}_k(M)$, and the $k$\textsuperscript{th} real-valued homology group 
\[
H_k(M;R)\cong \mathcal{H}_k(M).
\]
We can rephrase this as follows: the harmonic forms can be seen as unique and natural homology representatives.

Let us now get some intuition for what this means for continuous \TOPF{} features:
In dimension $0$, $0$-forms are functions $f\colon M\to R$.
In this case, we can simply take their value $f(x)$ at $x\in M$ as the corresponding point feature at $x$.
In dimension $1$, there is a correspondence between $1$-forms and vector fields on $M$. In this case, we can take the norm of the vector field that corresponds to a given harmonic form at a point $x$ via the map $x\mapsto|v(x)|$.

In general in dimension $k$, this is more complicated.
Luckily, as we are interested in point-features, we can do all computations point-wise.
We consider a point $x\in M$ and a harmonic form $\omega$.
At point $x$, $\omega$ determines an element in the exterior algebra on the dual of the tangent space of $M$ at $x$, i.e\ $\omega_x\in\bigwedge^kT^{*}_xM$ and we need to define a norm on this space.
An orthonormal basis on $T^{*} _x M$, $e_1(x),\dots,e_n(x)$ gives rise to a basis consisting of elements $e_{i_1}(x)\wedge ...\wedge e_{i_k}(x)$ of the exterior algebra. We can now evaluate $\omega_x$ on the elements of this basis $\omega_x (e_{i,1}(x)\wedge ...\wedge e_{i,k}(x))$ and take the square root of the sum of squares of the individual results

\[
\texttt{TOPF}_\omega\colon x\mapsto \sqrt{\sum_{1\le i_1<\dots<i_k\le n}\omega_x\left(e_{i_1}(x)\wedge \dots \wedge e_{i_k}(x)\right)^2}.
\]
As one can show, this result is independent of the choice of \textsc{onb} of $T^*_xM$ and gives a point-wise norm of the differential forms.

In the discrete case discussed in this paper, however, we don't evaluate on all basis elements as described above, but rather on all $k$-simplices adjacent to $x$.
These $k$-simplices can thus be seen as samples from the exterior algebra of $x$.
This sampling induces some differences to the continuous case, which we address with the post-processing procedure as described in the paper.
%\subsection{Theoretical Guarantees}

\subsection{Fixing real lifts of homology generators in $\mathbb{Z}/3\mathbb{Z}$-coefficients}
\label{app:fixcycles}
\begin{figure}[ht!]
	\begin{center}
		\includegraphics[width=0.4\linewidth]{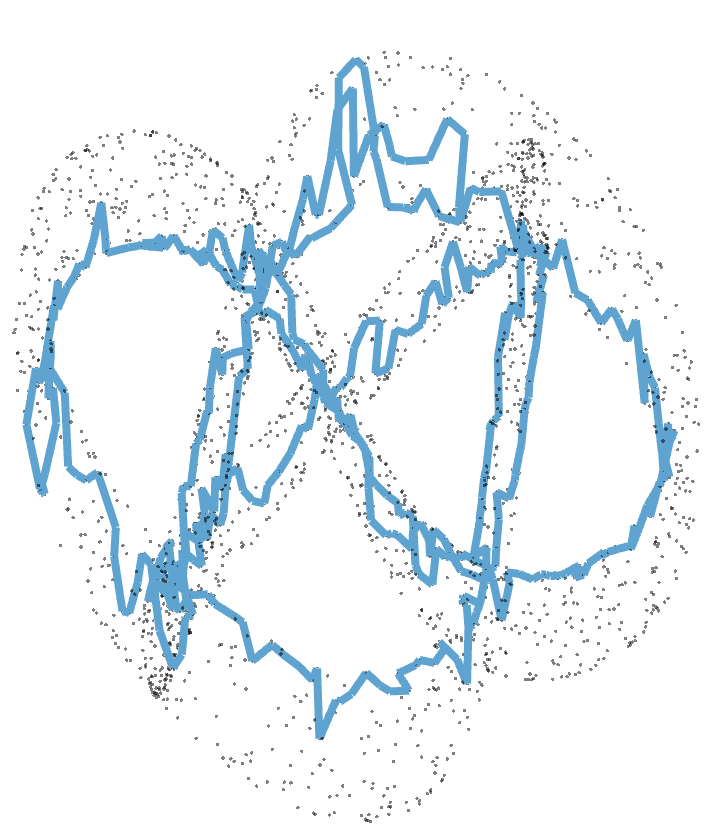}
		\includegraphics[width=0.4\linewidth]{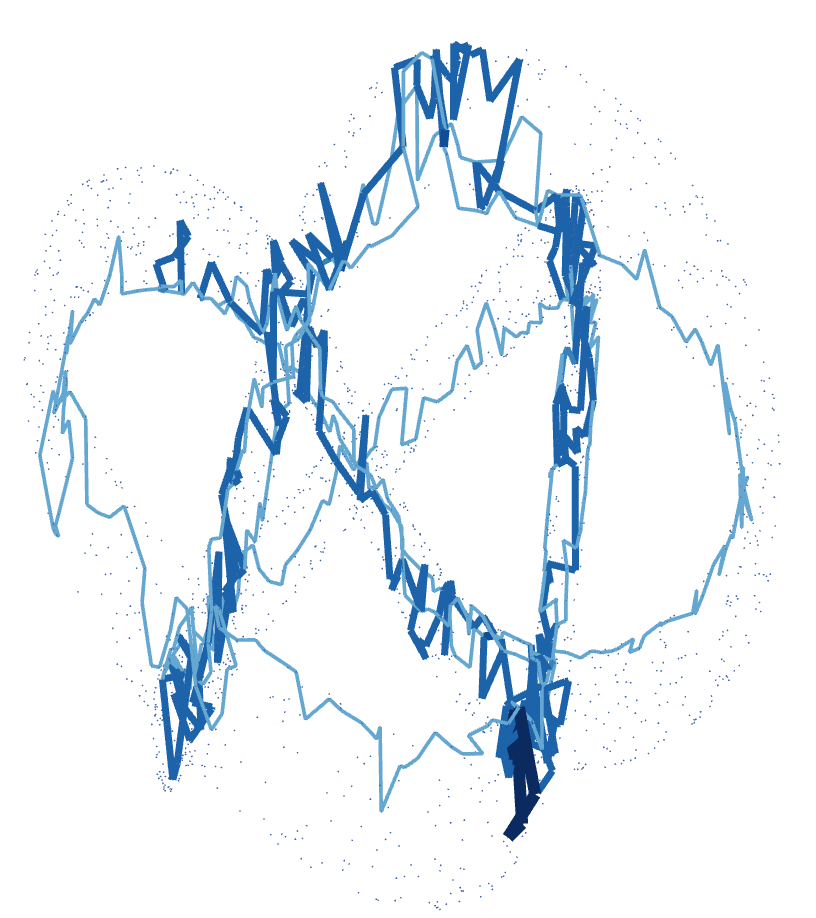}
		\caption{\figuretitle{$\ZthreeZ$-homology generator with \enquote{faulty} lift to $\R$-coefficients.}
			\emph{Left:} Initial lift produced in $\ZthreeZ$-coefficients.
			\emph{Right:} Fixed lift with \emph{light blue} edges representing the original lift with coefficients in $\{-1,1\}$ and \emph{dark blue} edges representing the fix with coefficients in $3\mathbb{Z}$.
			The underlying data is the trefoil knot embedded in $\R^3$, as generated by \cite{Perea2023}.
		}
		\label{fig:faultylifts}
	\end{center}
\end{figure}
In step \textbf{3.} of the \TOPF{} algorithm in \Cref{sec:algorithm}, we attempt to lift a homology generator with simplices $\sigma_j\in \Sigma_k$ and coefficients $c_j$ in $\mathbb{Z}/3\mathbb{Z}$-coefficients to a homology generator in $\mathbb{R}$-coefficients using the embedding
\[
e_k^i(\sigma)\coloneq\begin{cases}	
	\iota(c_j)&\exists \hat{\sigma}_j\in \Sigma_k^i:\sigma = \hat{\sigma}_j\\
	0 &\text{else}
\end{cases}.
\]
where $\iota\colon\Z/3\Z$ is the map induced by the canonical inclusion of $\{-1,0,1\}\hookrightarrow\R$.
Thus, we know that $e_k^i$ is a cycle modulo $3$: $\bound_{k-1}e_k^i\equiv 0\mod 3$.
A-prior however, there is no reason why this equality should hold in $\R$ as well.
As it turns out, in virtually all examples in practice, we already have $e_k^i\in \ker \bound_{k-1}$ which is very convenient.
The same phenomenon occurs for cocycles as well, see \cite{DeSilva2009persistent}.
We will now assume to be in a rare case of $e_k^i\not\in \ker \bound_{k-1}$.
We can now attempt to fix this using a linear integer program finding $y\in \mathbb{Z}^{|\SC_k|}$ such that
\[
3\bound_{k-1} y= \bound_{k-1}e_k^i.
\]
And then we set $\text{fix}(e_k^i)\coloneqq e_k^i-3y$.
If this linear program is not feasible, $e_k^i$ has no integer/real lift and \TOPF{} will skip the associated topological feature.
In \Cref{fig:faultylifts}, we give a visual example of a faulty lift and a fix obtained by the above procedure.
\begin{comment}
	From the persistent homology computation of the first step, we also obtain a generator of the feature $f_k^i$, consisting of a list $\Sigma_k^i$ of simplices $\smash{\hat{\sigma}_j\in \SC^{b_{i,k}}_k}$ and coefficients $c_j\in\Z/3\Z$.
	We need to turn this formal sum of simplices with $\Z/3\Z$-coefficients into a vector in the real vector space $\smash{\R[\SC_k^{t_{i,k}}(X)]}$:
	Let $\iota\colon\Z/3\Z$ be the map induced by the canonical inclusion of $\{-1,0,1\}\hookrightarrow\R$.
	We can now define an indicator vector $\smash{e_k^i\in\R[\SC_k^{t_{i,k}}(X)]}$ associated to the feature $f_k^i$ (Cf.\ \cite{DeSilva2009persistent}).
	\[
	e_k^i(\sigma)\coloneq\begin{cases}
		\iota(c_j)&\exists \hat{\sigma}_j\in \Sigma_k^i:\sigma = \hat{\sigma}_j\\
		0 &\text{else}
	\end{cases}.
	\]
\end{comment}
\section{Theoretical guarantees}

\label{sec:theoreticalguarantee}
In this section, we want to investigate possible theoretical guarantees for \TOPF{} computed on idealised datasets.
Instead of directly working with the $\alpha$-complexes, we will first prove guarantees for the closely related construction of $\alpha^*$-complexes.
$\alpha^*$ filtrations (\Cref{def:alphastar}) are closely related to $\alpha$-filtrations (\Cref{def:alpha}) used in the implementation and agree on the filtration values on most of the simplices, have however nicer properties.
Finally, we give conditions on when this result applies to the $\alpha$-filtrations, which is relevant to the actual implementation of \TOPF.
\begin{restatable}[Topological Point Features of Spheres]{theorem}{TheoreticalGuarantee}
	\label{thm:TOPFsphere}
	Let $X$ consist of at least $(n+2)$ points (denoted by $S$) sampled uniformly at random from a unit $n$-sphere in $\R^{n+1}$ and an arbitrary number of points with distance of at least $2$ to $S$.
	When we now consider the $\alpha^*$-filtration on this point cloud, with probability $1$ we have that \textbf{(i)} there exists an $n$-th persistent homology class generated by the $n$-simplices on the convex hull of $S$, 
	%\textbf{(ii)} the associated unweighted harmonic homology representative takes values in $\{0,\pm 1\}$ where the $2$-simplices on the boundary of the convex hull are assigned a value of $\pm1$,
	and \textbf{(ii)} the support of the associated topological point feature (\TOPF) $\mathcal{F}^*_n$ is precisely $S$: $\operatorname{supp}(\mathcal{F}^*_n)=S$.
	\textbf{(iii)} The same holds true for point clouds sampled from multiple $n_i$-spheres if the above conditions are met on each individual sphere.
	\textbf{(iv)} If there are no other points contained in or on the circumspheres of the  $n$-simplices of the convex hull of $S$, then the same above holds true for \TOPF{} computed using the $\alpha$-filtration.
	\textbf{(v)} More generally, for any contractible subset $M$ of the $(n+1)$-simplices such that all $n$-simplices on the boundary fulfill the above condition, this holds when using the $\alpha$-filtration, the discussed $n$-simplices (i) and the subset of points contained in $M$ (ii).
\end{restatable}
We will give a proof of this after the following remark.
\begin{remark}
	% Something about how this is idealised and how the proof works?
	The key idea of the proof is to use that $\alpha^*$-filtrations and partially $\alpha$-filtrations assign the filtration value based on the radius of the circumsphere of the $k$-simplex.
	Because all points in $S$ lie on the same $n$-sphere, with probability $1$ we can write down the filtration value of the $(n+1)$-simplices explicitly.
	This is of course a very idealised setting.
	However in practice, datasets with topological structure consist in a majority of cases of points sampled with noise from deformed $n$-spheres.
	The theorem thus guarantees that \TOPF{} will recover these structural information in an idealised setting.
	Experimental evidence suggests that this holds under the addition of noise as well which is plausible as harmonic persistent homology is robust against some noise \citep{Basu2022harmonic}.
	For the $2D$ setting, condition \textbf{(iv)} is equivalent to assuming that there is no line through the centre of the circle such that one half plane does not contain any points of $S$.
\end{remark}
We will now give the proof of the theorem that guarantees that \TOPF{} works. We discuss the assumptions of the theorem in \Cref{comment:IdealisedSetting}.

\begin{proof}
	Assume that we are in the scenario of the theorem.
	Now because the $n$-volume of $(n-1)$-submanifolds is zero, we have that with probability $1$ the points of $S$ don't lie on a single $(n-1)$ sphere inside the $n$-sphere.
	Let us now look at the $\alpha$-filtration of the simplices in $S$:
	Recall that the filtration values of a $k$-simplex is given by the radius of the $(k-1)$-sphere determined by its vertices.
	Because all of the $(n+1)$-simplices $\sigma_{n+1}$ with vertices $V\subset S$ in $S$ lie on the same unit $n$-sphere $S_n$, they all share the filtration value of $\alpha(\sigma_{n+1}) = 1$.
	By the same argument as above, with probability $1$ there are no $(n+1)$ points in $S$ that lie on an \emph{unit} $(n-1)$-sphere.
	Thus all of the $n$-simplices $\sigma_n$ lie on $(n-1)$-spheres $S_n$ with a radius $r<1$ smaller than $1$ and hence have a filtration value $\alpha(\sigma_n)$ smaller than $1$.
	Let 
	\[
	b\coloneq \max \left(\left\{ \alpha(\sigma_n) : \sigma_n \subset \partial \operatorname{hull}(S)\right\}\right)
	\]
	be the maximum filtration value of an $n$-simplex on the boundary of the convex hull of $S$.
	Then, a linear combination $g$ of the $n$-simplices of the boundary of the convex hull of $S$ with coefficients in ${\pm1}$ is a generator of a persistent homology class with life time $(b,1)$ (this follows from the fact that $n$-spheres and their triangulations are orientable).
	This proves claim \textbf{(i)}.
	
	Because of the assumption that all points not contained in $S$ have a distance of at least $2$ to the points in $S$, all $(n+1)$-simplices $\sigma_{n+1}$ with vertices both in $S$ and its complement in $X$ will have a filtration value $\alpha(\sigma_{n+1})\ge 1$ of at least $1$.
	Recall that all $(n+1)$-simplices $\sigma_{n+1}\subset S$ with vertices inside $S$ have a filtration value of $\alpha(\sigma_{n+1})=1$. Thus the adjoint of the $n$-th boundary operator $\bound_n^\top$ is trivial on the homology generator $g$ for a simplicial complex constructed during $(b,1)$.
	Thus, we have that for the $n$-th Hodge Laplacian
	\[
	L_n g =\bound_{n-1}^\top\bound_{n-1}g+\bound_n\bound_{n}^\top g = 0+0 = 0
	\]
	and hence $g$ is a harmonic generator for the entire filtration range of $(b,1)$.
	Claim \textbf{(ii)} and \textbf{(iii)} then follow from the construction of the \TOPF{} values.
	
	For claim \textbf{(iv)}, it suffices to notice that the $n$-simplices $\sigma$ on the convex hull are precisely non-Gabriel simplices iff there exist a point $x\in S$ lying inside the circumsphere of $\sigma$.
	Because the assumption of part \textbf{(iv)} excludes this case, all $n$-simplices are Gabriel simplices and the $\alpha$-filtration values (\Cref{def:alpha}) and $\alpha^*$-filtration values (\Cref{def:alphastar}) agree on the $n$-simplices and $(n+1)$-simplices of the Delaunay triangulation.
	Because these are the only considered filtration values for the proof, the claim follows.

	Finally for part \textbf{(v)}, we see that the same argument now applies to the simplices on the boundary, consisting of $n$-simplices of $M$, consisting of $(n+1)$-simplices.
	Because $M$ is contractible, it is connected and has trivial homology. Thus, its boundary is homotopy equivalent to an $n$-sphere.
	Because in this case only the points $X'$ on the boundary of $M$ are contained in the generating $n$-simplices of the considered homology class, the associated \TOPF{} feature will have support on $X'$.
	This concludes the proof.
\end{proof}
\begin{remark}
	\label{comment:IdealisedSetting}
	The setting of the theorem is a very special setting.
	In general on data, harmonic chains have entries with absolute values somewhere in $[0,1]$.
	In the setting of the theorem however, the harmonic representative will take values in ${-1,0,1}$.
	We will briefly explain why this is the case:
	Consider a $k$-homology representative $r$ with simplex-values in $\{-1,0,1\}$.
	We consider the harmonic projection $h$ of $r$ into the harmonic space.
	Then, for a $k$-simplex $\sigma$, we have that $h(\sigma)=\pm 1$ iff $r(\sigma) = 1$ \textbf{and} $\sigma$ is not the face of any $(k+1)$-simplices.
	
	This can for example be seen by representing $h$ as the difference of $r$ and its gradient and curl parts $h=r-r_\text{grad}-r_\text{curl}$. Because $r$ is already a cycle, $\bound_kr=0$ and thus $r_\text{grad}=0$. Now, the curl part can be written as stemming from a signal on the the $k+1$-simplices, $r_\text{curl}=B_{k+1}x_\text{curl}$. However, because $\sigma$ is not the face of a $k+1$-simplex, $r_\text{curl}(\sigma)=0$ and thus $h(\sigma) = r(\sigma)$.
	
	The setting of the theorem precisely constructs a case where we have $k$-simplices and no $(k+1)$-simplices in the relevant part of the filtration.
	This is the case because in an $\alpha$-filtration, the filtration value of a $k$-simplex is the radius of its circumscribed $(k-1)$-sphere.
	All relevant $(k+1)$-simplices in the theorem lie on the $k$-unit sphere, and thus have filtration value $1$.
	This is however an idealised setting.
	Because we construct the simplicial complex to compute the harmonic representative somewhere in the middle (determined by the interpolation coefficient) between the birth and death time of the homology class, empirically the majority of the $k$- simplices of the $k$-homology generators are faces of $(k+1)$-simplices.
	In this case, the harmonic representative smooths out the feature.
\end{remark}
\section{Topological Clustering Benchmark Suite}
\label{app:BenchmarkSuite}
We introduce seven point clouds for topological point cloud clustering in the topological clustering benchmark suite (\TCBS).
Some of these point clouds have been adapted from \cite{Grande:2023}.
The ground truth and the point clouds are depicted in \Cref{fig:TCBS}.
The point clouds represent a mix between $0$-, $1$- and $2$-dimensional topological structures in noiseless and noisy settings in ambient $2$-dimensional and $3$-dimensional space.
Point clouds \texttt{Ellipses} and \texttt{spaceship} already incorporate heterogeneous sampling/densities.
The results of clustering according to \TOPF{} can be found in \Cref{fig:TCBSTOPF}.

We constructed the benchmark by sampling from topologically different shapes with varying sampling density in different ambient dimensions. For example for the point cloud \texttt{2Spheres2Circles}, we combined points sampled from two spheres and two circles. We then divided the point cloud into four ground truth clusters, one for each of the two spheres and two circles and assigned every point the cluster corresponding to the object it was sampled from. Thus, the ground truth labels of a point corresponds to the topological structure it was sampled from.
\iftoggle{arxiv}{}{
We provide functions to resample some of the point clouds with variable point numbers and provide a method to provide sampling heterogeneity as described in \Cref{fig:Skewedsampling}.
}
\begin{figure}[tb!]
	\begin{center}
		\resizebox{\textwidth}{!}{
			\includegraphics[height=5cm]{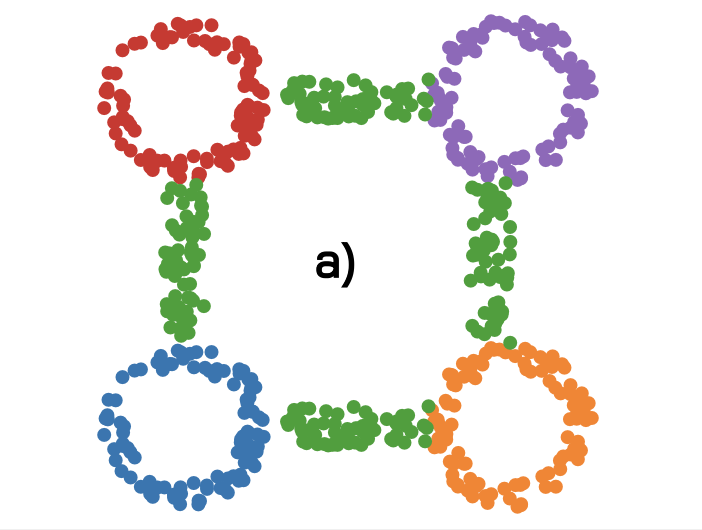}
			\includegraphics[height=5cm]{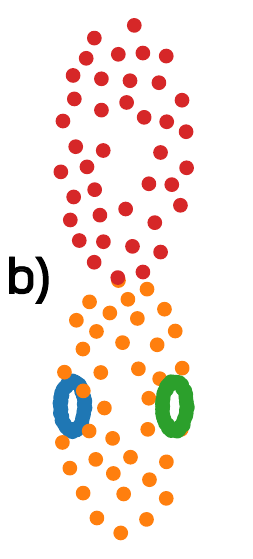}
			\includegraphics[height=5cm]{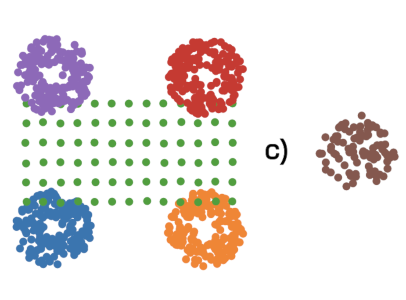}
			\includegraphics[height=5cm]{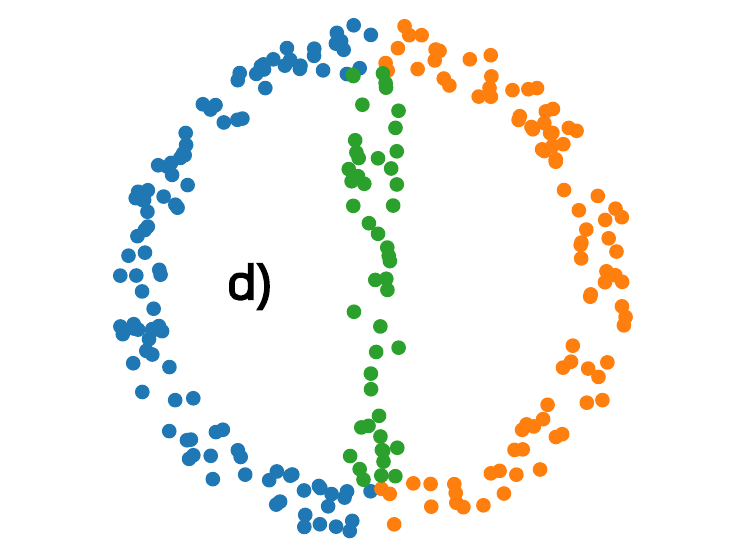}
		}
		\resizebox{\textwidth}{!}{
			\includegraphics[height=5cm]{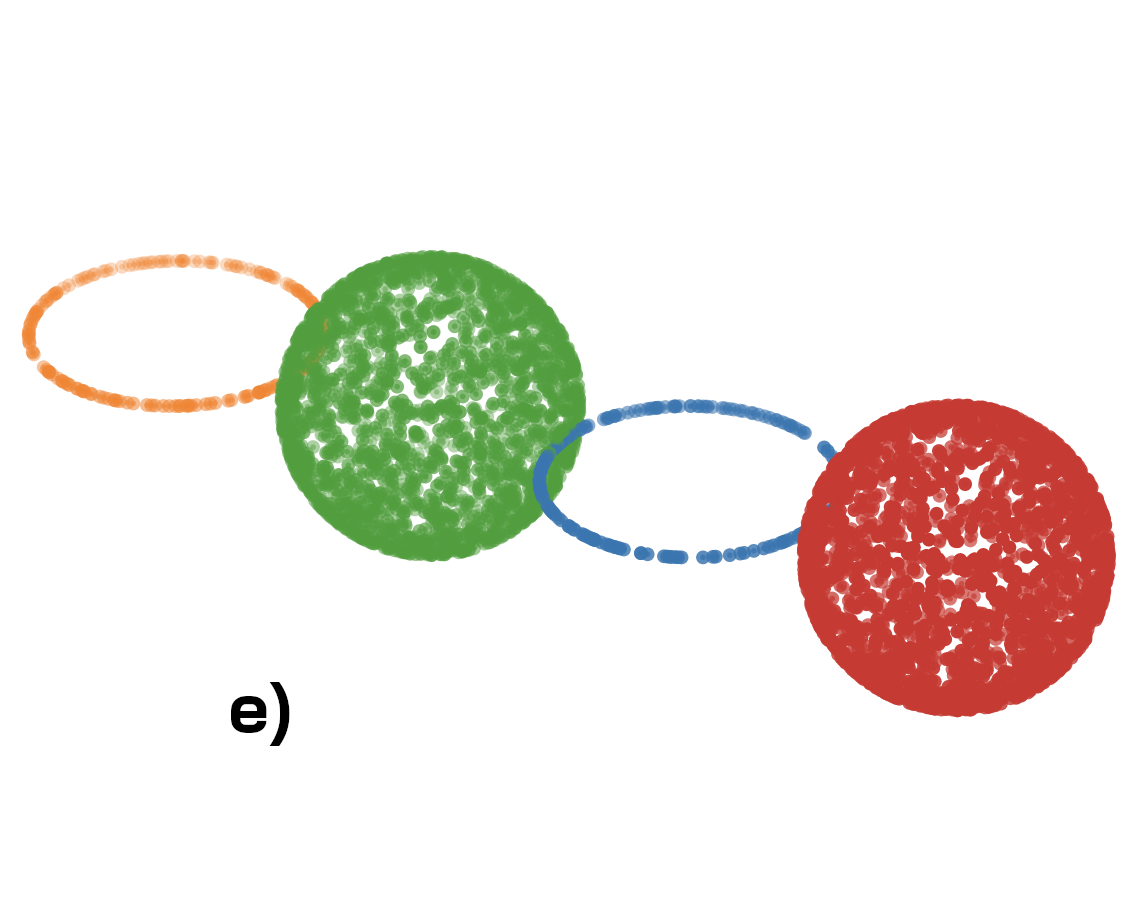}
			\includegraphics[height=5cm]{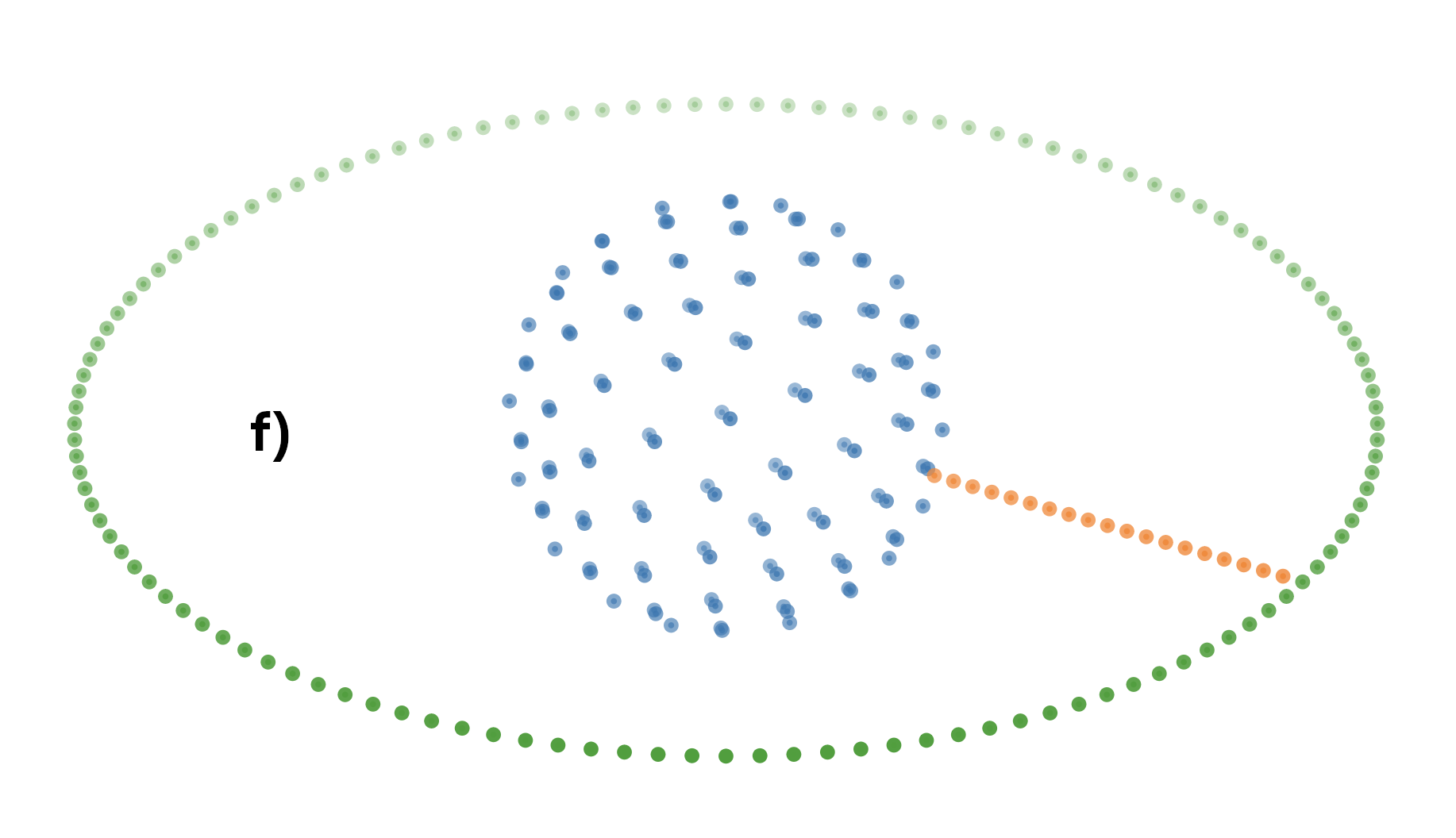}
			\includegraphics[height=5cm]{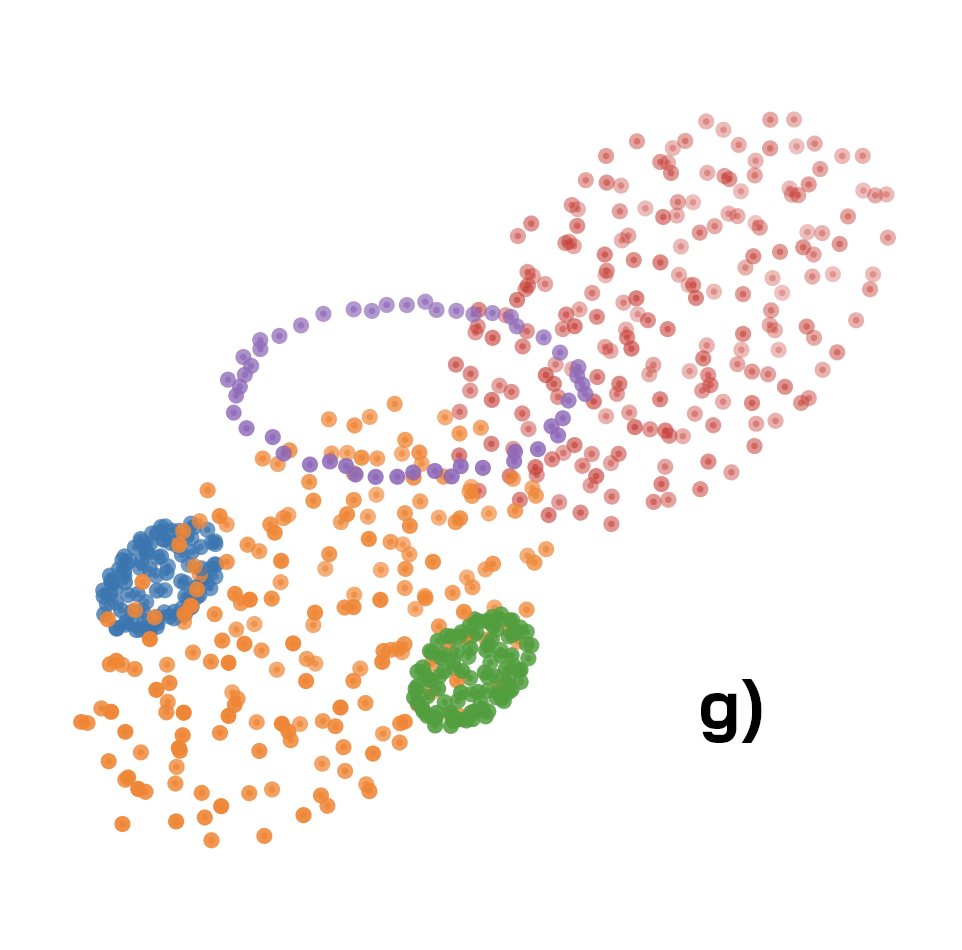}
		}
		\caption{\figuretitle{Data sets of the Topological Clustering Benchmark Suite (\TCBS) with true labels.}
			\emph{Top:} $2D$ data sets. \emph{From left to right:}
			\emph{a):} \texttt{4Spheres} (656 points),
			\emph{b):} \texttt{Ellipses} (158 points),
			\emph{c):} \texttt{Spheres+Grid} (866 points),
			\emph{d):} \texttt{Halved Circle} (249 points).
			\emph{Bottom:} $3D$ data sets. \emph{From left to right:}
			\emph{e):} \texttt{2Spheres2Circles} (4600 points),
			\emph{f):} \texttt{SphereinCircle} (267 points),
			\emph{g):} \texttt{spaceship} (650 points).}
		\label{fig:TCBS}
	\end{center}
\end{figure}

\begin{figure}[tb!]
	\begin{center}
		\resizebox{\textwidth}{!}{
			\includegraphics[height=5cm]{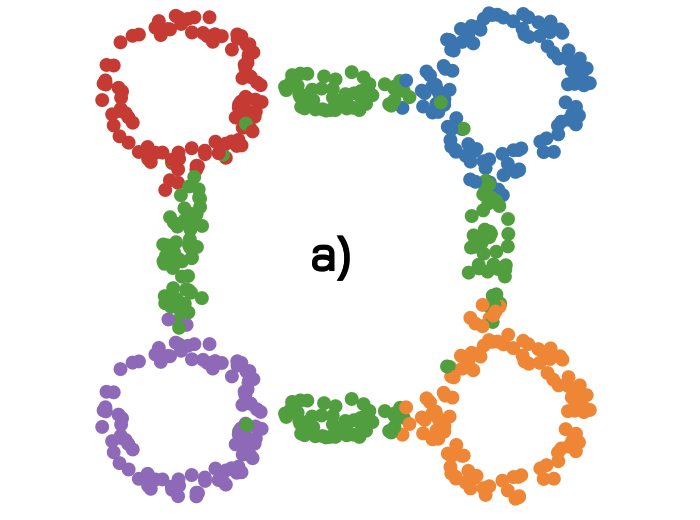}
			\includegraphics[height=5cm]{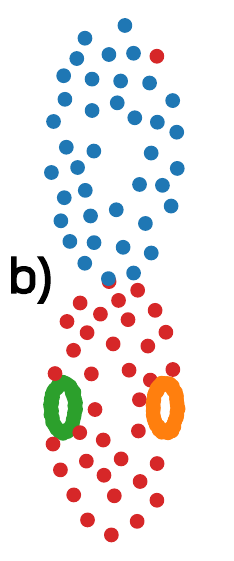}
			\includegraphics[height=5cm]{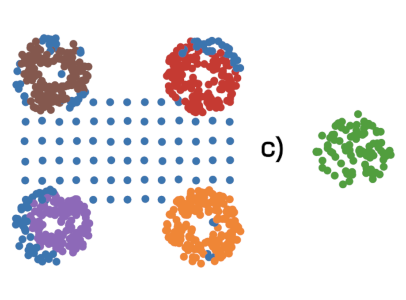}
			\includegraphics[height=5cm]{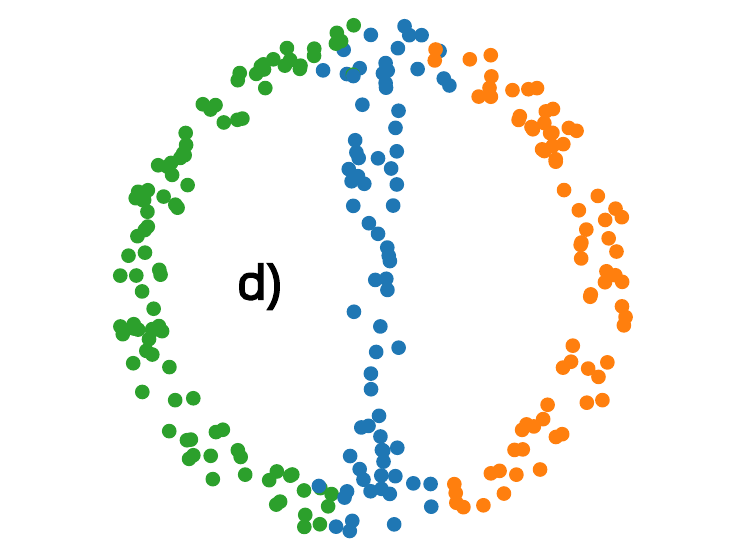}
		}
		\resizebox{\textwidth}{!}{
			\includegraphics[height=5cm]{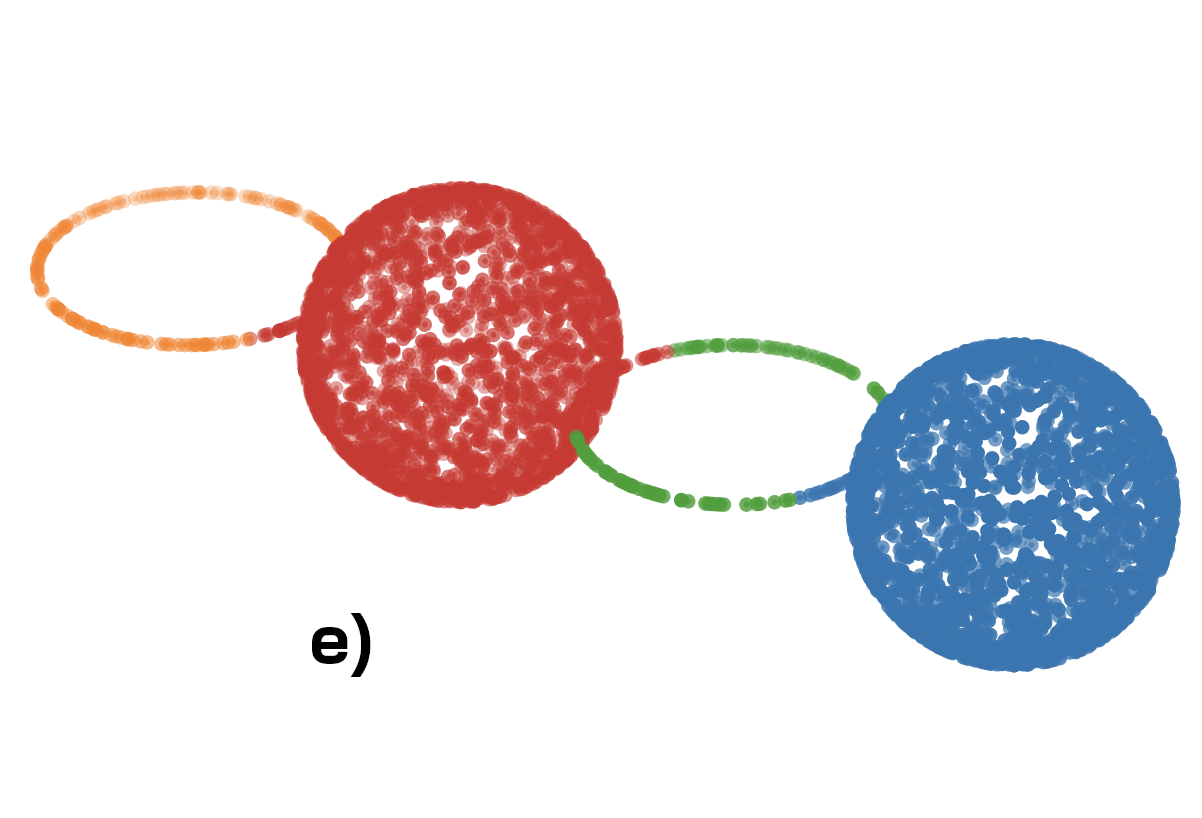}
			\includegraphics[height=5cm]{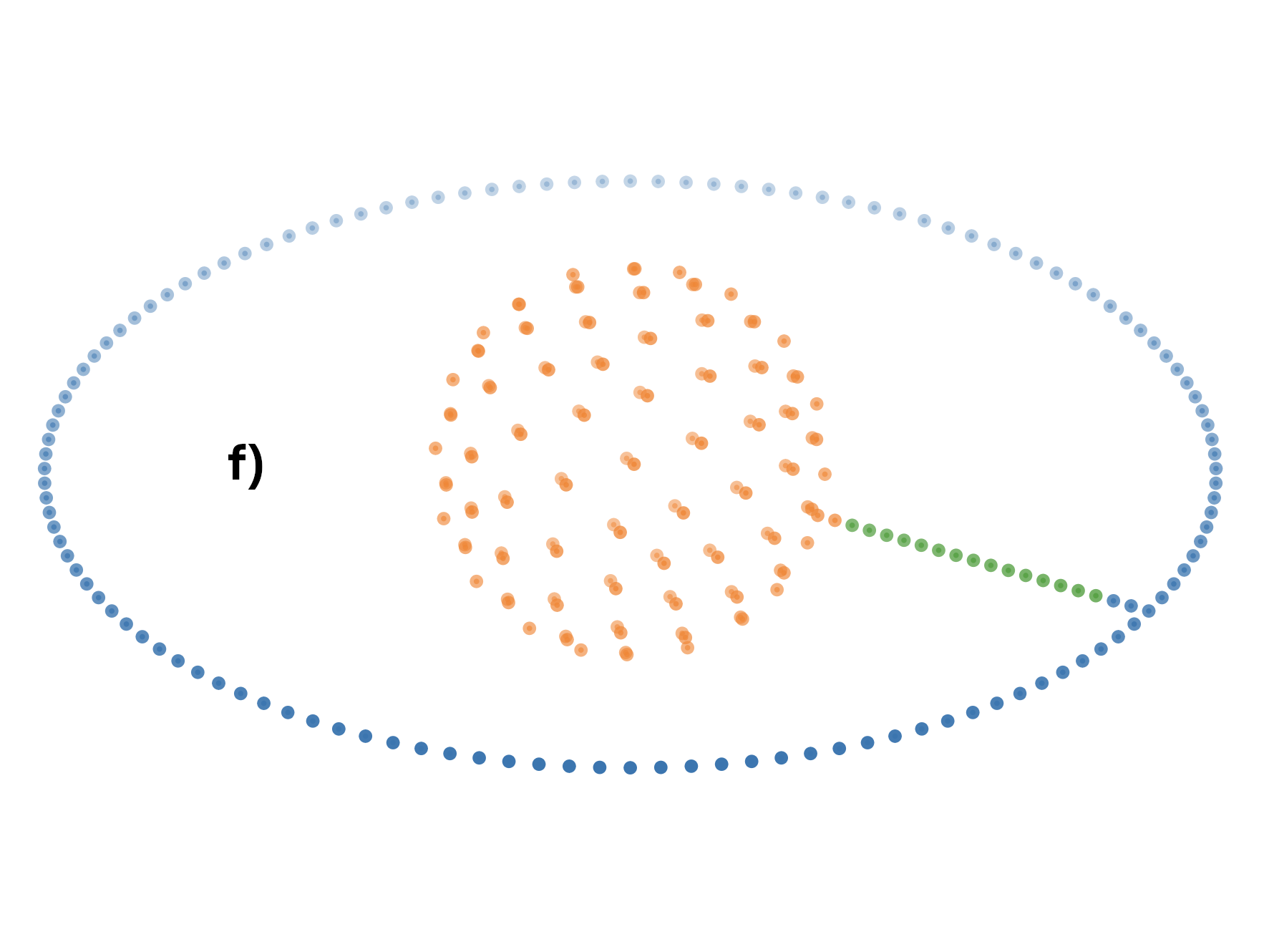}
			\includegraphics[height=5cm]{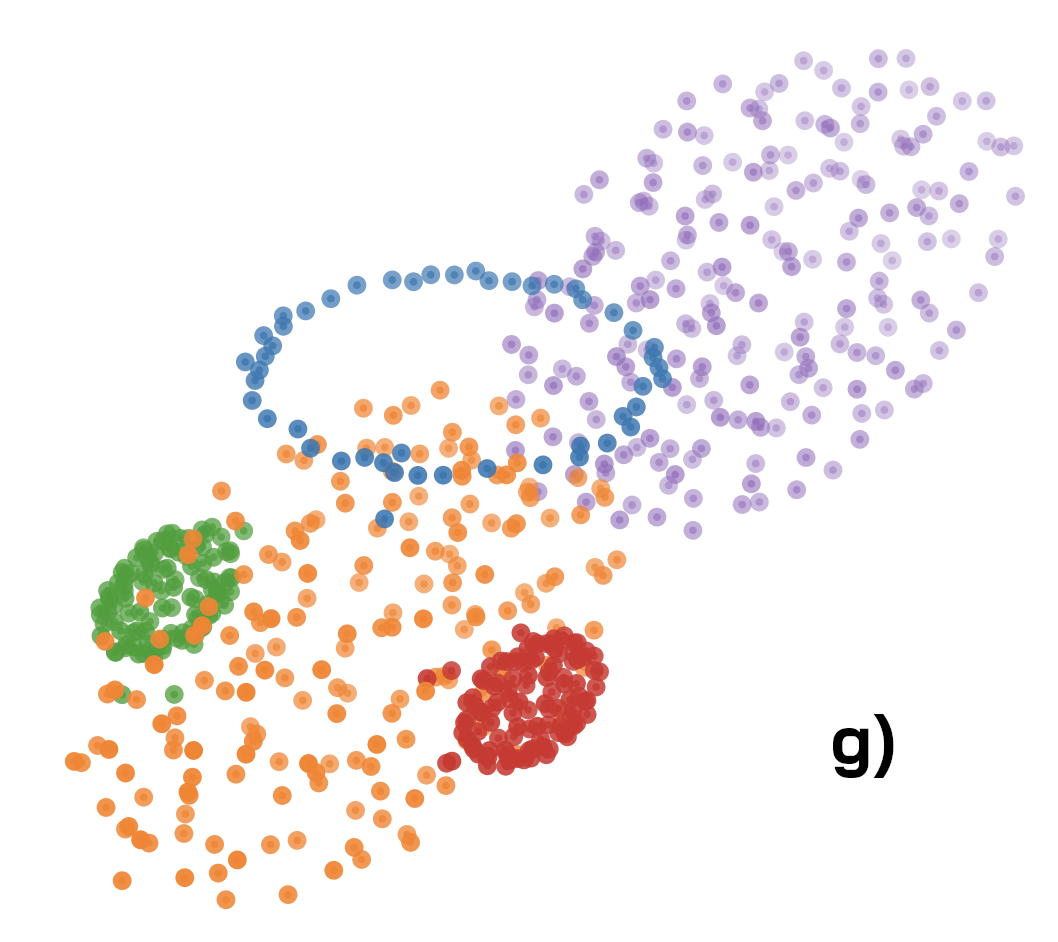}
		}
		\caption{\figuretitle{Data sets of the Topological Clustering Benchmark Suite (\TCBS{}) with labels generated by \TOPF{}.}
			\emph{Top:} $2D$ data sets. \emph{From left to right:}
			\emph{a):} \texttt{4Spheres} (0.81 \ARI{}),
			\emph{b):} \texttt{Ellipses} (0.95 \ARI{}),
			\emph{c):} \texttt{Spheres+Grid} (0.70 \ARI{}),
			\emph{d):} \texttt{Halved Circle} (0.71 \ARI{}).
			\emph{Bottom:} $3D$ data sets. \emph{From left to right:}
			\emph{e):} \texttt{2Spheres2Circles} (0.94 \ARI{}),
			\emph{f):} \texttt{SphereinCircle} (0.97 \ARI{}),
			\emph{g):} \texttt{spaceship} (0.92 \ARI{}).}
		\label{fig:TCBSTOPF}
	\end{center}
\end{figure}

\section{Implementation}
\label{app:implementation}
We have created an easy-to-use python package \TOPF{} which can be found
\iftoggle{arxiv}{
at github (\url{https://github.com/vincent-grande/topf}) and PyPi (\url{https://pypi.org/project/topf/})
with many examples from the paper and the topological clustering benchmark suite.
}{
in the supplementary material.
}
The package currently works under macOS and Linux.
The package contains both the code to generate the topological point features, as well as the code to reproduce the various visualisation steps in this paper.

All experiments were run on a Apple \textsc{m}\oldstylenums{1} Pro chipset with 10 cores and 32 \textsc{\lowercase{GB}} memory.
\TOPF{} and the experiments are implemented in Python and Julia.
For persistent homology computations, we used \textsc{\lowercase{GUDHI}} \citep{gudhi:urm} (\copyright{} The \textsc{\lowercase{GUDHI}} developers, \textsc{\lowercase{MIT}} license) and Ripserer  \citep{cufar2020ripserer} (\copyright{} mtsch, \textsc{\lowercase{MIT}} license), which is a modified Julia implementation of \citep{Bauer2021}.
For the least square problems, we used the \textsc{\lowercase{LSMR}} implementation of SciPy \citep{fong2011lsmr}.
We used the Node2Vec python implementation \url{https://github.com/eliorc/node2vec} (\copyright{} Elior Cohen, \textsc{\lowercase{MIT}} License) based on the Node2Vec Paper \citep{grover2016node2vec}.
We used the \texttt{pgeof} Python package for computation of geometric features \url{https://github.com/drprojects/point_geometric_features} (\copyright{} Damien Robert, Loic Landrieu, Romain Janvier, \textsc{\lowercase{MIT}} license).
We use parts of the implementation of \TPCC{} \url{https://git.rwth-aachen.de/netsci/publication-2023-topological-point-cloud-clustering} (\copyright{} Computational Network Science Group, \textsc{\lowercase{RWTH}} Aachen University, \textsc{\lowercase{MIT}} license).
We use the implementation of WSDesc \cite{Li2022WSDesc}, \url{https://github.com/craigleili/WSDesc} (Lei Li, Hongbo Fu, Maks Ovsjanikov. \textsc{\lowercase{CC BY-NC}} 4.0 License).
We use the implementation of An Tao, \url{https://github.com/antao97/dgcnn.pytorch} of Dynamic Graph \textsc{cnn}s \cite{wang2019dynamic}, \textsc{MIT} license.
The idea for the implementation of the fix of the lift of the $\ZthreeZ$-representative to $\R$-coefficients was inspired by DREiMac, \cite{Perea2023}.
\subsection{Hyperparameters}
\label{app:hyperparameters}

\begin{figure}[ht!]
	\begin{center}
		\includegraphics[width=0.37\linewidth]{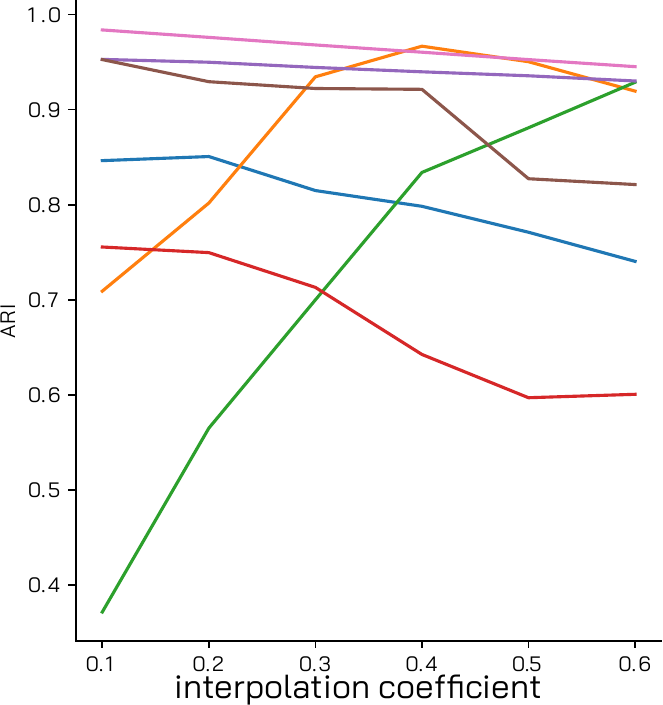}
		\includegraphics[width=0.6\linewidth]{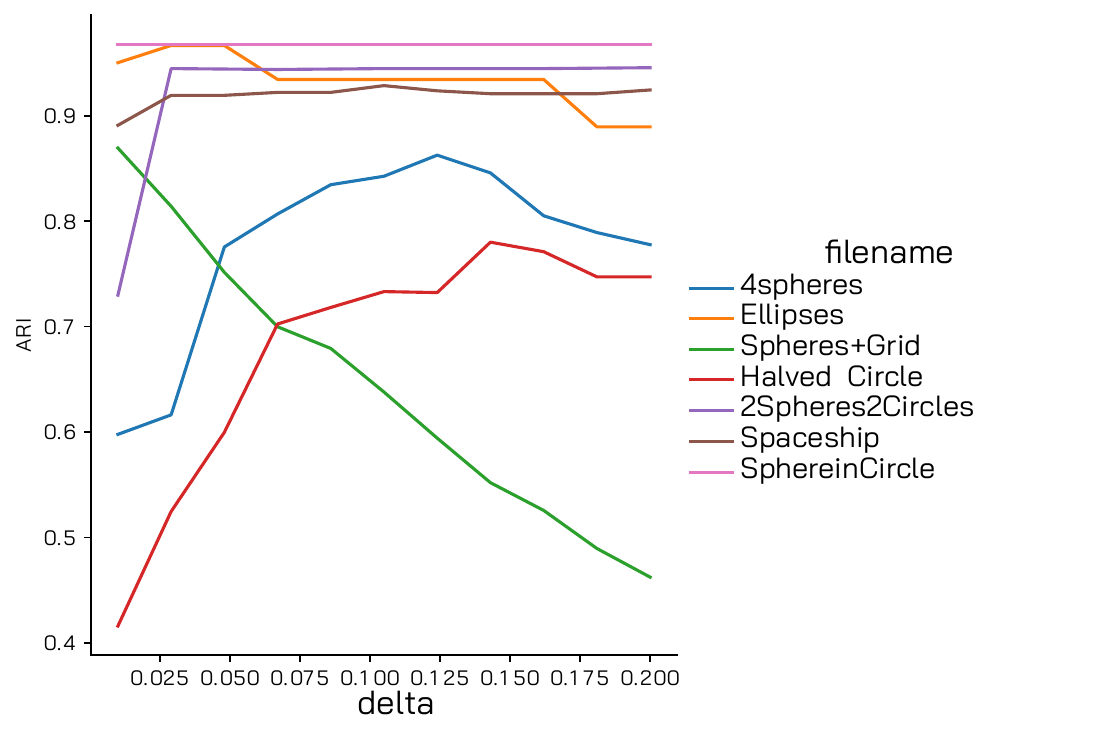}
		\includegraphics[width=0.37\linewidth]{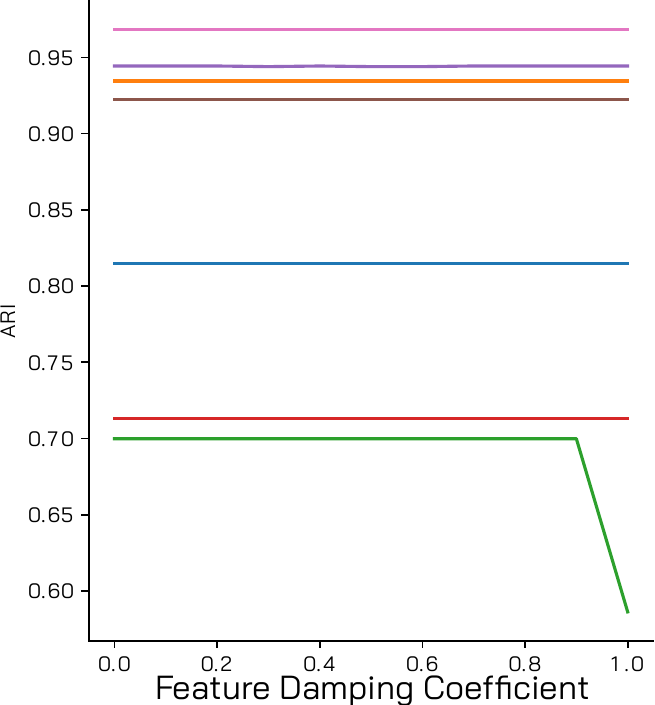}
		\includegraphics[width=0.6\linewidth]{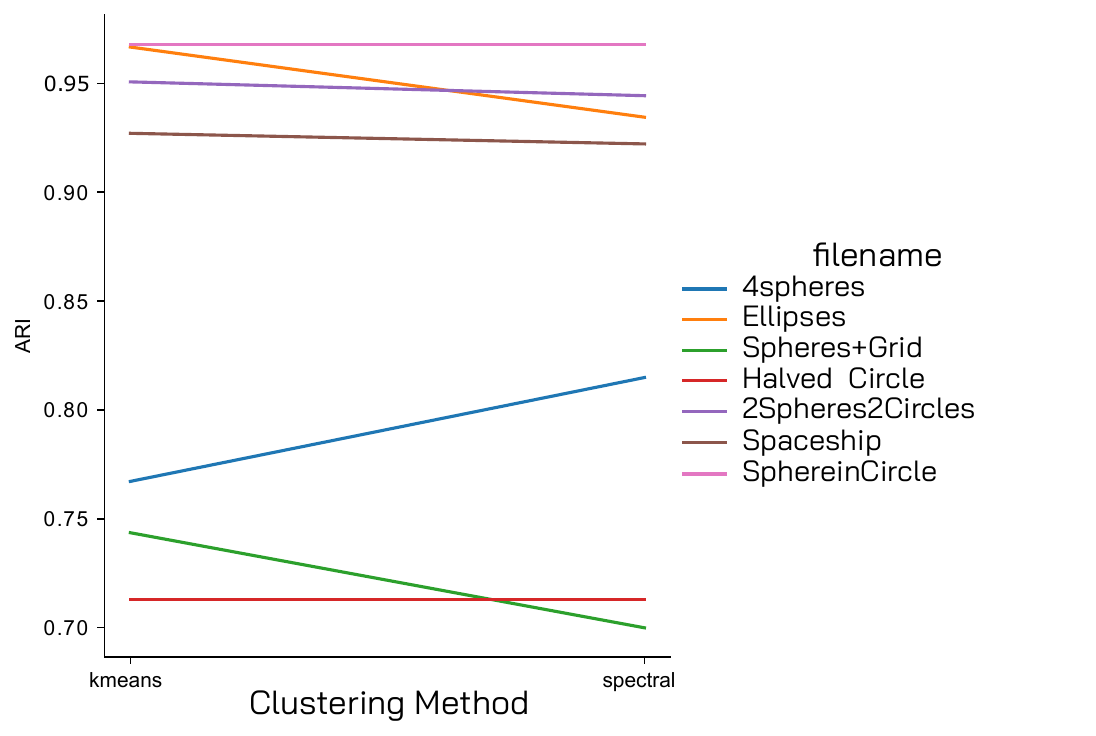}
		\caption{\figuretitle{Hyperparameter robustness of \TOPF{} on the \TCBS{}}.
			\TOPF{} performs reasonably well across a large range around the default parameters (Interpolation coefficient: $0.3$, $\delta$: $0.07$, damping coefficient (non-negative): $0$, Clustering method: \texttt{spectral}).
			The clustering method is used on the \TOPF{} vectors to determine the cluster-ids.
		}
		\label{fig:hyperparameterrobustness}
	\end{center}
	\vspace{-0.2in}
\end{figure}

All the relevant hyperparameters are already mentioned in their respective sections.
However, for convenience we gather and briefly discuss them in this section.
We note that \TOPF{} is robust and applicable in most scenarios when using the default parameters without tuning hyperparameters.
The hyperparameters should more be thought of as an additional way where detailed domain-knowledge can enter the \TOPF{} pipeline.
We have performed additional experiments to validate the robustness of \TOPF{} to moderate changes in parameters, see \Cref{fig:hyperparameterrobustness}.

\paragraph{Maximum Homology Dimension $d$}
The maximum homology dimension determines the dimensions of persistent homology the algorithm computes.

For the choice of the maximum homology degree $d$ to be considered there are mainly three heuristics which we will list in decreasing importance (Cf. \cite{Grande:2023}): 

\begin{enumerate}[I.]
	\item In applications, we usually know which kind of topological features we are interested in, which will then determine $d$.
	This means that $1$-dimensional homology and $d=1$ suffices when we are looking at loops of protein chains.
	On the other hand, if we are working with voids and cavities in 3d histological data, we need $d=2$ and thus compute $2$-dimensional homology.
	\item Algebraic topology tells us that there are no closed $n$-dimensional submanifolds of $\R^n$.
	Hence their top-homology will always vanish and all interesting homological activity will appear for $d<n$.
	\item In the vast majority of cases, the choice will be between $d=1$ or $d=2$ because empirically there are virtually no higher-dimensional topological features in practice.
\end{enumerate}

In our quantitative experiments, we have always chosen $d=n-1$.

\paragraph{Thresholding parameter $\delta$}

In step 4 of the algorithm, we normalise and threshold the harmonic representatives.
After normalising, the entries of the vectors lie in the interval of $[0,1]$.
The thresholding parameter $\delta$ now essentially determines an interval of $[0,\delta]$ which we will linearly map to $[0,1]$, while mapping all entries above $\delta$ to $1$ as well.
This is necessary as most of the entries in the vector $e_k^i$ are very close to $0$ with a very small number of entries being close to $1$.
Without this thresholding, \TOPF{} would now be almost entirely determined by these few large values.
Thus this step limits the maximum possible influence of a single entry.
However, because most of the entries of $e_k^i$ are concentrated around $0$, small changes in $\delta$ will not have a large effect and we chose $\delta = 0.07$ in all our experiments.
\paragraph{Interpolation coefficient $\lambda$}

The interpolation coefficient $\lambda \in (0,1)$ determines whether we build our simplicial complexes close to the birth or the death of the relevant homological features at time $t=b^{1-\lambda}d$.
This then in turns controls how localised or smooth the harmonic representative will be.
In general, the noisier the ground data is the higher we should choose $\lambda$.
However, \TOPF{} is not sensitive to small changes in $\lambda$.
We have picked $\lambda=0.3$ for all the quantitative experiments, which empirically represents a good choice for a broad range of applications.

\paragraph{Feature selection factor $\beta$}

Increasing $\beta$ leads to \TOPF{} preferring to pick a larger number of relevant topological features.
Without specific domain-knowledge, $\beta = 0$ represents a good choice.

\paragraph{Feature selection quotients \texttt{max\_total\_quot}, \texttt{min\_rel\_quot}, and \texttt{min\_0\_ratio}}

These are technical hyperparameters controlling the feature selection module of \TOPF{}.
For a technical account of them, see \Cref{app:featureselection}.
In most of the cases without domain knowledge, they do not have an effect on the performance of \TOPF{} and should be kept at their default values.

\paragraph{Simplicial Complex Weights}

Although the simplicial weights are not technically a hyperparameter, there are many potential ways to weigh the considers \tSC{}s that can highlight or suppress different topological and geometric properties.
In all our experiments, we use $w_\Delta$ weights discussed in \Cref{app:WeightedSCs}.
\paragraph{Ablation study}
In \Cref{tab:ablation}, we performed an ablation study and benchmarked \TOPF{} against \TOPF{} while skipping the harmonic projection step with the Hodge Laplacian (Step \textbf{3.}).
The result show that in general, skipping the step with the Hodge Laplacian decreases the performance of \TOPF{} by a wide margin. The exception to this are \texttt{SphereInCircle} or \texttt{2Spheres2Circles}, where the two methods have similar performance. In these two last examples, points are directly sampled from a manifold, and thus the loops and spheres are very "thin". Thus, the homology generator is already a good approximation of all the simplices responsible for a topological feature. In general however, this is not the case, necessating the use of Differential Geometry and the Hodge Laplacian. 
\subsection{Runtime}
In \Cref{fig:runtime}, we analyse the runtime of the current \TOPF{} implementation on two point clouds from the topological clustering benchmark suite.
We increase the point density while keeping the structure of the point cloud intact.
The runtime of our implementation with \texttt{sparsification = 'off'} in the regime of \Cref{fig:runtime} appears to scale linearly with the total number of points with a significant constant term.
The significant constant term is probably due to our implementation starting a new julia kernel and communicating between python and julia.
A python-only or julia-only implementation would thus speed up computations.
\paragraph{Theoretical runtime complexity}
\begin{enumerate}
\item \textbf{a). Constructing the complex:} The Vietoris--Rips complex will have $O(n^{k+1})$ simplices, where $n$ is the number of points and $k$ the maximum homology dimension. Apart from this, there is no significant computational effort needed. Note that we can decrease this if we set a maximum VR-radius, as done in many works from the literature. Computing the alpha-complex in ambient dimension $d$ is equivalent to computing the convex hull in ambient dimension $d+1$, which has a complexity of $\sim O(n^{d/2})$. For a fixed dimension $d$, the number of simplices is linear in $n$. \textbf{b). Computing persistent homology and generators:} As we need the homology generators, we use the recently introduced involuted persistent homology \cite{vcufar2023fast}. While the authors discuss the runtime performance, they do not give a concrete complexity bound in \cite{vcufar2023fast}. However, they deduce that it has a similar run-time to usual persistent cohomology computations, whose performance is for example discussed in \cite{de2011dualities}.
	
\item \textbf{Picking the significant generators} has negligible impact.
	
\item \textbf{Computing the harmonic projections.} Given a homology representative $r$ in dimension $k$, computing the harmonic representative amounts to computing a sparse matrix vector product of $B_{k+1}x$ and a sparse least squares problem $\text{lsmr}(B_{k+1},r)$, i.e. solving \[\min_{x\in \mathbb{R}^{S_{k+1}}} \lVert r-B_{k+1}x\rVert_2\]
	where $S_{k+1}$ is the number of $(k+1)$-simplices in the simplicial complex and $B_{k+1}$ has $(k+2)S_{k+1}$ non-zero entries. This is a sparse least-squares problem, which we solve using the iterative sparse solver \texttt{lsmr} \citep{fong2011lsmr}. Because this is an iterative solver, the authors do not give runtime complexity, but discuss the computational requirements in \cite{fong2011lsmr}.
		
\item \textbf{Pooling and averaging} over the simplicial neighbours has negligible runtime constraints.
\end{enumerate}
\begin{figure}[ht!]
	\begin{center}
		\includegraphics[width=0.45\textwidth]{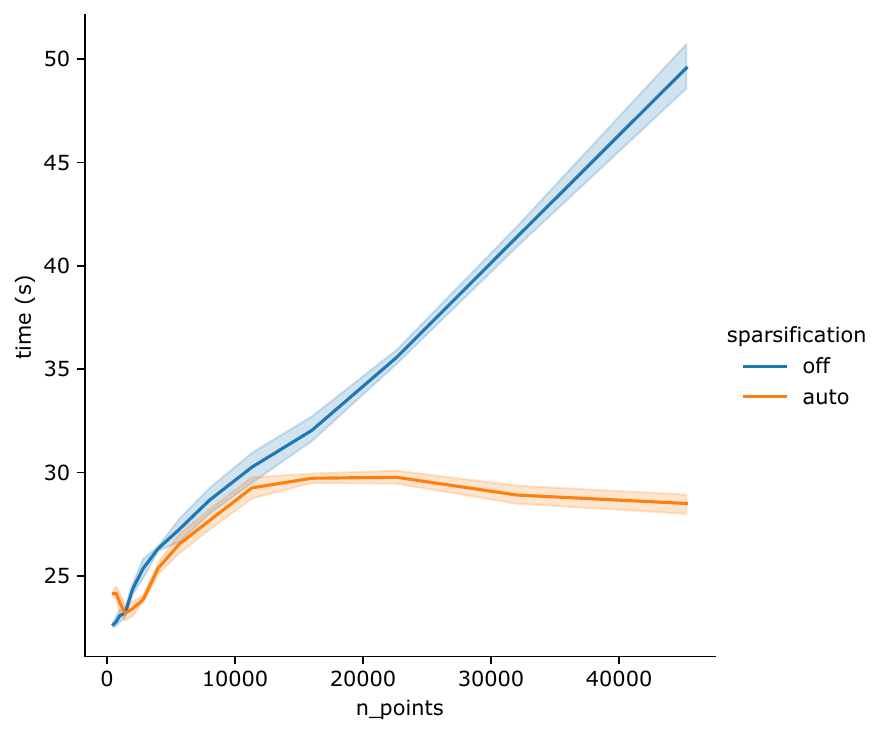}
		\includegraphics[width=0.45\textwidth]{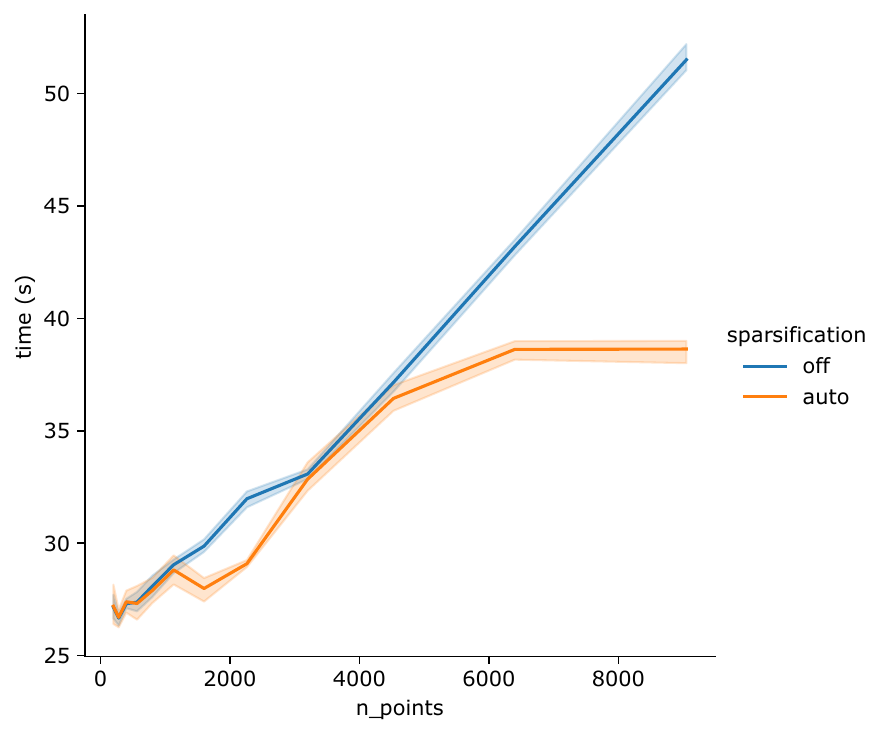}
		\caption{\figuretitle{Runtime on two examples from Topological Clustering Benchmark Suite while increasing point density.}
			\emph{Left:} \texttt{Halved Circle}, \emph{Right:} \texttt{2 Spheres 2 Circles}. Experiments kept the topological structure intact while increasing the number of points sampled. Both the runtime of the vanilla \TOPF{} algorithm, as well as of \TOPF{} using automatic subsampling heuristic as implemented in the python package, are listed.
		}
		\label{fig:runtime}
	\end{center}
\end{figure}
\section{More Details on the Experiments}
\label{app:experiments}
	\paragraph{High-dimensional Point Clouds}
	\begin{comment}
\begin{figure}[ht!]
	\begin{center}
		\includegraphics[width=0.58\textwidth]{figs/Cyclo8Heatmap3plots.png}
		\includegraphics[width=0.19\textwidth]{figs/Cyclo83clusters.png}
		\includegraphics[width=0.19\textwidth]{figs/Cyclo84clusters.png}
		
		\caption{\figuretitle{\TOPF{} on a high-dimensional point cloud.}
			We used \TOPF{} on $6500$ points sampled from the $24$-dimensional conformation space of cyclooctane \cite{Martin2011}.
			Due to the dimension constraints of papers, we have to show the \textsc{\lowercase{ISOMAP}} projection from $24$ dimensions down to $3$ dimensions.
			\emph{Left:} Three (one $1$-dimensional and two $2$-dimensional) features automatically selected by \TOPF{}.
			\emph{Right:} Result of clustering for three (automatically chosen by \TOPF{}) and four clusters.
		}
		\label{fig:cyclo8}
	\end{center}
	\vspace{-0.2in}
\end{figure}
\end{comment}
In \Cref{fig:cyclo8}, we use \TOPF{} on a high-dimensional input point cloud representing the conformation space of cyclooctane \cite{Martin2011}.
Because the ambient dimension is too large for $\alpha$-filtrations, a Vietoris--Rips filtration with more simplices is used.
The \VR{} filtration depends on the ambient dimension only for computing the distance matrix, which is a negligible part of the runtime.
Hence, a similar performance can be expected for higher dimensions than $24$ as well.
To counteract the increase in computational complexity, \TOPF{} automatically downsamples the point cloud using a mixture of landmark and random sampling.
We see in the \textsc{isomap} projections of \Cref{fig:cyclo8} that \TOPF{} extracts reasonable features representing the topology of the conformation space.
Setting $n_\text{clusters}$ to $4$ even reveals the manifold anomalies in a separate cluster (\emph{blue}), similar to the results of \cite{Stolz2020}.
The runtime for this high-dimensional experiment was $189.8$ seconds.

\paragraph{Variational Autoencoders}
We have trained a \VAE{} on image patches sampled with a topological structure, see \Cref{fig:figcow}.
We have sampled the image patches around two loops and a connecting line.
Using \TOPF{}, we could recover this structure in the latent space of the \VAE{} highlighting how topological point features can help in explainable \textsc{ai}.
We note, of course, that many latent spaces do not carry any higher-order topological information, making characterisation by \TOPF{} infeasible.
Rather we showed that when we expect a topological structure to exist in the \emph{data set} due to some specifics of the data, we can recover this topological structure even in the latent space of the \VAE{} using \TOPF{}.
We have repeated the experiments for a higher-dimensional latent space and received similar results, see \Cref{fig:figcowhigh}.
\begin{figure}[ht!]
	\begin{center}
		\includegraphics[width=\textwidth]{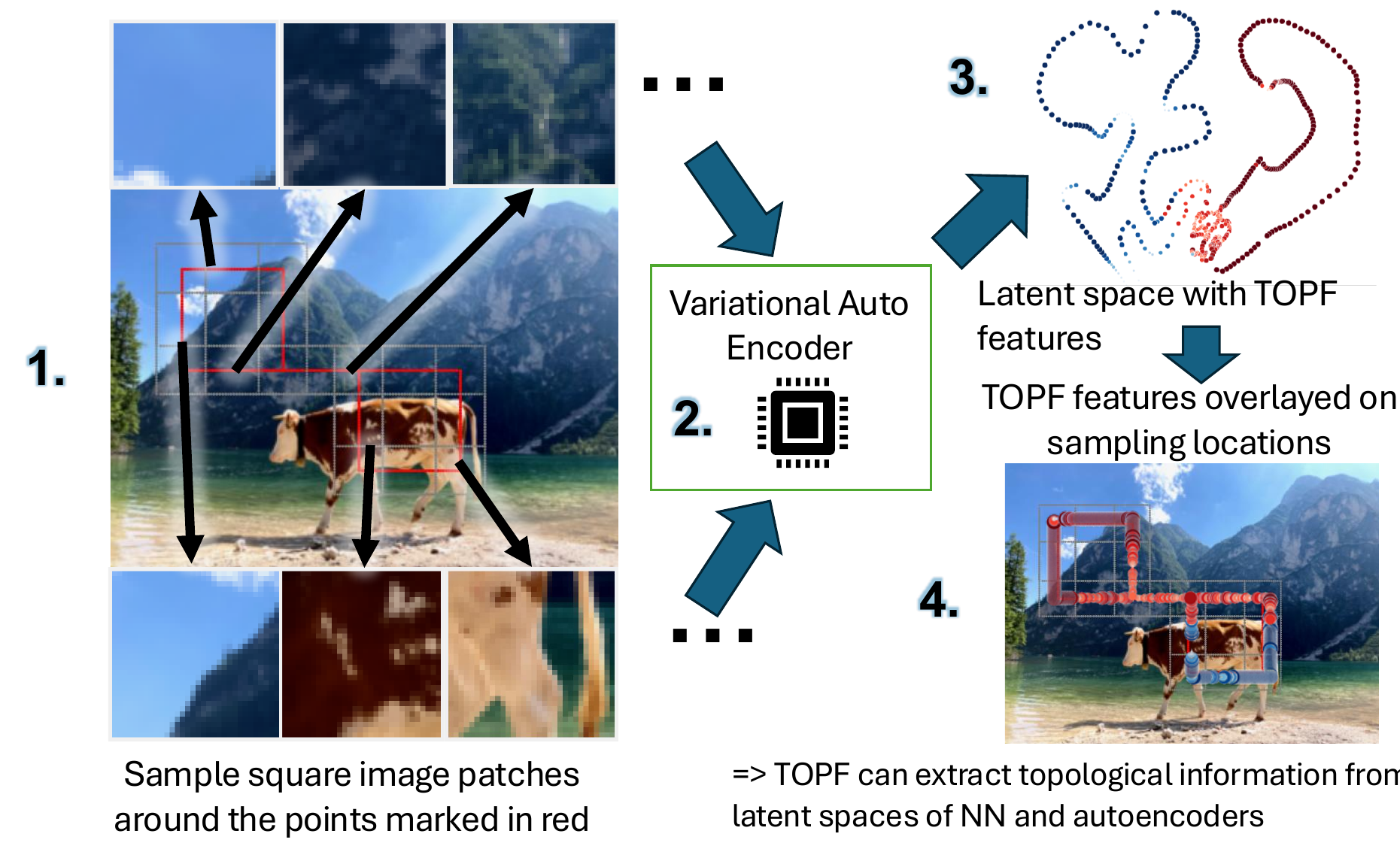}
		\caption{\figuretitle{Using \TOPF{} to explore the topological structure in latent/embedding spaces of Variational Auto Encoders (\textsc{vae})}
			\textbf{1.} Given a picture (of a cow), we sample $578$ square patches around the centre points marked in red in the image.
			The centre points are taken from the sides of two squares and a line connecting the two squares.
			The assumption is that this topological structure (with two holes) is present in the sample space as well.
			\textbf{2.} We train a \textsc{vae} on the set of $578$ square patches with latent space dimension of $3$ (down from $33\cdot 33 \cdot 3$.)
			\textbf{3.} We run \TOPF{} with \texttt{fixed\_num\_features} set to $[0,2]$ on the latent space to extract the two most significant point-wise topological features.
			\textbf{4.} We overlay the topological features from the latent space over the centre points of the corresponding image patches. We see that \TOPF{} has roughly \emph{recovered the topological structure inherent in the sample space} as described in step 1. We note that the image patches with centre on the middle horizontal line are almost identical, making it virtually impossible to distinguish them. Note that although \TOPF{} thinks they contribute to the left loop, their corresponding feature entries are a lot smaller than those of the image patches (shown by lighter colours and smaller dots in the plot.)
		}
		\label{fig:figcow}
	\end{center}
	\vspace{-0.2in}
\end{figure}
\begin{figure}[ht!]
	
	\begin{center}
		\includegraphics[width=0.2\linewidth]{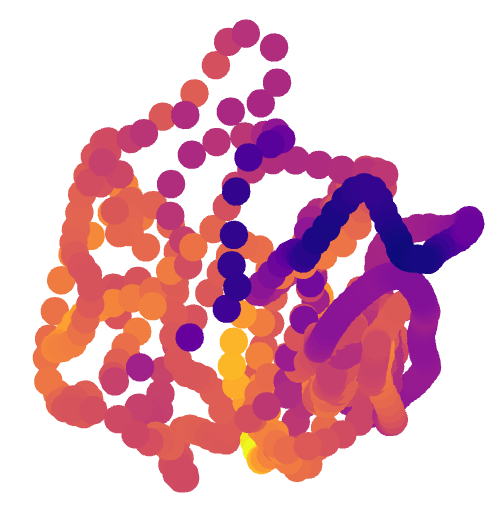}
		\includegraphics[width=0.2\linewidth]{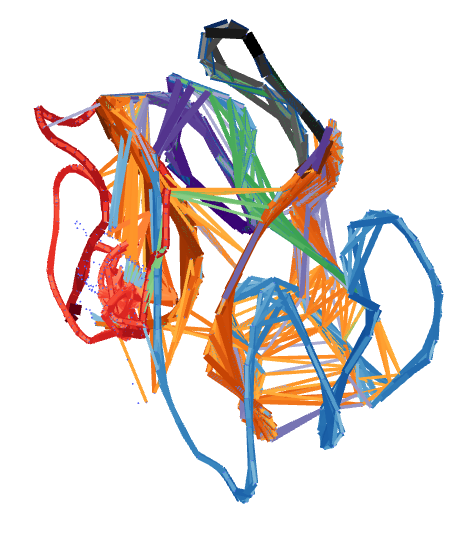}
		\includegraphics[width=0.28\linewidth]{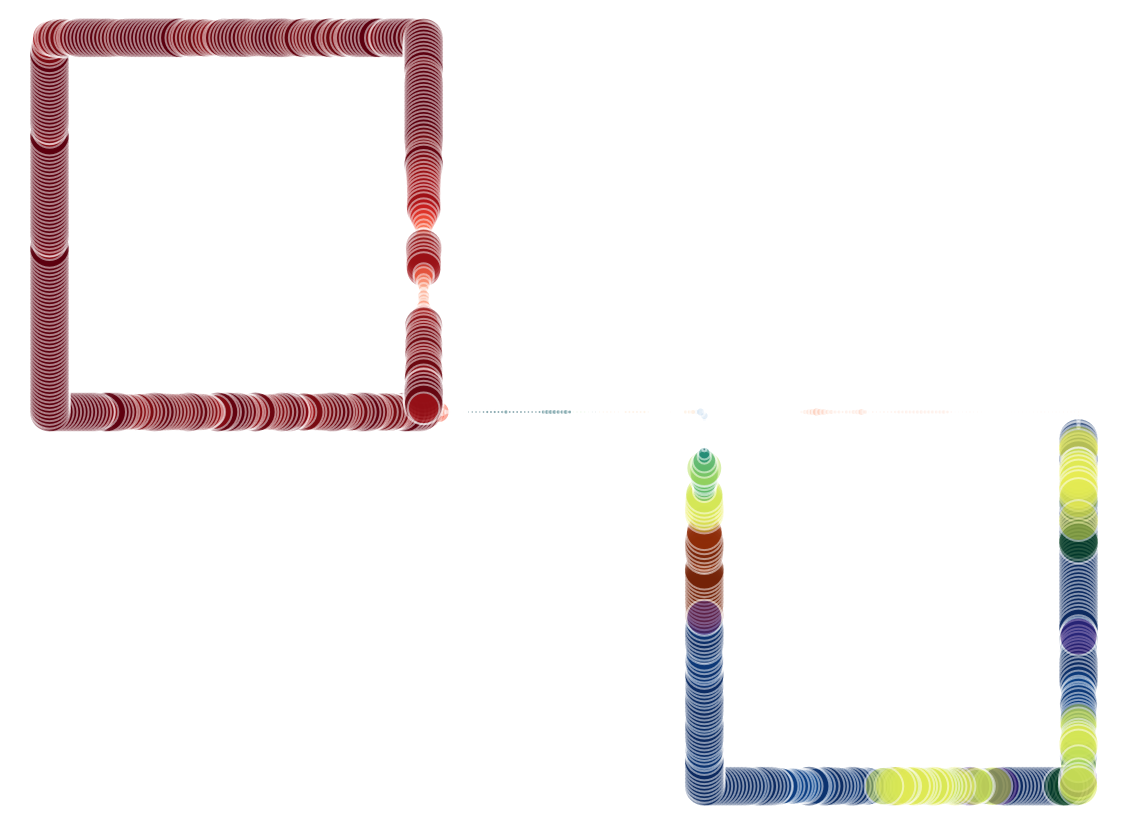}
		\includegraphics[width=0.28\linewidth]{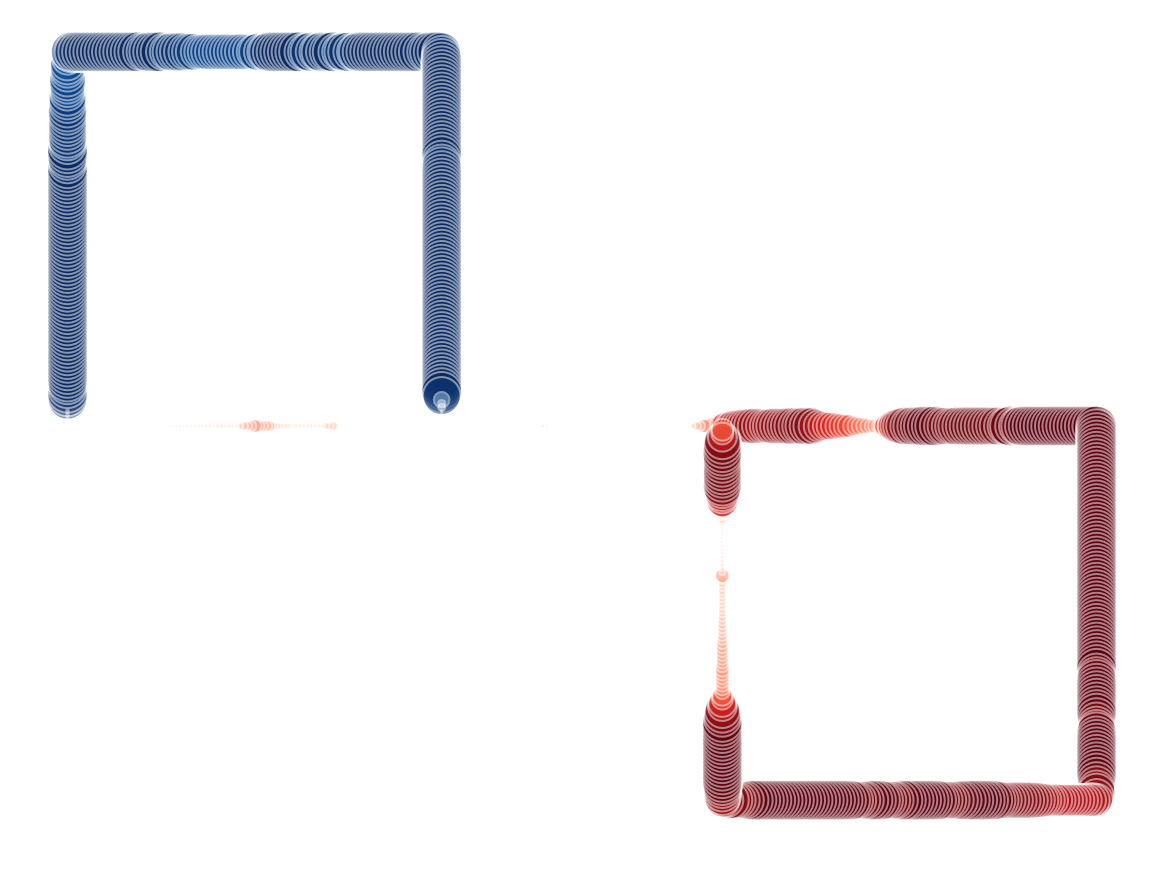}
		
		\caption{\figuretitle{\textsc{vae} experiment described in \Cref{fig:figcow} for a 16-dimensional latent space of the \textsc{\lowercase{VAE}} and the original 8748-dimensional image space.}
			\emph{Left:} The first three of the $16$ dimension of the latent space embeddings, colour represents the fourth dimension.
			\emph{Centre left:} The weighted harmonic representatives of the selected topological features.
			\emph{Centre right:} The features produced by \TOPF{} overlayed on the centre pixels of the associated image patches, as done in \Cref{fig:figcow} Step 4., just without the picture of the cow.
			\emph{Right:} \TOPF{} features obtained from the original $8748$-dimensional image space.
			This shows that \TOPF{} can help analyse topological structure in high-dimensional data. 
		}
		\label{fig:figcowhigh}
	\end{center}
\end{figure}
\begin{comment}
\begin{figure}[ht!]
	\begin{center}
		\includegraphics[width=0.2\textwidth]{figs/CowVAEHighDimLatentspacebase.png}
		\includegraphics[width=0.2\textwidth]{figs/CowVAEHighDimGenerators.png}
		\includegraphics[width=0.28\textwidth]{figs/CowVAEHighDimResult.png}
		\includegraphics[width=0.28\textwidth]{figs/TOPFimagespace.png}
		
		\caption{\figuretitle{Repeating the \VAE{} experiment of \Cref{fig:figcow} for a $16$-dimensional latent space of the \textsc{vae} and the original $8748$-dimensional image space.}
			\emph{Left:} The first three of the $16$ dimension of the latent space embeddings, colour represents the fourth dimension.
			\emph{Centre left:} The weighted harmonic representatives of the selected topological features.
			\emph{Centre right:} The features produced by \TOPF{} overlayed on the centre pixels of the associated image patches, as done in \Cref{fig:figcow} \textbf{4.}, just without the picture of the cow.
			The two most prominent and significant features roughly correspond to the image patches sampled from the left and the right square.
			However, due to the high dimension of the embedding space, \TOPF{} extracts some weaker features corresponding to details of image patches close to the rear part of the cow.
			\emph{Right:} \TOPF{} features obtained from the original $8748$-dimensional image space.
			This experiment shows that \TOPF{} can help detect and quantify topological structure in high-dimensional data. 
		}
		\label{fig:figcowhigh}
	\end{center}
	\vspace{-0.2in}
\end{figure}
\end{comment}
\paragraph{Additional heatmaps on proteins}
In \Cref{fig:OldProtein}, we provide additional experiments of \TOPF{} on protein data.
This time, we report every single feature vector produced by \TOPF{} in a separate plot.
In the experiment with Cys123, we wanted to show that \TOPF{} can detect so-called protein pockets. Because pockets do not form holes detectable by \TDA{} straightforward, the literature suggests adding (and later removing) these points on the convex hull, which turns the pockets into holes, see for example \cite{oda2024novel}.
\begin{figure}[tb!]
	\vskip -0.1in
	\begin{center}
		\resizebox{\textwidth}{!}{
			\includegraphics[height=10cm]{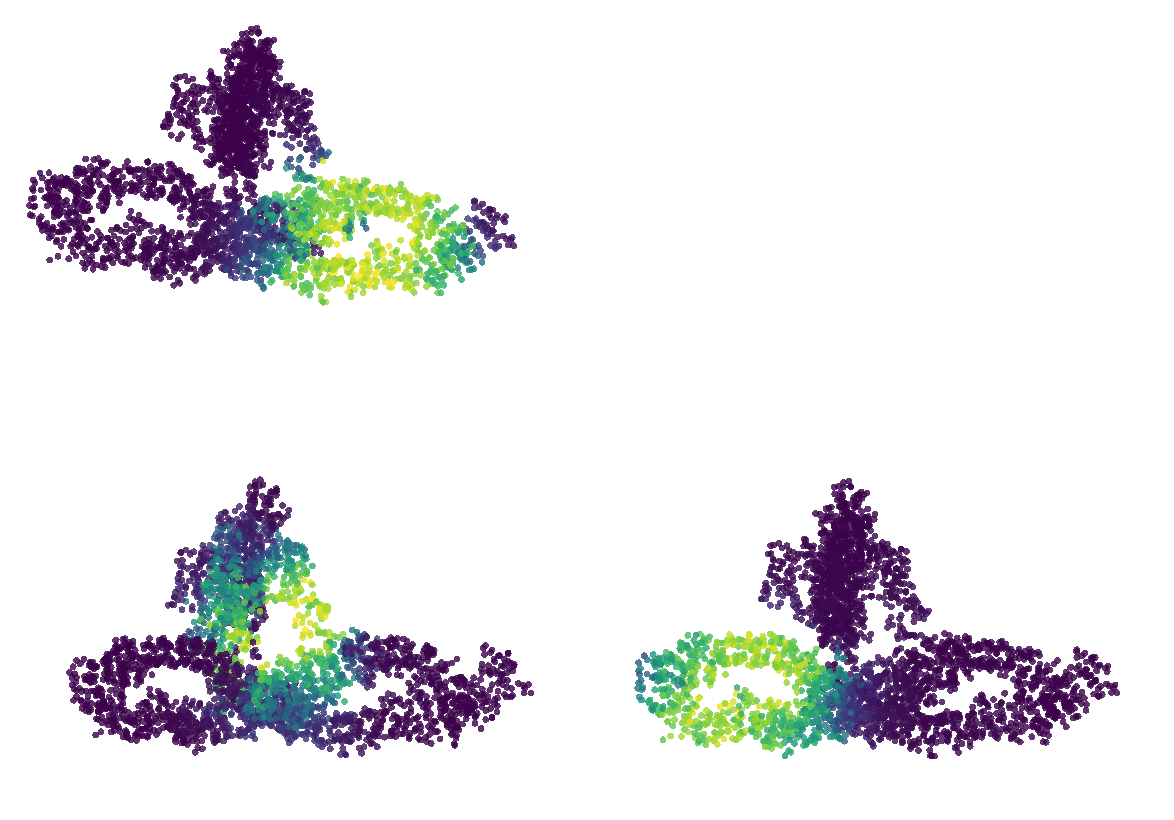}
			\includegraphics[height=10cm]{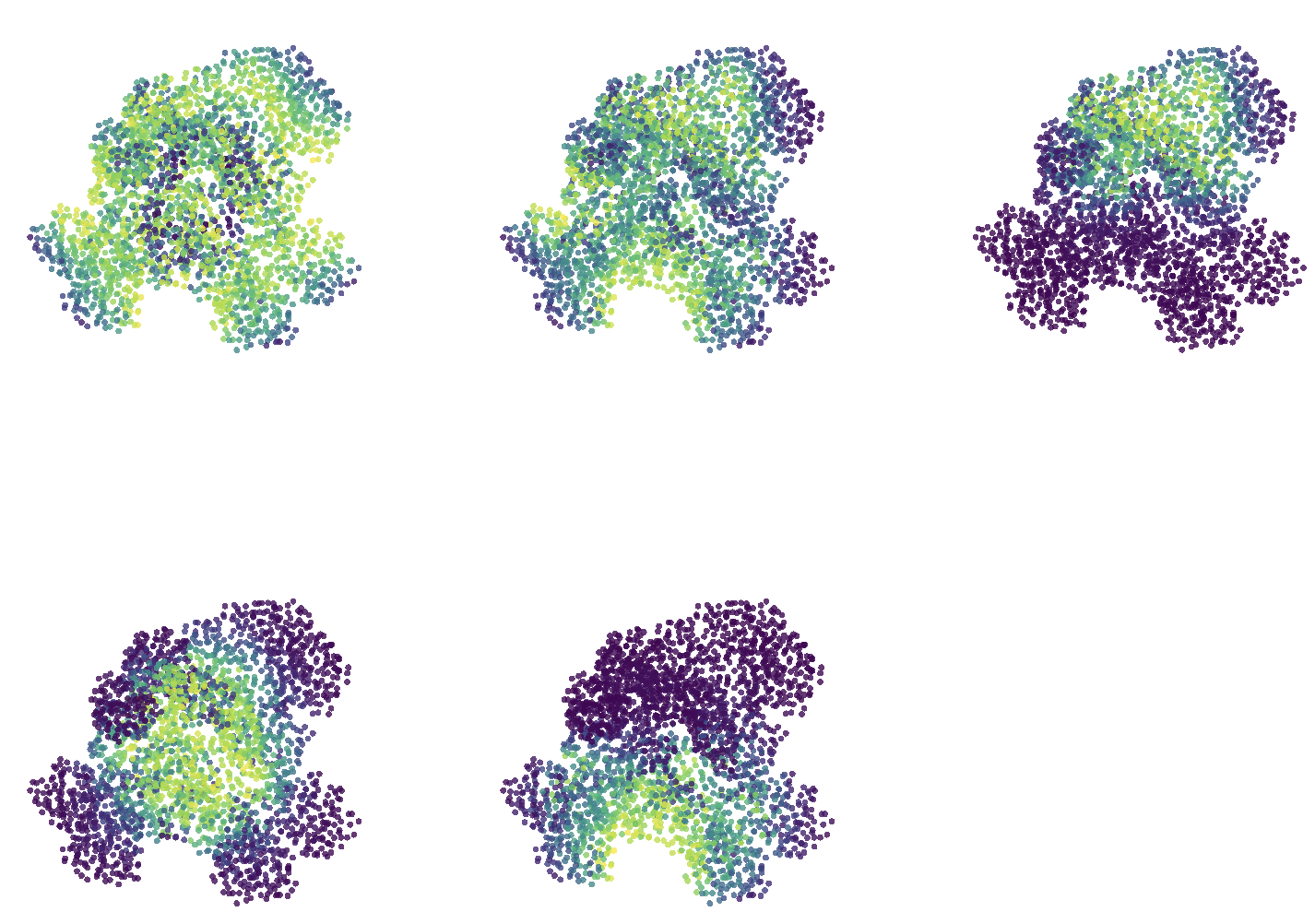}
		}
		\includegraphics[width=\linewidth]{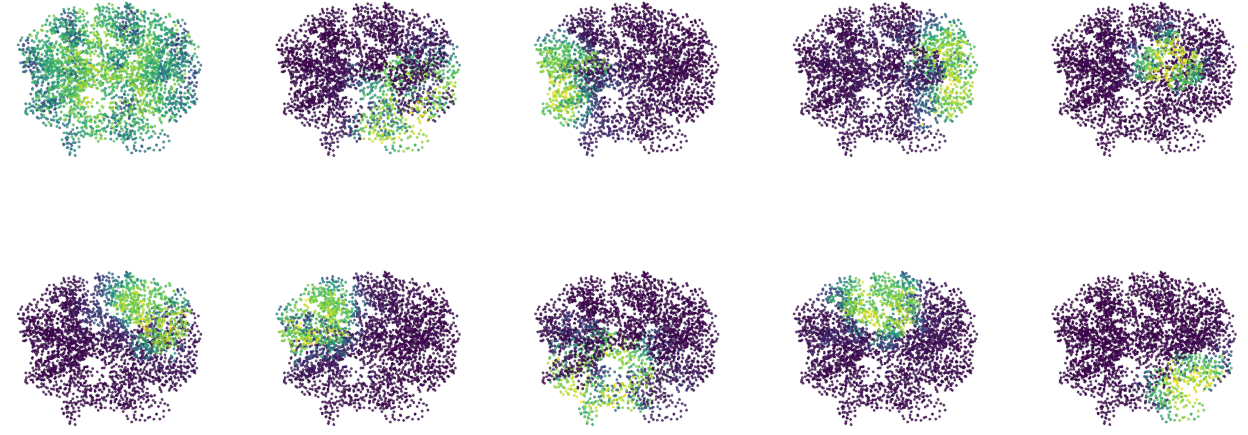}
		\caption{\figuretitle{
				\TOPF{} heatmaps for three proteins.}
			\emph{Top left:} \textsc{nalcn} channelosome \citep{Kschonsak:2022}
			\emph{Top right:} Mutated Cys123 of E.\ coli \citep{hidber2007participation}, with convex hull added during computation, only $2$-dimensional homology features
			\emph{Bottom:} GroEL of E.\ coli \citep{chaudhry2004exploring} (Selected features).
		}
		\label{fig:OldProtein}
	\end{center}
	\vskip -0.2in
\end{figure}

\paragraph{Performance with decreasing sampling density}
In \Cref{fig:Downsampling}, we analyse the robustness of \TOPF{} with respect to a decrease in sampling density.
We have selected \texttt{2Spheres2Circles} as the dataset with the highest original sampling density for this experiment.
It shows that \TOPF{} can still produce meaningful results even for significantly downsampled point clouds.
We note that the considered point clouds are in particular very sparse when compared to the point clouds usually considered by neural 3D point cloud architectures.
 We interpret the results as basically repeating the good performance of \TOPF{}, except for the datasets \texttt{EllipsesinEllipses} and \texttt{spaceship}, where the performance drops rapidly (at least in a logarithmic plot). Our interpretation of this is that both datasets consist of submanifolds with different sampling density, which intersect/are contained in one another. Thus, the point density is a key to distinguishing the classes, and a random change in these densities due to downsampling can make the encoded structure hard to detect or even causes them to vanish in a sense of topology.

\paragraph{Performance under heterogeneous sampling}
In \Cref{fig:Skewedsampling}, we analyse the robustness of \TOPF{} with respect to heterogeneous sampling.
The results show that TOPF performs well even in the presence of inhomogeneous sampling, and the degrading of performance can largely be attributed to a decrease in minimum sampling rate, rather than the inhomogeneous nature.

From a theoretical point of view, this should not be surprising: The birth of the topological feature will now depend on the minimum sampling density, while the death time still depends on the size of the feature. The heuristic employed by TOPF will then choose a maximal triangulation radius which lies between these two radii. In the dense part, the triangulation will be denser then in the sparse parts. However, due to the the character of the alpha filtration, the weighting of the simplices discussed in \Cref{app:WeightedSCs} and the thresholding this does not cause an issue in the features generated by TOPF.
\begin{figure}[tb!]
	\vskip -0.1in
	\begin{center}
		\includegraphics[width=\linewidth]{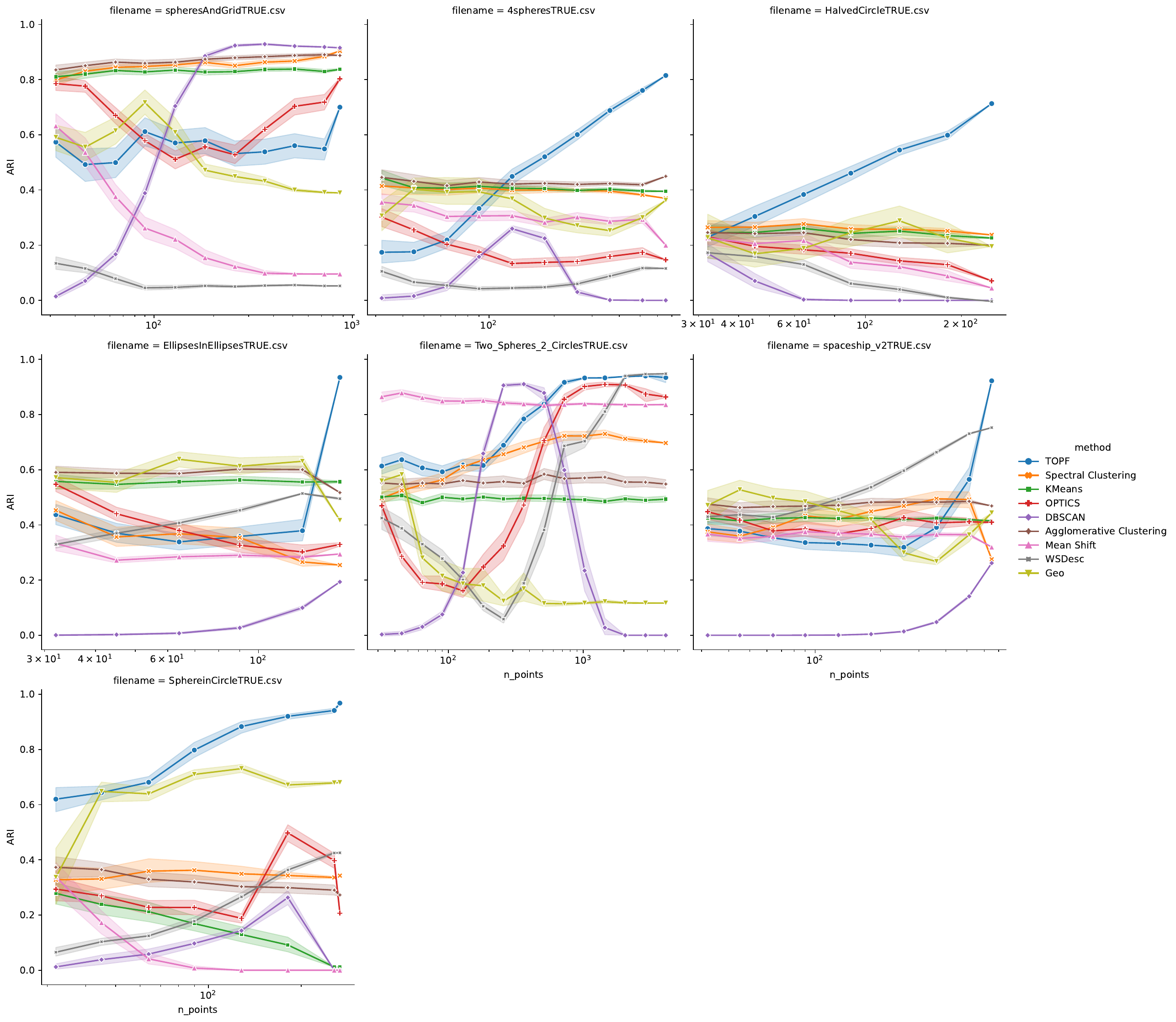}
		\caption{\figuretitle{
				Performance of \TOPF{} while decreasing sampling density.}
				The $x$-axis is scaled logarithmically and \TOPF{} achieves an adjusted rand index (\ARI{}, random algorithms achieve an \ARI{} of over 0.9 for $\sim 700$ samples, down from $4600$ original points on the \texttt{2Spheres2CirclesDataset}. The smallest considered sample size was $16$ points.
				The points were sampled at random.
				We ran each experiment $100$ times and report the standard deviation.
		}
		\label{fig:Downsampling}
	\end{center}
	\vskip -0.2in
\end{figure}
\begin{figure}[tb!]
	\vskip -0.1in
	\begin{center}
		\includegraphics[width=\linewidth]{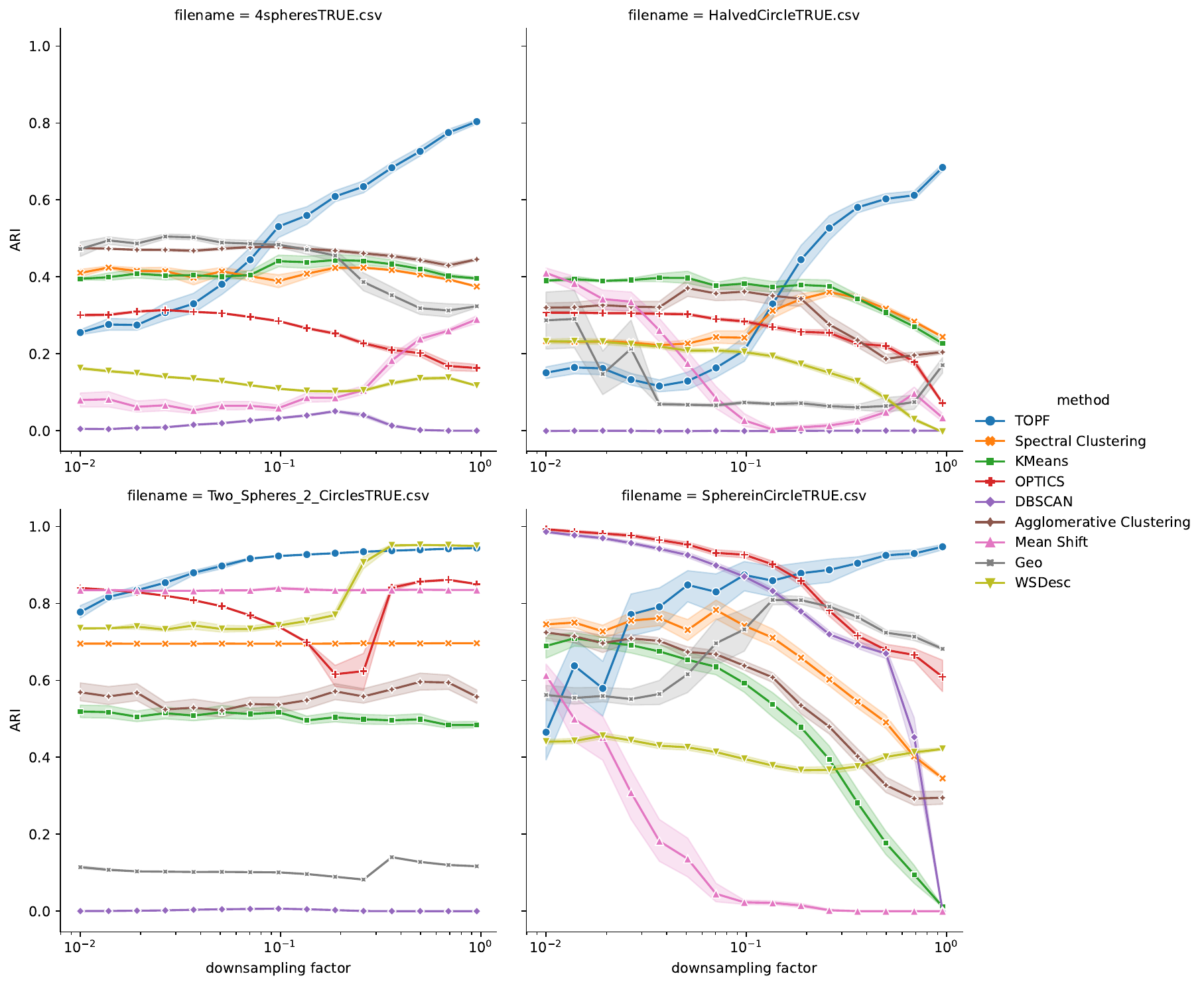}
		\includegraphics[width=0.3\linewidth]{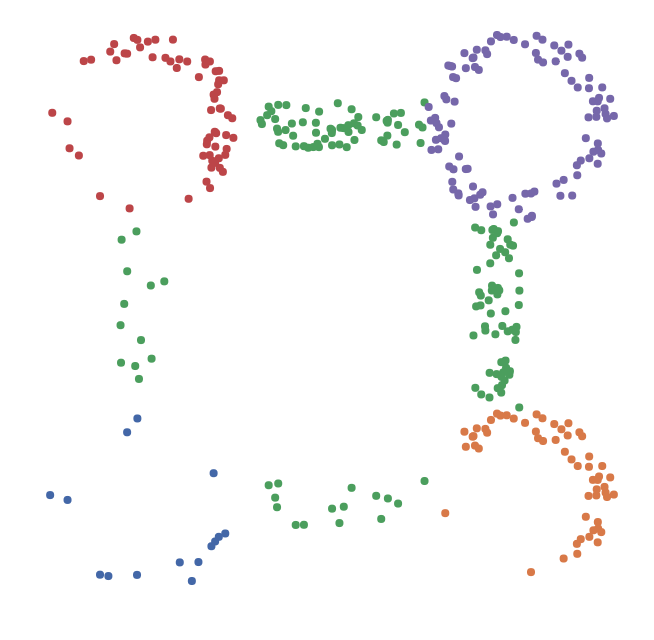}
		\includegraphics[width=0.69\linewidth]{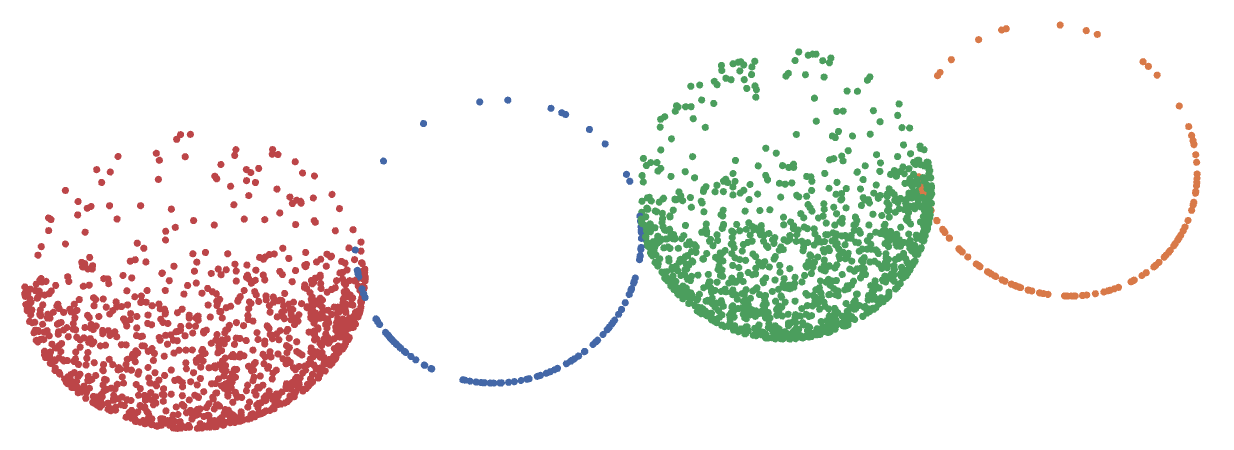}
		\caption{\figuretitle{
				Performance of \TOPF{} under heterogeneous sampling.}
\emph{Top: }We divide four of the point clouds with a symmetry into two halves. We downsample only one of the halves, creating sampling irregularities within the individual topological features. We repeat the experiment $100$ times and report the achieved \ARI{}.
			Note that the lowest value on the $x$-axis means downsampling by a factor of $100$, meaning that the smaller point clouds will then almost only consist of one half.
			We compare the performance of \TOPF{} with the other baselines, noting that \TOPF{} outperforms them for reasonable heterogeneities.
			\emph{Bottom Left:} \texttt{4spheres} with downsampling factor $0.2$ with true labels.
			\emph{Bottom Right:} \texttt{2spheres2circles} with downsampling factor $0.1$ with true labels.
		}
		\label{fig:Skewedsampling}
	\end{center}
	\vskip -0.2in
\end{figure}

\paragraph{Details on experiments with WSDesc}
WSDesc (Weakly Supervised 3D Local Descriptor Learning for Point Cloud Registration), \cite{Li2022WSDesc}, uses voxel-based representations of point clouds to extract robust 3D local descriptors for point cloud registration.
As showcased in the original paper, WSDesc showcases a good generalisation for point clouds outside of the training set.
We used WSDesc pretrained on the 3DMatch data.
WSDesc was used with $512$ keypoints, and tested with a varying of keypoints up to $\sim 5000$.
We thus tested WSDesc with the maximum number of keypoints possible per input point cloud.
\paragraph{Details on experiments with DGCNN}
We use the implementation of An Tao, \url{https://github.com/antao97/dgcnn.pytorch} of Dynamic Graph \textsc{cnn}s \cite{wang2019dynamic}.
We use the data pretrained on the ShapeNetPart segmentation dataset, where we use the features of the second last layer.
This showed the best performance.
We determine the optimal hyperparameters ($k$, object class, clustering method, layer to use) for every dataset of \TCBS{} individually and only present the results on the best parameter set.
The spectral clustering always outperformed $k$-means.
The authors of the original paper state that \textsc{DGCNN} performs very well on datasets with above $500$ points, and still is robust to downsampling beyond this point.
Thus we always input the maximum available number of points to \textsc{DGCNN}.
As described in \cite{wasserman2018topological}, we scale the input point cloud to fit into a unit sphere.
We pad the 2d data sets with zeroes in the third dimension.
\TOPF{} still outperforms \textsc{DGCNN} on the \textsc{TCBS}.
We believe that this is due to the heterogeneous sampling, small sampling set size and most importantly the lack of training data for the neural network.

\section{How to pick the most relevant topological features}
\label{app:featureselection}
\paragraph{Simplified heuristic}
The persistent homology $P_k$ module in dimension $k$ is given to us as a list of pairs of birth and death times $(b_i^k,d_i^k)$.
We can assume these pairs are ordered in non-increasing order of the durations $l_i^k=d_i^k-b_i^k$.
This list is typically very long and consists to a large part of noisy homological features which vanish right after they appear.
In contrast, we are interested in connected components, loops, cavities, etc.\ that \emph{persist} over a long time, indicating that they are important for the shape of the point cloud.
Distinguishing between the relevant and the irrelevant features is in general difficult and may depend on additional insights on the domain of application.
In order to provide a heuristic which does not depend on any a-priori assumptions on the number of relevant features we pick the smallest quotient $q_i^k\coloneq l_{i+1}^k/l_i^k>0$ as the point of cut-off $\smash{N_k\coloneq\argmin_{i} q_i^k}$.
The only underlying assumption of this approach is that the band of \enquote{relevant} features is separated from the \enquote{noisy} homological features by a drop in persistence.
\paragraph{Advanced Heuristic}
However, certain applications have a single very prominent feature, followed by a range of still relevant features with significantly smaller life times, that are then followed by the noisy features after another drop-off.
This then could potentially lead the heuristic to find the wrong drop-off.
We propose to mitigate this issue by introducing a hyperparameter $\beta\in\R_{>0}$.
We then define the $i$-th importance-drop-off quotient $q_i^k$ by
\[
q_i^k\coloneq \nicefrac{l_{i+1}^k}{l_i^k}\left(1+\nicefrac{\beta}{i}\right).
\]
The basic idea is now to consider the most significant $N_k$ homology classes in dimension $k$ when setting $N_k$ to be
\[
N_k\coloneq\argmin_{i} q_i^k.
\]
Increasing $\beta$ leads the heuristic to prefer selections with more features than with fewer features.
Empirically, we still found $\beta=0$ to work well in a broad range of application scenarios and used it throughout all experiments.
There are only a few cases where domain-specific knowledge could suggest picking a larger $\beta$.

To catch edge cases with multiple steep drops or a continuous transition between real features and noise, we introduce two more checks:
We allow a minimal $q_i^k$ of $\mathtt{min\_ rel\_ quot}=0.1$ and a maximal quotient $\nicefrac{q_1^h}{q_i^k}$ of $\mathtt{max\_total\_quot}=10$ between any homology dimensions.
Because features in $0$-dimensional homology are often more noisy than features in higher dimensions, we add a minimum zero-dimensional homology ratio of $\mathtt{min\_0\_ratio}=5$, i.e.\ every chosen $0$-dimensional feature needs to be at least $\mathtt{min\_0\_ratio}$ more persistent then the minimum persistence of the higher-dimensional features.
Because these hyperparameters only deal with the edge cases of feature selection, \TOPF{} is not very sensitive to them.
For all our experiments, we used the above hyperparameters.
We advise to change them only in cases where one has in-depth domain knowledge about the nature of relevant topological features.

\paragraph{Fixed number of topological features}
Alternatively, it is possible to specify a fixed desired number $N_d$ of topological features per dimension $d$.
\TOPF{} then automatically returns the $N_d$ most relevant features in dimension $d$.
For practical purposes, we can then weigh these features by a function $w(-)$ in the life time $d_i-b_i$ or the life quotient $d_i/b_i$ of the features.
Possible picks for  $w$ include an exponential function $w(x)=e^x$, a quadratic function $w(x)=x^2$, a linear function $w(x)=x$ or a scaled sigmoid function.
Our intuition suggests that when picking functions like a linear function, the many short-lived features will be given too much weight in comparison to the few more relevant features.
Selection the best weight function is an interesting open problem for future work.
\section{Simplicial Weights}

\label{app:WeightedSCs}
\begin{figure*}[tb!]
	\vskip 0.2in
	\begin{center}
		\centerline{\includegraphics[width=\linewidth]{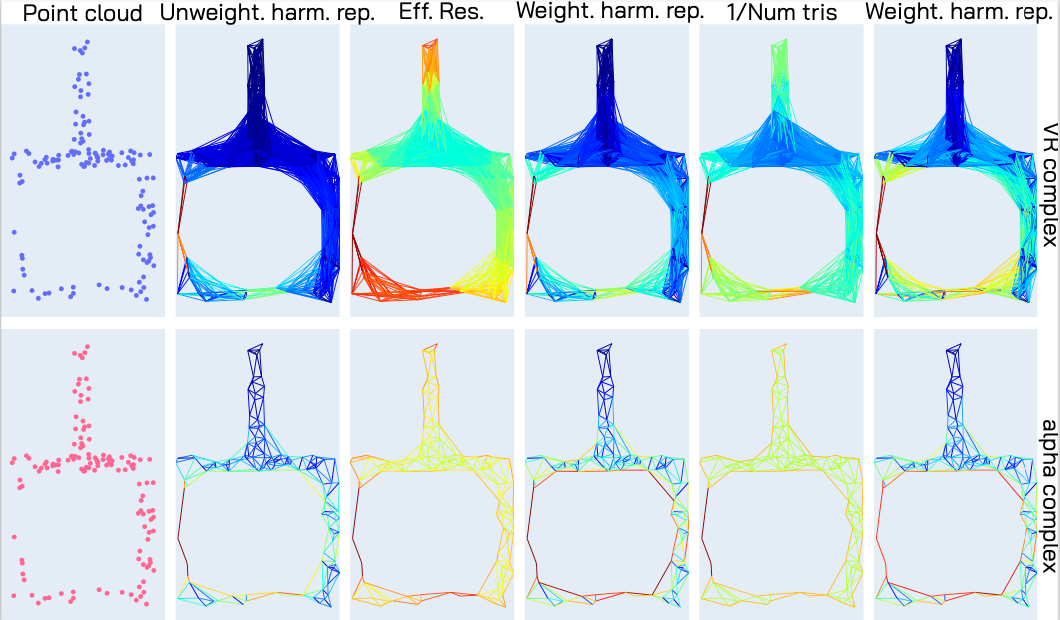}}
		\caption{\figuretitle{Effect of weighing a simplicial complex on harmonic representatives.}
			\emph{Top:} \VR~complex.
			\emph{Bottom:} $\alpha$-complex
			\emph{Left:} The base point cloud with different densities.
			\emph{2\textsuperscript{nd} Left:} Unweighted harmonic homology representative of the large loop.
			\emph{3\textsuperscript{rd} Right:} Effective resistance of the $1$-simplices.
			\emph{3\textsuperscript{rd} Right:} Harmonic homology representative of the complex weighted by effective resistance.
			\emph{2\textsuperscript{nd} Right:} Inverse of number of incident triangles (\Cref{def:effresweights}).
			\emph{Right:} Harmonic homology representative of the complex weighted by number of incident triangles.
			Up to a small threshold, the standard harmonic representative in the \VR~complex is almost exclusively supported in the low-density regions of the simplicial complex.
			This leads to poor and unpredictable classification performance in downstream tasks.
			In contrast, the harmonic homology representative of the weighted \VR~complex has a more homogenous support along the loop, while still being able to discriminate the edges not contributing to the loop.
			The $\alpha$-complex suffers less from this phenomenon (at least in dimension $2$), and hence reweighing is not necessarily required.
		}
		\label{fig:EffResExample}
	\end{center}
	\vskip -0.2in
\end{figure*}
In an ideal world, the harmonic eigenvectors in dimension $k$ would be vectors assigning $\pm 1$ to all $k$-simplices contributing to $k$-dimensional homological feature, a $0$ to all $k$-simplices not contributing or orthogonal to the feature, and a value in $(-1,1)$ for all simplices based on the alignment of the simplex with the boundary of the void.
However, this is not the case:
In dimension $1$, we can for example imagine a total flow of $1$ circling around the hole.
This flow is then split up between all parallel edges which means \emph{two} things:
\textbf{I} Edges where the loop has a \emph{larger diameter} have \emph{smaller harmonic values} than edges in thin areas and
\textbf{II} in \VR~complexes, which are the most frequently used simplicial complexes in \TDA{}, edges in areas with a \emph{high point density} have \emph{smaller harmonic values} than edges in low-density areas.
Point \textbf{II} is another advantage of $\alpha$-complexes: The expected number of simplices per point does not scale with the point density in the same way as it does in the \VR~complex, because only the simplices of the Delaunay triangulation can appear in the complex.

We address this problem by weighing the $k$-simplices of the simplicial complex.
The idea behind this is to weigh the simplicial complex in such a way that it increases and decreases the harmonic values of some simplices in an effort to make the harmonic eigenvectors more homogeneous.
For weights $w\in\R^{\SC_k}$, $W=\diag (w)$, the symmetric weighted Hodge Laplacian \citep{schaub2020random} takes the form of
\[
L_k^w=W^{1/2}\bound_{k-1}\bound_{k-1}^\top W^{1/2}+W^{-1/2}\bound_{k}\bound_{k}^\top W^{-1/2}.
\]
Because we want the homology representative to lie in the weighted gradient space, we have to scale its entries with the weight and set $e^i_{k,w}\coloneq W^{-1/2}e^i_{k}$.
With this, we have that
\[
\bound_{k-1}^\top W^{1/2}e^i_{k,w}=\bound_{k-1}^\top W^{1/2}W^{-1/2}e^i_{k}=\bound_{k-1}^\top e^i_{k}=0
\]
We propose two options to weigh the simplicial complex. The first option is to weigh a $k$-simplex by the square of the number of $k+1$-simplices the simplex is contained in:
\[
w_\Delta(\sigma_k) = 1/(|\{\sigma_{k+1}\in\SC_{k+1}^t :\sigma_k\subset\sigma_{k+1} \}|+1)^2
\]
where the $+1$ is to enforce good behaviour at simplices that are not contained in any higher-order simplices.
One of the advantages of the $\alpha$-complex is that we don't have large concentrations of simplices in well-connected areas.
The proposed weighting $w_\Delta$ is computationally straightforward, as it can be obtained as the column sums of the absolute value of the boundary matrix $|\bound_k|$.
The weights also deal with the previously mentioned problem \textbf{II}: As the homology representative is scaled inversely to the weight vector $w$, the simplices in high-density regions will be assigned a low weight and thus their weighted homology representative will have a larger entry.
By the projection to the orthogonal complement of the curl space, this large entry is then diffused among the high-density region of the \tSC{} with many simplices, whereas the lower entries of the simplices in low-density regions are only diffused among fewer adjacent simplices.

\cite{paik2023circular} consider the dual problem of weighting the simplicial complex such that a harmonic \emph{co}homology representative becomes sampling-density independent. They show that the dual to our approach, i.e. by dividing by the sum of lower-incident simplices (i.e.\ nodes) of the two vertices of an edge produces an harmonic representative stable under heterogeneous density.

However, the first weight is not able to incorporate the number of parallel simplices into the weighting.
This is why we propose a second simplicial weight function based on generalised effective resistance.
\begin{definition}[Effective Hodge resistance weights]
	\label{def:effresweights}
	For a simplicial complex $\SC$ with boundary matrices $(\bound_k)$, we define the effective Hodge resistance weights $w_R$ on $k$-simplices to be:
	\[
	w_R\coloneq \diag \left( \bound_{k-1}^+\bound_{k-1}\right)^2
	\]
	where $\diag(-)$ denotes the vector of diagonal entries and $(-)^+$ denotes taking the Moore--Penrose inverse.
\end{definition}
Intuitively for $k=1$, we can assume that every edge has a resistance of $1$ and then the effective resistance coincides with the notion from Physics.
Thus simplices with many parallel simplices are assigned a small effective resistance, whereas simplices with few parallel simplices are assigned an effective resistance close to $1$.
However, computing the Moore--Penrose inverse is computationally expensive and only feasible for small simplicial complexes.

In \Cref{fig:EffResExample}, we show that the weights $w_\Delta$ are a good approximation of the effective resistance in terms of the resulting harmonic representative.
The standard form of \TOPF{} used in all experiments uses $w_\Delta$-weights.
\section{Limitations}
\label{app:limitations}
\paragraph{Topological features are not everywhere}
The proposed topological point features take relevant persistent homology generators and turn these into point-level features.
As such, applying \TOPF{} only produces meaningful results on point clouds that have a topological structure.
On these point clouds, \TOPF{} can extract structural information unobtainable by non-topological methods.
Although \TDA{} has been successful in a wide range of applications, a large number of data sets does not possess a meaningful topological structure.
Applying \TOPF{} in these cases will produce no additional information.
Other data sets require pre-processing before containing topological features.
In \Cref{fig:QualitativeExperiments} \emph{left}, the $2d$ topological features characterising protein pockets of Cys123 only appear after artificially adding points sampled on the convex hull of the point cloud (Cf.\ \cite{oda2024novel}).
\paragraph{Computing persistent homology can be computationally expensive}

As \TOPF{} relies on the computation of persistent homology including homology generators, its runtime increases on very large point clouds.
This is especially true when using \VR{} instead of $\alpha$-filtrations, which become computationally infeasible for higher-dimensional point clouds.
Persistent homology computations for dimensions above $2$ are only feasible for very small point clouds.
Because virtually all discovered relevant homological features in applications appear in dimension $0$, $1$, or $2$, this does not present a large problem.
Despite these computational challenges, subsampling, either randomly or using landmarks, usually preserves relevant topological features and thus extends the applicability of \TDA{} in general and \TOPF{} even to very large point clouds.

\paragraph{Automatic feature selection is difficult without domain knowledge}
While the proposed heuristics works well across a variety of domains and application scenarios, only domain- and problem-specific knowledge makes truthful feature selection feasible.

\paragraph{Experimental Evaluation}
There are no benchmark sets for topological point features in the literature, which makes benchmarking \TOPF{} not straightforward.
On the level of clustering, we introduced the topological clustering benchmark suite to make quantitative comparisons of \TOPF{} possible, and benchmarked \TOPF{} on some of the point clouds of \cite{Grande:2023}.
On both the level of point features and real-world data sets, it is however hard to establish what a \emph{ground truth} of topological features would mean.
Instead we chose to qualitatively report the results of \TOPF{} on proteins and real-world data, see \Cref{fig:QualitativeExperiments}.

\end{document}